\documentclass[leqno,11pt]{article}
\usepackage[margin=1in]{geometry}

\usepackage[latin,english]{babel}
\usepackage[utf8]{inputenc}
\usepackage[T1]{fontenc}
\usepackage{eufrak, color, mathrsfs,dsfont}
\usepackage{times}
\usepackage{geometry}
\usepackage{amsmath}
\usepackage{amssymb}
\usepackage{amsthm}
\usepackage{tikz}
\usepackage{mathtools}
\usepackage{mathabx}
\usepackage{cases}
\usepackage{braket}
\usepackage[toc,page]{appendix}
\usepackage{pdfsync}
\usepackage{hyperref}
\usepackage{psfrag}
\mathtoolsset{showonlyrefs}
\usepackage[mathscr]{euscript}
\usepackage{math}

\numberwithin{equation}{section}

\theoremstyle{plain}
\newtheorem{thm}{Theorem}[section]

\newtheorem{lem}[thm]{Lemma}
\newtheorem{prop}[thm]{Proposition}

\theoremstyle{definition}
\newtheorem{defn}[thm]{Definition}

\theoremstyle{remark}
\newtheorem{rem}[thm]{Remark}

\DeclareMathOperator{\diag}{diag}

\definecolor{blush}{rgb}{0.87, 0.36, 0.51}

\newcommand{\OpM}{{\mathcal O}{\mathcal P}{\mathcal M}}

\renewcommand{\bar}{\overline}

\newcommand{\even}{{\rm even}}
\newcommand{\odd}{{\rm odd}}

\newcommand{\pa}{\partial}
\newcommand{\vs}{\varsigma}
\newcommand{\vphi}{\varphi}

\newcommand{\Lip}{\mathrm{Lip}}

\newcommand{\balpha}{{\boldsymbol{\alpha}}}
\newcommand{\bbeta}{{\boldsymbol{\beta}}}

\newcommand{\uu}{{\bold{u}}}

\title{\bf Large amplitude quasi-periodic traveling waves in two dimensional forced rotating fluids}

\author{Roberta Bianchini, Luca Franzoi,  Riccardo Montalto, Shulamit Terracina}

\begin{document}

\date{}

\maketitle

\noindent
{\bf Abstract.} 
We establish the existence of quasi-periodic traveling wave solutions for the $\beta$-plane equation on $\mathbb{T}^2$ with a large quasi-periodic traveling wave external force. These solutions exhibit large sizes, which depend on the frequency of oscillations of the external force. Due to the presence of small divisors, the proof relies on a nonlinear Nash-Moser scheme tailored to construct nonlinear waves of large size. To our knowledge, this is the first instance of constructing quasi-periodic solutions for a quasilinear PDE in dimensions greater than one, with a 1-smoothing dispersion relation that is highly degenerate - indicating an infinite-dimensional kernel for the linear principal operator.
This degeneracy challenge is overcome by preserving the traveling-wave structure, the conservation of momentum and by implementing normal form methods for the linearized system with sublinear dispersion relation in higher space dimension.

\smallskip



\smallskip 

\tableofcontents

\section{Introduction}


Incompressible rotating fluids in three dimensions are described by the Euler-Coriolis equations, which, on the three dimensional torus $\T^3 = (\R/2\pi\Z)^3$, read as
\begin{equation}\label{euler.coriolis}
	\begin{aligned}
		\pa_{t} \uu + (\uu \cdot \nabla) \uu + \ff\, \overrightarrow{e_3} \land \uu + \nabla P &= \bF \,, \\
		\nabla \cdot \uu&=0\,,
	\end{aligned}
\end{equation}
where $\uu(t, x) : \R \times \T^3 \rightarrow \R^3$ is the divergence-free velocity field, $P(t, x):  \R \times \T^3 \rightarrow \R$ is the scalar pressure, $\ff=\ff(x_2)$ is a scalar function, $\overrightarrow{e_3}=(0,0,1)^\top$ is the rotation vector, and $\bF(t,x):\R\times \T^3 \to \R^3$ is an external force. The term $\ff\,\overrightarrow{e_3}\land \uu$ is usually referred as the Coriolis force. It is customary in geophysical fluid dynamics to simplify the model introducing some approximations. First, noting that $\overrightarrow{e_3}\land \uu = (-u_2,u_1,0)^\top := (\uu_h^\perp, 0)^\top$, according to the notation $\uu=(u_1,u_2,u_3)=: (\uu_h,u_3)$, we assume a trivial dependence with respect to the third direction, namely $\pa_{x_3} \uu =0$ and $\pa_{x_3}\bF  = 0$. As part of the standard elliptic constraint $\nabla\cdot \big((\uu \cdot \nabla )\uu \big) + \Delta P =0$ on the pressure, we also assume $\pa_{x_3} P =0$.
This way, setting
\begin{equation}
	\nabla=(\pa_{x_1},\pa_{x_2},\pa_{x_3})^\top =: (\nabla_{h},\pa_{x_3})^\top\,, \quad \bF= (\tF_{1},\tF_{2},\tF_{3})^\top=: (\bF_{h},\tF_{3})^\top\,,
\end{equation}
we are led to the reduced two dimensional problem
\begin{equation}\label{euler.coriolis.2D}
	\begin{aligned}
		\pa_{t} \uu_h + (\uu_h \cdot \nabla_h) \uu_h + \ff\, \uu_h^\perp + \nabla_{h} P &= \bF_{h} \,, \\
		\nabla_{h} \cdot \uu_h&=0\,.
	\end{aligned}
\end{equation}
A solution to the original three dimensional problem \eqref{euler.coriolis} is then recovered by solving the linear transport equation $	\pa_{t} u_3 + \uu_h \cdot\nabla_{h} u_3 = \tF_{3}$.

Next, we introduce the \emph{$\beta$-plane approximation} by choosing
$$\ff(x_2)=\beta x_2,$$
where $\beta \in \R\setminus\{0\}$ is a non-trivial constant. Note that we choose $\ff$ to be a \emph{linear} function, and the vertical dynamics are considered trivial, i.e. $\pa_{x_3}\uu = 0$ as above. For further details, we refer to works such as McWilliams \cite{McWill}, Pedlosky \cite{Pedlosky}, Chemin, Desjardins, Gallagher \& Grenier \cite{chemin2006}, Gallagher \& Saint-Raymond \cite{laure}, and Pusateri \& Widmayer \cite{PusaWid}. The constant $\beta \in \R\setminus\{0\}$ represents the speed of rotation of the frame system around the rotation vector $ \overrightarrow{e_3}$. Throughout this work, we treat $\beta$ as a fixed constant.

\medskip
Introducing the scalar vorticity (i.e. the third component of the vorticity field $\nabla \land \uu$) 
\begin{align}\label{eq:vorticity}
v:= {\rm curl} \, \uu_h := \pa_{x_1} u_2 - \pa_{x_2} u_1,
\end{align}
and applying the ${\rm curl}$ operator to \eqref{euler.coriolis.2D}, yields the scalar equation 
%
\begin{equation}\label{rotating.2D}
	\partial_t v + \uu_h \cdot \nabla_{h} v + \beta\, \tL v  = {\rm curl}\, \bF_{h} \,,\quad \bu_h=\nabla_{h}^\perp (-\Delta)^{-1} v \,, 
\end{equation}
which is the so-called \emph{$\beta$-plane equation}, where we have, under the incompressibility condition $\nabla_{h}\cdot \uu_h = 0$,
\begin{equation}
	{\rm curl}( \ff \,\uu_h^\perp )= \pa_{x_1}(\beta x_2 u_1) - \pa_{x_2}(-\beta x_2 u_2) =  \beta u_2 = \beta \, \tL v\,,
\end{equation}
and the operator
\begin{equation}\label{operator.tL}
	\tL := (-\Delta)^{-1} \pa_{x_1}\,,
\end{equation}

namely the second component of $\nabla_{h}^\perp (-\Delta)^{-1}$ with $\nabla_{h}^\perp = (-\pa_{x_2},\pa_{x_1})^\top$,
acts on any function $h(x):\T^2\to\R$ with zero average as the Fourier multiplier
\begin{equation}\label{lambda.m.def}
	\tL h(x) 
	:= \sum_{\xi\in\Z^2\setminus\{0\}} \im\,\tL(\xi) \whh(\xi) e^{\im\,\xi\cdot x}\,, \quad \tL (\xi) :=  \frac{\xi_{1}}{|\xi|^2} \quad \forall\,\xi=(\xi_{1},\xi_{2})^\top\in \Z^2\setminus\{0\}\,.
\end{equation} 

\begin{rem}
	 The sign in front of $\beta$ in equation \eqref{rotating.2D} may be different compared to previous works in literature (for instance, see \cite{ElgWid, LWZZ, PusaWid}). This discrepancy depends on how the operator $\tL$ in \eqref{operator.tL}-\eqref{lambda.m.def} is defined, which may involve the Riesz transform or a change of sign in front of the Laplacian. We remark that this is purely a sign convention and does not modify the validity of the results at all.
\end{rem}
 
Approximating models for rotating fluids in two and three dimensions are commonly used in oceanogra\-phy and geophysical fluid dynamics. An interesting regime that is frequently studied is
when the fluid is rapidly rotating, namely when $|\beta| \gg 1$ in the $\beta$-plane approximation, even in the viscous case. In a series of works, Babin, Mahalov \& Nicolaenko \cite{BaMaNi97,BaMaNi99,BaMaNi2000} proved regularity and integrability properties of solutions in 3D resonant tori.  Dalibard \& Saint-Raymond \cite{DSR2010} devoted their analysis of the long-time behavior of solutions The singular limit in the vanishing Rossby number $\text{Ro} \sim \beta^{-1}$ was investigated by Charve \cite{charve2004} and Dutrifoy \cite{Dutrifoy}. The long-time behavior of solutions with high-speed rotation was studied by Koh, Lee \& Takada \cite{KohLeeTak} and Takada \cite{takada} by means of Strichartz estimates, by Dalibard  \cite{dalibard2009} under the effect of random forcing and by Angullo-Castillo \& Ferreira \cite{AngFer} in Besov spaces. The relation between the speed of rotation and the lifespan of solutions was investigated by Ghoul, Ibrahim, Lin \& Titi \cite{GILT} for the three dimensional primitive equations.

Independently of the speed of rotation, the global well-posedness of solutions was studied for the $\beta$-plane equation by Elgindi \& Widmayer \cite{ElgWid} and Pusateri \& Widmayer \cite{PusaWid}. The global well-posedness for the Euler-Coriolis equations was studied by  Guo, Huang, Pausader \& Widmayer \cite{GHPW} and Guo, Pausader \& Widmayer \cite{GuoPauWid} with axisymmetric initial data, and, very recently, by Ren \& Tian \cite{RenTian} with general non-axisymmetric initial data. For the $\beta$-plane equation, we mention also the works of Wei, Zhang \& Zhu \cite{WeiZZ} on the linear inviscid damping around shear flows, and of Wang, Zhang \& Zhu \cite{WangZhangZhu} on the existence of non-shear solutions close to the Couette flow depending  on the size of the rotational speed $\beta$.

\subsection{Main result}

Led by the physical model discussed above, in this paper we focus our attention on the effect of a \emph{large} forcing term of the form
\begin{equation}
	 {\rm curl}\, \bF_{h} (t,x) =  \lambda^{\alpha} f(\lambda \,\omega t,x)
\end{equation}
where the parameter $\lambda\gg 1$ is large, $\alpha>0$ is a positive exponent and, given $\nu\in\N$, $\nu\geq 2$, the function $f(\vf,x)$  is  \emph{quasi-periodic in time} with  frequency vector $\lambda\,\omega\in\R^\nu\setminus\{0\}$, evaluated at the linear flow $t\mapsto f(\vf,x)|_{\vf = \lambda \, \omega t} $. We shall assume that $f \in {\mathcal C}^\infty(\T^\nu \times \T^2)$, $f\nequiv 0$, has zero average in $x$, that is 
\begin{equation}\label{ipotesi forcing}
 \int_{\T^2} f(\vf, x)\, \wrt x = 0\,, \quad \forall \vf \in \T^\nu\,. 
\end{equation}  
The parameter $\alpha >0$ will be fixed later in order to construct solutions of large amplitude. 

Therefore, with a minor change in the notation,  from now on we work with the following equation, on  the two dimensional torus $\T^2$, 
\begin{equation}\label{beta.waves.large.eq}
	\pa_{t} v+ \beta \, \tL v + u \cdot \nabla v = \lambda^{\alpha} f(\lambda \,\omega t,x) \,, \quad t\in\R\,, \ x \in 
	\T^2\,,
\end{equation}
where
the two dimensional velocity field $u:\R\times\T^2 \to \R^2$ is determined from the scalar vorticity $v : \R\times \T^2 \to \R$ by the classical Biot-Savart law 
\begin{equation}\label{biot-savart}
	u(t,x) := \fB[v (t,x)] := \nabla^\perp (-\Delta)^{-1} v(t,x)  = \sum_{\xi\in\Z^2\setminus\{0\}} \frac{\im(-\xi_{2},\xi_{1})^\top}{|\xi|^2} \whv (t,\xi) e^{\im\, \xi\cdot x}\,.
\end{equation}


With equation \eqref{beta.waves.large.eq}, we aim to study how the presence of a forcing term, \emph{of substantial magnitude} and \emph{highly oscillating in time}, affects the dynamics in the $\beta$-plane approximation \emph{without relying on a perturbative regime}. Additionally, we seek to determine whether the quasi-periodic structure of the external force produces a time quasi-periodic solution of \eqref{beta.waves.large.eq}, \emph{with substantial size} and \emph{rapid} oscillations, specifically with a frequency vector of oscillation $\lambda\,\omega \in \R^\nu\setminus\{0\}$ of significant magnitude. The interaction between the size $\lambda$ of the vector $\lambda\,\omega$ and the size $\lambda^\alpha$ of the forcing term will be discussed later.

In particular, we search for \emph{quasi-periodic traveling wave} solutions, according to the following definition.
\begin{defn}\label{def1.QPT}
	{\bf (Quasi-periodic traveling wave).}
	Let $\nu \in \N$. We say that a function $v:\R \times \T^2\to \R$ is  a \emph{quasi-periodic traveling wave} with irrational frequency vector $\omega= (\omega_{1},...,\omega_{\nu})\in\R^{\nu}\setminus\{0\}$, that is $\omega\cdot \ell \neq 0$ for any $\ell\in\Z^\nu\setminus\{0\}$, and  ``wave vectors'' $\bar\jmath_{1},...,\bar\jmath_{\nu}\in \Z^2$ if there exists a function $\breve{v}:\T^\nu \to\R$ such that
	\begin{equation}
		v(t,x) = \breve{v} (\omega_{1}t- \bar{\jmath}_1\cdot x,...,\omega_{\nu}t-\bar{\jmath}_{\nu}\cdot x) = \breve{v}(\omega t- \pi( x)) \,, \quad \forall (t,x) \in\R \times\T^2\,,
	\end{equation}
	where  $\pi :\R^2 \to \R^\nu$ is the linear map $x \mapsto \pi( x) := (\bar{\jmath}_{1}\cdot x,...,\bar{\jmath}_{\nu}\cdot x)$. We also denote by $\pi^\top$ the transpose of the map $\pi$. 
\end{defn}

We fix once and for all the $\nu$ vectors $\bar\jmath_{1},...,\bar\jmath_{\nu}\in \Z^2$ such that ${\rm dim}\,{\rm span}\{\bar{\jmath}_{1},...,\bar{\jmath}_{\nu}\} = 2$, in order to avoid trivial cases.
Therefore, assuming that the forcing term $f(\lambda\,\omega t, x)$ is a quasi-periodic traveling function according to Definition \ref{def1.QPT}, for most values of the parameter $\omega\in\R^\nu \setminus\{0\}$,
we search for quasi-periodic traveling wave solutions $v(\vf,x)|_{\vf = \lambda\,\omega t}$ with \emph{large} frequency vector $\lambda \,\omega\in\R^{\nu}$ (with a slight abuse of notation between $v(t,x)$ and $v(\lambda\,\omega t,x)$) of the equations \eqref{beta.waves.large.eq}-\eqref{biot-savart}, which we rewrite as
\begin{equation}\label{QP.Eq.to.solve}
	\lambda \,\omega \cdot \pa_{\vf} v + \beta \,\tL v + u \cdot \nabla v = \lambda^{\alpha} f(\lambda\, \omega t,x) \,, \quad u = \fB[v] =  \nabla^\perp (-\Delta)^{-1} v\,.
\end{equation}
We look for solutions $v \in H^s_0(\T^\nu \times \T^2)$, for $s \gg 1$ large enough, where  
$$
\begin{aligned}
H^s(\T^\nu \times \T^2) & := \Big\{ u(\vf, x) = \sum_{(\ell, j) \in \Z^\nu \times \Z^2 } \widehat u(\ell, j) e^{\im (\ell \cdot \vf + j \cdot x)} : \| u \|_s^2 := \sum_{(\ell, j) \in \Z^\nu \times \Z^2} \langle \ell, j \rangle^{2 s} |\widehat u(\ell, j)|^2 < + \infty \Big\}\,, \\
H^s_0(\T^\nu \times \T^2) & := \Big\{ u \in H^s(\T^\nu\times\T^2)  : \int_{\T^2} u(\vf, x) \, d x = 0 \Big\}\,, \quad \langle \ell, j \rangle := \max \{ 1, |\ell|, |j| \}\,.  
\end{aligned}
$$


Let $\Omega\subset \R^\nu$ be the reference compact set
\begin{equation}\label{anello}
	\Omega := \{\omega\in\R^\nu \,: \, 1 \leq |\omega|  \leq 2 \}\,.
\end{equation}
The main theorem of this paper is the following.

\begin{thm}\label{teo principale beta plane}
Let $\beta\in\R\setminus\{0\}$, $ \alpha \in (1,2)$ and $\nu \in \N$, $\nu\geq 2$, be fixed. Let $f \in {\mathcal C}^\infty(\T^\nu \times \T^2, \R)$ be a quasi-periodic traveling wave (see  Definition \ref{def1.QPT}) satisfying \eqref{ipotesi forcing}. There exists $\bar s = \bar s(\nu, \alpha) > 0$ such that, for any $s \geq \bar s(\nu, \alpha)$,
there exist $\lambda_0 = \lambda_0(f, s,  \nu, \alpha, \beta) \gg 1$ large enough and constants $C_1, C_2>0$, with
$C_i = C_i(f, s, \nu, \alpha, \beta) $ for  $i=1,2$, 
such that, for every $\lambda \geq \lambda_0$, the following holds. 
There exists a Borel set $\Omega_\lambda \subset \Omega$ 
of asymptotically full Lebesgue measure as $\lambda \to + \infty$,
that is $\lim_{\lambda \to + \infty} |\Omega \setminus \Omega_\lambda| = 0$, 
such that, for every $\omega \in \Omega_\lambda$, there exists a quasi-periodic traveling wave $v_\lambda(\,\cdot\, ; \omega) \in H^s_0(\T^\nu \times \T^2)$ that solves the equation \eqref{QP.Eq.to.solve}. Moreover 
\begin{equation}\label{stima.grande.solutione}
	 C_1 \lambda^{\alpha-1} \leq \inf_{\omega \in \Omega_\lambda} \| v_\lambda(\,\cdot\,; \omega) \|_s  \leq    \sup_{\omega \in \Omega_\lambda} \| v_\lambda(\,\cdot\,; \omega) \|_s \leq C_2 \lambda^{\alpha-1+\tc}\,,
\end{equation}
for some $\tc \in (0,\tc_{0})$ arbitrarily small, with $\tc_{0}= \tc_{0}(\alpha,\nu)$. Finally, if we assume that $f(\vf,x)$ that is even in the pair $(\vf,x)$, then the solution $v_{\lambda}(\omega t,x)$ is linearly stable.
\end{thm}

With Theorem \ref{teo principale beta plane}, we prove that the presence of a large external forcing term that is quasi-periodic trave\-ling and with large frequency of oscillations induces solutions to the forced $\beta$-plane equation \eqref{QP.Eq.to.solve} of large amplitude with the same properties for any value of $\lambda\gg 1$ sufficiently large. Let us make some remarks on the result.
\\[1mm]
\noindent 1) {\it The role of the quasi-periodic traveling structure.} When the free dynamics of a system (that is, without an external force) is invariant by translations in space, traveling wave functions, either periodic or quasi-periodic, form a natural class where to search a solution to the evolution problem that preserves this symmetry. In the context of the model considered here, it is actually an important structure that we need to impose in order to avoid potential degeneracies and multiplicity of the eigenvalues, already at the linear level. More specifically, the linear dynamics of the equation \eqref{QP.Eq.to.solve} is governed by the operator $\lambda \, \omega\cdot \pa_{\vf} + \beta \,\tL $, whose spectrum when acting on functions in $H^s(\T^\nu\times \T^2)$ is given by
\begin{equation}
	{\rm spec} \big( \lambda \, \omega\cdot \pa_{\vf} + \beta \,\tL  \big) = \big\{ \im (\lambda\,\omega \cdot \ell +\beta \,\tL(j) ) \,: \, \ell \in \Z^\nu , \ j =(j_1,j_2)^\top \in \Z^2  \big\}\,.
\end{equation}
Assuming non-resonance conditions as the parameter varies (that shall be imposed later), the degeneracy is due to the fact the $0\in {\rm spec} \big( \lambda \, \omega\cdot \pa_{\vf} + \beta \,\tL  \big)$ has infinite multiplicity, since it corresponds to any $(\ell,j)=(0,(0,j_2)^\top)$, with $j_2\in\Z$, recalling \eqref{lambda.m.def}. When acting on quasi-periodic traveling waves $v(\vf,x)=\breve{v}(\vf-\pi(x))$, where $(\vf,x)\in\T^\nu\times\T^2$ and $\vf-\pi(x)\in\T^\nu$ (see also Definition \ref{def2.QPT}), it is worth noting that the average in the angle $\vf\in\T^\nu$ coincides with the average taken over both coordinates $(\vf,x)$.
This follows from the simple computation
\begin{equation}
	 \int_{\T^\nu} v(\vf,x) \wrt \vf = \int_{\T^\nu} \breve{v}(\vf-\pi(x)) \wrt \vf =  \int_{\T^\nu} \breve{v}(\vf) \wrt\vf = \text{constant w.r.t.}\,\, x\,.
\end{equation}
At the Fourier level, this implies that if $\ell=0$, then necessarily $j=0$ as well. \emph{The previous degeneracy is avoided when the operator $\lambda \, \omega\cdot \pa_{\vf} + \beta \,\tL $ acts on quasi-periodic traveling waves.}
\\[1mm]
\noindent 2) {\it The role of the large frequency and of the large forcing term.} The size of the forcing term is $\lambda^{\alpha}$, with $\lambda \gg 1$. We consider the regime $\alpha\in(1,2)$ for the following reasons. At a purely heuristic level, one can expect the size of a (quasi-)periodic solution to be given by:

\begin{equation}\label{euristica}
    \text{size of the solution} = \frac{\text{size of the force}}{\text{size of the frequency}} = \frac{O(\lambda^{\alpha})}{O(\lambda)} = O(\lambda^{\alpha-1}) \,.
\end{equation}

When $\alpha\in(0,1)$, the expected solution will be of small size with respect to $\lambda\gg1$. If $\alpha=1$, it will not be of small size, but uniformly bounded in $\lambda\gg1$. The regime of interest is when $\alpha>1$. However, searching for solutions of size $O(\lambda^{\alpha-1})$ implies that the quadratic nonlinearity $u \cdot \nabla v $ is of size $O(\lambda^{2\alpha-2})$. To apply perturbative arguments in our scheme, we need the size of the nonlinear terms to be smaller than the size $O(\lambda^{\alpha})$ of the external forcing term. This condition implies the restriction $\alpha<2$, which we apply. It is worth noting that we do not know if solutions in the regime $\alpha\geq 2$ might exist or not: handling such cases would likely require different strategies and techniques. Moreover, since the construction of the solution deals with small divisors and relative non-resonances conditions, the actual estimates for the solutions do not follow exactly the heuristic in \eqref{euristica}, but we have to slightly worsen the rate for the upper bound, as stated in \eqref{stima.grande.solutione}. 
  \\
 We also want to emphasize that the threshold $\lambda_0(f,s,\nu,\beta)\gg 1$ in Theorem \eqref{teo principale beta plane} needs to be strictly larger than $|\beta|>0$. This condition is crucial to obtain solutions of effective large size, satisfying the lower bound in \eqref{stima.grande.solutione}. 
The solutions will exist for any fixed value of $\beta\in\R\setminus\{0\}$, regardless of its size. However, it is essential to note that these solutions will exhibit large amplitude only if the fast external oscillations $\lambda\,\omega$, with $\lambda\gg 1$, are stronger than the internal speed of rotation $\beta\neq 0$.
\\[1mm]
\noindent 3) {\it Reversible solutions and linear stability.}  It is possible to show that the free dynamics, namely equation \eqref{beta.waves.large.eq} without the external forcing term, is \emph{reversible}, which is a reflection of the well-known property of the Euler-Coriolis equations. If we assume parity conditions on the external forcing, in particular $f(\vf,x)$ being even in the pair $(\vf,x)$, then the equation \eqref{QP.Eq.to.solve} is still reversible and we can adapt our scheme to produce solutions $v_{\lambda}$ in Theorem \ref{teo principale beta plane} that are odd in the pair $(\vf,x)$. Moreover, we can prove under this extra symmetry that these quasi-periodic traveling waves solutions are \emph{linearly stable}, meaning that the Floquet exponent of the linearized operator at each solution, with the action of the latter restricted to quasi-periodic traveling waves, are purely imaginary. However, the symmetry of the reversible structure is \emph{not necessary} for the construction of the solutions in Theorem \eqref{teo principale beta plane} itself, but rather only for the analysis of their stability.


\subsubsection*{Literature on quasi-periodic waves in fluids.}
In the last years, there has been a discrete surge of works proving the existence of time quasi-periodic waves for PDEs arising in fluid dynamics.
Most of these results in literature are proved by means of KAM for PDEs techniques, to deal with the presence of small divisors issues and consequent losses of regularity.
For the two dimensional water waves equations, we mention
Berti \& Montalto \cite{BM}, Baldi, Berti, Haus \& Montalto \cite{BBHM} for time quasi-periodic standing waves and Berti, Franzoi \& Maspero \cite{BFM,BFM21}, Feola \& Giuliani \cite{FeoGiu} for time quasi-periodic traveling wave solutions. 
Recently, the existence of time quasi-periodic solutions was proved for the contour dynamics of vortex patches in active scalar equations. 
We mention Berti, Hassainia \& Masmoudi \cite{BertiHassMasm} for vortex patches of
the Euler equations close to Kirchhoff ellipses, Hmidi \& Roulley \cite{HmRo} for the 
quasi-geostrophic shallow water equations,  Hassainia, Hmidi \& Masmoudi \cite{HaHmMa} for generalized surface quasi-geostrophic equations, Roulley \cite{Roulley} for Euler-$\alpha$ flows, Hassainia \& Roulley \cite{HaRo} for Euler equations in the unit disk close to Rankine vortices and Hassainia, Hmidi \& Roulley \cite{HaHmRou} for 2D Euler annular vortex patches.
Time quasi-periodic solutions were also constructed for the 3D Euler equations with time quasi-periodic external force \cite{BM20} and for the forced 2D Navier-Stokes
equations \cite{FrMo} approaching in the zero viscosity limit time quasi-periodic solutions of the 2D Euler equations for all times. 

We also mention that time quasi-periodic solutions for the Euler equations were constructed also by Crouseilles \& Faou \cite{CF} in
2D, and by Enciso, Peralta-Salas \& de Lizaur \cite{EPsTdL} in 3D and
even dimensions: we remark that these latter solutions are engineered so that there
are no small divisors issues to deal with, with consequently much easier proofs and a
drawback of not having information on the potential stability of the solutions. Their construction has been recently adapted in \cite{FM23} to prove the existence of time almost-periodic solutions  for the Euler equation in 3D and even dimensions.

We conclude this part of the introduction with the following remarks. Our result, in companion with the recent work \cite{CiMoTe} on bi-periodic traveling wave solution for the non-resistive MHD system, is the first one in which  the existence of {\it large amplitude} quasi-periodic solutions for a PDE with a forcing term of large size is proved. Up to now, large amplitude quasi-periodic solutions were constructed only for small perturbations of defocusing NLS and KdV equations by Berti, Kappeler \& Montalto \cite{BKM18,BKM}, where the integrable structures of these two equations is exploited. We also mention that KAM techniques have been developed to study the dynamics of linear wave and Klein Gordon equations with potential of size $O(1)$ and large frequencies $\lambda\,\omega$ with $\lambda\gg1$, see \cite{FM19,Franzoi23}.

\subsection{Strategy of the proof}\label{strategy.section}

The solution $v_{\lambda}(\lambda \,\omega t,x)$ in Theorem \ref{teo principale beta plane} that we are going to construct is a quasi-periodic traveling wave, according to Definition \ref{def1.QPT}, of the form, for some $\theta>0$ to determine,
\begin{equation}
	v_{\lambda}(\vf,x) = \lambda^{\theta} w_{\lambda}(\vf,x) \,, \quad w_{\lambda}(\vf,x)=v_{{\rm app},N} (\vf,x) + g(\vf,x)\,,
\end{equation}
so that $v_\lambda= O(\lambda^{\theta})$ and $w_{\lambda} = O(1)$ with respect to $\lambda\to \infty$.  According to this rescaling, 
the function $w_{\lambda}(\vf,x)$ is searched as a quasi-periodic traveling solution of the equation
\begin{equation}\label{strategy.rescaled}
	\big(\lambda \,\omega \cdot \pa_{\vf}+ \beta \,\tL \big)w_{\lambda} +\lambda^{\theta} u \cdot \nabla w_{\lambda} = \lambda^{\alpha-\theta} f(\lambda\, \omega t,x) \,, \quad u = \fB[w_{\lambda}]\,.
\end{equation}

The leading order of function $w_{\lambda}(\vf,x)$ is the first approximation given by the function $v_{{\rm app},N}$ of size $O(1)$.  The function $g(\vf,x)$, on the other hand, will be a correction to $v_{{\rm app},N}(\vf,x)$, bifurcating from the latter, of the smaller size $o(1)$ with respect to $\lambda\to \infty$. 
If we follow the heuristic as in \eqref{euristica}, we would expect to simply fix $\theta=\alpha-1$. However, due to the issue of small divisors in inverting the operator $\lambda\,\omega\cdot \pa_{\vf} +\beta\,\tL$, we have to worsen the size and take $\theta$ slightly larger than $\alpha-1$ to get an effective upper bound on $w_\lambda$ of size $O(1)$.
We now sketchily illustrate how to construct these two terms and the main ideas of the strategy.

\smallskip

\noindent
{\bf Construction of the approximate solution $v_{{\rm app},N}$.} In order to retrieve a perturbative framework when $\lambda\gg 1$, the first step is to search for a function that solves \eqref{strategy.rescaled} up to some error term that is perturbatively small, namely of size $o(1)$ with respect to $\lambda \to \infty$. We do so by searching for a quasi-periodic traveling wave function $v_{{\rm app},N}(\vf,x)$  of the form
\begin{equation}\label{sanremo}
	v_{{\rm app},N}(\vf,x) = v_0(\vf,x) + v_{1}(\vf,x) +...+ v_{N}(\vf,x)\,,
\end{equation}
where $N\in\N_{0}$ is the number of iterations needed to produce the desired correction to the approximation. The first term of this finite expansion is given as the solution of the linear equation
\begin{equation}
	\big(\lambda \,\omega \cdot \pa_{\vf}+ \beta \,\tL \big) v_{0} = \lambda^{\alpha-\theta} f \,.
\end{equation}
Building upon the preceding discussion below 
Theorem \ref{teo principale beta plane},
in order to mitigate the otherwise infinite-dimensional degeneracy in the kernel, the linear operator
 $\lambda \,\omega \cdot \pa_{\vf}+ \beta \,\tL$ has to be inverted on quasi-periodic traveling waves with zero average. But, to deal with the presence of small divisors, we need to impose a non-resonance condition  with a consequent loss of derivatives in the action of the inverse. This non-resonance condition is given by
\begin{equation}
	|\lambda\,\omega \cdot \ell  + \beta\, \tL(j)| \geq \lambda \frac{\gamma}{|\ell|^\tau} \,, \quad \forall\,(\ell,j) \in\Z^{\nu+2}\setminus\{0\} \,, \quad \pi^\top (\ell)+j=0\,. 
\end{equation}
 for some given $\gamma\in (0,1)$ and $\tau>\nu+1$, where $\pi^\top (\ell)+j=0$ is the Fourier restriction on the modes of quasi-periodic traveling waves. For this reason, we obtain
 \begin{equation}\label{bound.brunori}
 	\| \big( \lambda \,\omega \cdot \pa_{\vf}+ \beta \,\tL \big)^{-1} \|_{\cL(H^{s+\tau}_0,H^s_0)} \leq O(\lambda^{-1}\gamma^{-1})\,.
 \end{equation}
 To impose smallness conditions later on, the small parameter $\gamma\in(0,1)$ coming from the non-resonance conditions will be linked to $\lambda$ by choosing $\gamma=\lambda^{-\tc}$, for some $\tc \in (0,1)$ sufficiently small. This implies that the upper bound in \eqref{bound.brunori} is actually of size $O(\lambda^{-1+\tc})$.
 On the other hand, together with the easy upper bound on the denominator $	|\lambda\,\omega\cdot\ell + \beta\,\tL(j)| \leq  \lambda |\omega||\ell|$, under the assumption that $\lambda \geq |\beta|$ and recalling \eqref{lambda.m.def}, we will show that $	\| \big( \lambda \,\omega \cdot \pa_{\vf}+ \beta \,\tL \big)^{-1} \|_{\cL(H^{s}_0,H^{s+1}_0)} \geq O(\lambda^{-1})$. Therefore, we obtain that
 \begin{equation}\label{prima.stima}
 	O(\lambda^{\alpha-\theta-1})  \leq \| v_{0} := \big(\lambda \,\omega \cdot \pa_{\vf}+ \beta \,\tL\big)^{-1}[\lambda^{\alpha-\theta} f] \|_{s} \leq O(\lambda^{\alpha-\theta-1+\tc})  \,.
 \end{equation}
 To get the upper bound  in \eqref{prima.stima} of size $O(1)$, we fix $\theta=\alpha-1+\tc$, whereas the lower bound becomes of size $O(\lambda^{-\tc})$. These estimates of the leading order term of the solutions have the same rate, after rescaling, of the estimates of the whole solution in \eqref{stima.grande.solutione}. For $N=0$, the error that is produced by approximately solving \eqref{strategy.rescaled} is given by the nonlinear term $\lambda^{\theta} \fB[v_0] \cdot \nabla v_0$, which is of size $O(\lambda^{\theta})=O(\lambda^{\alpha-(1-\tc)})$, smaller than the size of the forcing term $\lambda^{\alpha-\theta} f= \lambda^{1-\tc}f$ for $\alpha\in (1,2)$ and $\tc>0$ small enough so that $\alpha<2(1-\tc)$. The idea is to iterate this argument to construct the next-order approximate solution $v_1$, which solves the linear equation with a forcing term represented by the error from the previous step. That is
 \begin{equation}
 	v_1 := - \big( \lambda \,\omega \cdot \pa_{\vf}+ \beta \,\tL \big)^{-1} [\lambda^{\theta} \fB[v_0] \cdot \nabla v_0] = O(\lambda^{\theta-1+\tc}) \,,
 \end{equation}
which is small with respect to $\lambda\to \infty$, so that the approximate solution $v_{{\rm app},1}=v_0+v_1$ solves \eqref{strategy.rescaled} up to another error term that is even smaller with respect to the one created by $v_0$ alone. Iterating this procedure for $N$ steps, we obtain the approximate solution $v_{{\rm app},N} = v_0 + v_1 + \ldots + v_N$, where each $v_i$, $i=1,\dots, N$, is small as $\lambda\to\infty$, and $v_{{\rm app},N}$ in \eqref{sanremo} satisfies 
\begin{equation}
		\big(\lambda \,\omega \cdot \pa_{\vf}+ \beta \,\tL \big)v_{{\rm app},N}+\lambda^{\alpha-1} \fB[v_{{\rm app},N}] \cdot \nabla v_{{\rm app},N} - \lambda f(\vf ,x) = q_{N}(\vf,x) \,,
\end{equation}
with $q_N$ a quasi-periodic traveling wave such that $\| q_N \|_{s} \leq O(\lambda^{N(\alpha-2(1-\tc))+\alpha-(1-\tc)})$. Choosing $N > \frac{\alpha-(1-\tc)}{2(1-\tc)-\alpha}$ ensures $\| q_N\|_s  \lesssim o(1)$ as $\lambda\to\infty$. The detailed construction is provided in Proposition \ref{lemma.approx}.

\smallskip

\noindent {\bf Linearized operator and Nash-Moser iteration.}  We are now in place to search for solution for \eqref{strategy.rescaled} by means of a Nash-Moser iteration scheme, that is we look for the zeroes of the nonlinear functional
\begin{equation}\label{strategy.F}
	\cF(w) :=\big(\lambda \,\omega \cdot \pa_{\vf}+ \beta \,\tL \big) w +\lambda^{\theta} \fB[w] \cdot \nabla w - \lambda^{\alpha-\theta} f(\vf,x) \,,
\end{equation}
with $\theta:=\alpha-1+\tc>0$, as fixed before.
We use as initial guess the approximate solution $v_{{\rm app},N}$ constructed before and we want to find a sequence of quasi-periodic traveling wave functions $(w_n = v_{{\rm app},N} + g_n)_{n\in\N_{0}}$, $g_0:=0$, that converge to a zero of the nonlinear operator $\cF$. Roughly speaking, the sequence is determined via a Newton scheme, with
\begin{equation}
	w_{n+1} - w_{n} = - \big(\di_{w} \cF(w_n) \big)^{-1}\big[ \cF(w_n) \big]\,. 
\end{equation}
The main issue, therefore, is to analyze the linearized operator $ \di_{w} \cF(w_n)$, determine its invertibility and estimate the inverse operator. Linearizing \eqref{strategy.F} at $w_n$, we focus on
\begin{equation}\label{strategy.linear.op}
	\cL := \di_{w} \cF(w_n) := \lambda\,\omega\cdot \pa_{\vf} + \beta\,\tL + \lambda^{\theta} \fB[w_n] \cdot \nabla + \lambda^{\theta} \nabla w_n \cdot \fB\,.
\end{equation}
The main issue is that, in the regime $\lambda\gg 1$, the operator $\cL$ in \eqref{strategy.linear.op} is a  perturbation of large size $O(\lambda^\theta)$, $\theta <  1$ of the diagonal operator $\lambda\,\omega\cdot \pa_{\vf} + \beta\,\tL$. The strategy to establish invertibility and related estimates for the operator $\cL$ in \eqref{strategy.linear.op} involves reducing it to a diagonal operator with respect to the Fourier basis. This reduction procedure will be carried out in three steps.
\\[1mm]
\indent 1) Reduction of the transport.
\\[1mm]
\indent 2) Reduction to {\it smoothing} and {\it perturbative} of the large remainder.
\\[1mm]
\indent 3) KAM reducibility scheme.
\\[1mm]
\noindent If we achieve this reduction with a bounded invertible transformation close to the identity, we can deduce tame estimates for the action of the operator $\big(\di_{w} \cF(w_n) \big)^{-1}$. This operator will also experience a loss of derivatives, akin to the operator $\lambda\,\omega\cdot \pa_{\vf} + \beta\,\tL$. However, the Nash-Moser scheme compensates for this loss, ensuring fast convergence of the iterations in a low regularity norm while allowing controlled divergence in a high regularity norm. The invertibility of the operator $\cL$ in \eqref{strategy.linear.op} is discussed in Section \ref{sez:inverti} after its reduction, while the Nash-Moser iteration is proven in Section \ref{sezione:NASH} and finally Theorem \ref{teo principale beta plane} in Section \ref{sez:measures}.

The remaining part of this introduction will briefly describe the steps involved in the reduction scheme of Section \ref{ridusezione}, which constitutes a substantial part of this paper.

\smallskip 
\noindent {\it 1) Reduction of the transport.} The first step is to reduce the highest order term in \eqref{strategy.linear.op}, which is represented by the transport operator
\begin{equation}
	\cT := \lambda \,\omega\cdot \pa_{\vf} + \lambda^{\theta} \fB[w_n] \cdot \nabla = \lambda\big( \omega\cdot \pa_{\vf}  + \lambda^{\theta-1}  \fB[w_n] \cdot \nabla \big)\,.
\end{equation}
Since $\theta-1=\alpha-2+\tc<0$ for $\alpha< 2$ and $\tc>0$ small enough, the rescaled vector field $\lambda^{\theta-1} \fB[w_n]$  is perturbatively small with respect to $\lambda\to\infty$. Moreover, it has zero average in space and it is divergence free. Thanks to these properties,  we will conjugate the operator $\cT$ to fully reduce it to the operator $\lambda\,\omega\cdot \pa_{\vf}$, as soon as the frequency vector $\omega\in\R^\nu\setminus\{0\}$ is Diophantine (see \eqref{DC.2gamma}). This is the content of Proposition \ref{ridu trasporto}, which follows the scheme in \cite{BM20}. 

The remaining terms in \eqref{strategy.linear.op}, namely $\beta\,\tL + \lambda^{\theta} \nabla w_n\cdot \fB$, undergo the same transformation to an operator of the form $\beta\,\tL + \cE_{1}$, where $\cE_{1}$ is a remainder operator of size $\lambda^{\theta}$ belonging to the class $\OpM_{s}^{-1}$ of operators that exhibit 1-smoothing according to off-diagonal matrix decay (see Definition \eqref{block norm} and subsequent properties for precise definitions). Additionally, since $w_n(\vf,x)$ is a quasi-periodic traveling wave, the remainder $\cE_{1}$ is \emph{momentum preserving}, meaning it maps quasi-periodic traveling waves into quasi-periodic traveling waves (refer to Section \ref{subsec:momentum} for definitions and characterizations). This property is crucial for the next two steps.

\smallskip 
\noindent {\it 2) Reduction to a small and smoothing remainder.} After the previous transformation, we are left to work with the operator
\begin{equation}
	\cL_{1} := \lambda\,\omega\cdot \pa_{\vf} + \beta\,\tL + \cE_{1}\,.
\end{equation}
For the  KAM reduction, the remainder $\cE_{1}$ lacks necessary smoothing for derivative compensation and remains large for $\alpha\in (1,2)$. The next step is to conjugate the operator $\cL_{1}$ to reduce both the size and order of the remainder $\cE_{1} \in \OpM_{s}^{-1}$. This situation represents a novel development in normal form techniques compared to the previous literature.

To address this, we conjugate the operator $\cL_{1}$ in a series of $M-1$ iterations, culminating in the remainder $\cE_{M}$ of the last iteration being of size $o(1)$ as $\lambda\to\infty$ and in the class $\OpM_{s}^{-M}$. We briefly descrive the first iteration. We conjugate $\cL_{1}$ with $\Phi_{1}= {\rm exp} (\cX_{1})$, with this transformation being invertible for $\cX_{1}=\cX_{1}(\vf) \in \OpM_{s}^{-1}$ sufficiently small. By expanding $\Phi_{1}^{-1} \cL_{1} \Phi_{1}$ via the standard Lie expansion, the only term of order $- 1$ is given by 
$$
	\lambda \,\omega\cdot \pa_{\vf} \cX_{1} (\vf)+ \cE_{1}(\vf)\,.
$$
The sublinear nature of the dispersion relation is crucial. Indeed the effect of the operator $\lambda\,\omega\cdot \pa_{\vf}$ is stronger than the dispersive operator $\beta \, \tL$. This fact implies that the commutator $[\cX_{1} (\vf)\,,\, \beta\,\tL ]$ is of order $- 2$ and hence it contributes to the remainder of the Lie expansion. 
We solve the homological equation
\begin{equation}\label{strategy.homolo}
	\lambda \,\omega\cdot \pa_{\vf} \cX_{1} (\vf)+ \cE_{1}(\vf) = \wh\cE_{1}(0) \,, \quad \cE_{1}(\vf) = \sum_{\ell\in\Z^\nu} \wh\cE_{1}(\ell) e^{\im\ell\cdot \vf}  
\end{equation}
in which we remove the dependence on time from the remainder $\cE_{1}(\vf)$ by inverting the operator $\lambda\,\omega\cdot \pa_{\vf}$, provided the frequency vector is $\omega $ is Diophantine.
A second key property here is that, since $\cE_{1}$ is a momentum-preserving operator, the operator $\wh\cE_{1}(0)$, that is the average in time of $\cE_{1}(\vf)$, is diagonal with respect to the Fourier basis $(e^{\im j \cdot x})_{j\in\Z^2}$ (as shown in Lemma \ref{lem:mom_pres}). We obtain that $\cX_{1} \in \OpM_{s}^{-1}$ is of size $O(\lambda^{\theta-1+\tc}) = O(1)$, as $\lambda\to\infty$, ensuring estimates on the maps $\Phi_{1}^{\pm 1}$ uniform with respect to $\lambda\gg1$ large enough.

The operator $\wh\cE_{1}(0)$ contributes to a new diagonal operator $\lambda\,\omega\cdot \pa_{\vf} + \beta\,\tL + \wh\cE_{1}(0)$. The new remainder, denoted by $\cE_{2}$, contains terms that are neither solved by the homological equation \eqref{strategy.homolo} nor part of the normal form. The leading term of this new remainder is given by the commutator $[\cX_{1},\cE_{1}]$, which is of size $O(\lambda^{2\theta-1+\tc})$  and, by algebraic properties, in the class $\OpM_{s}^{-2}$. Therefore, the ensuing remainder $\cE_{2}$ is smaller in size for $\alpha-2(1-\tc)<0$ (we then take $0 < \mathtt c \ll 1$ small enough) and exhibits more smoothing than the previous remainder $\cE_{1}$. As in the construction of the approximate solution $v_{{\rm app},N}$, we iterate this procedure to reduce the remainder to a perturbatively small size. The details of this reduction are provided in Proposition \ref{prop normal form lower orders}.


\smallskip 
\noindent {\it 3) KAM perturbative reduction.} The conclusion of the previous step is the linear operator
\begin{equation}\label{strategy.cLM}
	\cL_{M} = \lambda\,\omega\cdot \pa_{\vf} + {\bf D}_0 + {\bf E}_0, \quad {\bf D}_0 :=  \beta\,\tL + \cZ_{M}, \quad {\bf E}_0 := \cE_{M}\,.
\end{equation}

Here, the momentum preserving and constant coefficients operator $\lambda\,\omega\cdot \pa_{\vf} + \beta\,\tL + \cZ_{M} $ is diagonal with respect to the Fourier basis, featuring eigenvalues $\im\,\lambda\, \omega\cdot \ell + \mu_{0}(j)$, where $\mu_{0}(j):=\im \,\beta\, \tL(j) + \tz_{0,M}(j)\in \C$. For any $\ell\in\Z^\nu\setminus\{0\}$ and $j\in\Z^2\setminus\{0\}$, the operator $\cE_{M}\in \OpM_{s}^{-M}$ is small, with a size of $O(\lambda^{M(\theta-1+\tc)+1-\tc})$ as $\lambda\to\infty$ (note that $\theta-1+\tc < 0$). 
If we assume parity properties on the external forcing term and we preserve the reversible structure of the system in all the steps of the reduction, than we get the diagonal elements $\tz_{0,M}(j)$ in $\cZ_{M}$ are purely imaginary. 
Now, the goal is to fully reduce the operator $\cL_{M}$ to a constant-coefficients diagonal operator through a series of conjugations. This process progressively annihilates the size of the perturbative remainder, ultimately yielding the final normal form $\lambda\,\omega\cdot \pa_{\vf} + \beta\,\tL + \bZ_{\infty}$. A crucial problem, typical in PDEs in higher space dimension, is the presence of strong resonance phenomena. In this case, the difference of unperturbed eigenvalues $\tL(j) - \tL(j')$ is zero for infinitely many $j, j' \in \Z^2 \setminus \{ 0 \}$. Also at this step, this degeneracy issue is overcome by using the conservation of momentum. Indeed, at the first iterative step, the non-resonance conditions, required in the iterative KAM procedure, namey the second-order Melnikov conditions, take the form
\begin{equation}\label{strategy.melnikov2}
	|\im \lambda\,\omega\cdot \ell + \mu_{0}(j)-\mu_{0}(j') | \geq \lambda\frac{\gamma}{\braket{\ell}^\tau |j'|^\tau} \,, 
\end{equation}
for any $\ell \in \Z^\nu\setminus\{0\}$ and $j,j'\in\Z^2\setminus\{0\}$.  It is noteworthy that the non-resonance conditions in \eqref{strategy.melnikov2}, 
 needed for estimating small divisors in the solution of the homological equations, lose regularity in both time and space.  However, the key feature is the gain in size by a factor $\lambda\gg1$. This is crucial to impose the appropriate smallness condition, ensuring the convergence of the scheme.
 
 \noindent
 By imposing the non-resonance condition \eqref{strategy.melnikov2} above, one can construct a map $\Psi_0$ solving the {\it homological equation}
$$
\lambda\, \omega \cdot \partial_\vf \Psi_0(\vf) + [{\bf D}_0\,,\, \Psi_0(\vf)] + {\bf E}_0(\vf) = \widehat{\bf E}_0(0), \quad \widehat{\bf E}_0(0) = \int_{\T^\nu} {\bf E}_0(\vf)\, \wrt \vf\,,
$$
and one gets 
$$
e^{- \Psi_1} {\mathcal L}_M e^{\Psi_1} = \lambda\, \omega \cdot \partial_\vf + {\bf D}_0 + \widehat{\bf E}_0(0) + {\bf E}_1\,,
$$
where ${\bf E}_1$ is of order $- M$ and it has size ${\bf E}_1 \simeq {\bf E}_0^2$ which is essentially quadratic with respect to the size of ${\bf E}_0$. It is now crucial to use the conservation of momentum, which allows to deduce that $\widehat{\bf E}_0(0)$ and also ${\bf D}_1 : =  {\bf D}_0 + \widehat{\bf E}_0(0)$ are diagonal operators with respect to the Fourier basis $\{ e^{\im j \cdot x} : j \in \Z^2 \setminus \{ 0 \} \}$. 

\noindent
  The detailed construction of the KAM reduction are provided in Sections \ref{sez:iteraKAM} and \ref{sez convergenza KAM}. As remarked below Theorem \ref{teo principale beta plane}, if we assume the preservation of the reversible structure, then we can infer that the final constant-coefficients diagonal normal form has purely imaginary eigenvalues and deduce the linear stability of the final solution. Nevertheless, 
the construction of these large amplitude solutions in Theorem \ref{teo principale beta plane} does not depend on this property.


\medskip

\noindent {\bf Acknowledgements.} L. Franzoi, R. Montalto and S. Terracina are supported by the ERC STARTING GRANT 2021 “Hamiltonian Dynamics, Normal Forms and Water Waves” (HamDyWWa), Project Number: 101039762. Views and opinions expressed are however those of the authors only and do not necessarily reflect those of the European Union or the European Research Council. Neither the European Union nor the granting authority can be held responsible for them.
R. Bianchini and R. Montalto acknowledge the support of PRIN 2022 “Turbulent effects vs Stability in Equations from Oceanography” (TESEO), project number: 2022HSSYPN. R. Bianchini is partially supported by the Royal Society International Exchange Grant 2020 (CNR - Imperial College London). 
R. Montalto is also supported by INdAM-GNFM. R. Bianchini and L. Franzoi and S. Terracina are also supported by INdAM-GNAMPA.

\noindent
The authors warmly thank Taoufik Hmidi and Alberto Maspero for useful discussions and comments.

\section{Function spaces, norms and linear operators}\label{sez:functional}

\subsection{Function spaces}
\label{subsec:function spaces}
Let $d=2,3$ be the space dimension. Let $a : \T^\nu \times \T^d \to E$, $a = a(\vf,x)$, 
be a function, with $E = \C^m$ or $\R^m$. Then, for $s \in \R$, its Sobolev norm $\| a \|_s$ is defined as 
\begin{equation*} \label{def Sobolev norm generale}
	\| a \|_s^2 := \sum_{(\ell, j) \in \Z^\nu \times \Z^d} 
	\langle \ell, j \rangle^{2s} | \widehat a(\ell,j) |^2 \,,
	\quad \ 
	\langle \ell, j \rangle := \max \{ 1, |\ell|, |j| \} \,,
\end{equation*}
where $\widehat a(\ell,j)$ (either scalars or vectors) 
are the Fourier coefficients of $a(\vf,x)$, namely
\[
\widehat a(\ell,j) := \frac{1}{(2\pi)^{\nu+d}} \int_{\T^{\nu+d}} 
a(\vf,x) e^{- \im (\ell \cdot \vf + j \cdot x)} \, \wrt \vf \wrt x \,.
\]
We denote, for $E = \C^m$ or $\R^m$,
\begin{equation}	\label{def sobolev}
	\begin{aligned}
		H^s 
		&
		:= H^s_{\vf,x} 
		:= H^s(\T^{\nu} \times \T^d) \\
		&
		:= H^s(\T^{\nu} \times \T^d, E) 
		:= \{ u : \T^{\nu} \times \T^d \to E, \ \| u \|_s < \infty \} \,.
	\end{aligned}
\end{equation}

In this paper, we use Sobolev norms for 
(real or complex, scalar- or vector-valued) functions $u( \vf, x; \omega)$, 
$(\vf,x) \in \T^\nu \times \T^d$, being Lipschitz continuous with respect to the parameter 
$\omega\in\R^\nu\setminus\{0\}$.
We fix the threshold regularity
\begin{equation}\label{definizione s0}
	s_0 \geq  \nu+d + 3
\end{equation}
and we define the weighted Sobolev norms in the following way. 

\begin{defn} 
	{\bf (Weighted Sobolev norms).} 
	\label{def:Lip F uniform} 
	Let   $\gamma \in (0,1]$, $\Lambda \subseteq \R^{\nu}$ be an arbitrary closed set
	and $s \geq s_0$ with $s_0$ as in \eqref{definizione s0}.
	Given a function $u : \Lambda \to H^s(\T^\nu \times \T^d)$, 
	$\omega \mapsto u(\omega) = u(\vf,x; \omega)$ 
	that is Lipschitz continuous with respect to $\omega$, 
	we define its weighted Sobolev norm by
	$$
	\| u \|_{s}^{\Lip(\gamma)} := \| u\|_{s}^{\sup} + \gamma \,\| u\|_{s-1}^{\rm lip}\,,
	$$
	where
	\begin{equation*}
		\| u\|_{s}^{\sup} := \sup_{\omega\in \Lambda} \| u(\omega)\|_{s}\,, \quad \| u\|_{s}^{\rm lip}:= \sup_{\omega_1,\omega_2\in \Lambda  \atop \omega_1\neq \omega_2} \frac{\| u(\omega_1)-u(\omega_2)\|_{s}}{| \omega_1-\omega_2|}\,.
	\end{equation*}
	For $u$ independent of $(\vf,x)$, we simply denote by 
	$| u |^{\Lip(\gamma)}:= | u|^{\sup} + \gamma \, | u|^{\rm lip} $.
\end{defn}

\begin{lem}{\bf (Product).}
	\label{lemma:LS norms}
	For all $ s \geq s_0$, 
	\begin{align}
		\| uv \|_{s}^{\Lip(\gamma)}
		& \lesssim_s C(s)  \| u \|_{s}^{\Lip(\gamma)} \| v \|_{s_0}^{\Lip(\gamma)}+ C(s_0)  \| u \|_{s_0}^{\Lip(\gamma)} \| v \|_{s}^{\Lip(\gamma)}\,.
		\label{p1-pr}
	\end{align}
\end{lem}

\smallskip
{\bf Notation.}
We will have also a dependence on  
$\T^{\nu}\ni\vphi\to w(\vphi, \cdot)\in L^2_0(\T^2)$ 
of which we want to control the regularity for the purposes of the nonlinear scheme.
In particular, we will ask to be Lipschitz  with respect to the function $w$ as a parameter
and therefore it is natural to introduce the following quantity.
Given the map
$w \mapsto g(w)$
where $g$ is an operator (or a map or a scalar function), we define 
\begin{equation}\label{deltaunodue}
	\Delta_{12} g := g(w_{2}) - g(w_{1})\,.
\end{equation}

\subsubsection{Quasi-periodic traveling functions}\label{sec:travelfunction}

%

We restate the definition of quasi-periodic traveling as given in Definition \ref{def1.QPT}, for functions of $(\vf,x)$ instead of $(t,x)$.

\begin{defn}\label{def2.QPT}
	{\bf (Quasi-periodic traveling waves).} Let $\bar\jmath_{1},...,\bar\jmath_{\nu}\in \Z^d$ be a given choice of $\nu$ vectors in $\Z^d$.
	A function $u:\T^\nu \times \T^d\to \R$ is a \emph{quasi-periodic traveling wave} if there exists a function $\breve{u}:\T^\nu \to\R$   such that
	\begin{equation}
		u(\vf,x) =  \breve{u}(\vf- \pi( x)) \,, \quad \forall (\vf,x) \in\T^\nu \times\T^d\,,
	\end{equation}
	where  $\pi :\R^d \to \R^\nu$ is the linear map $x \mapsto \pi( x) := (\bar{\jmath}_{1}\cdot x,...,\bar{\jmath}_{\nu}\cdot x)$. 
\end{defn}
Comparing with Definition \ref{def1.QPT} (with $d=2$), it is convenient to call \emph{quasi-periodic} traveling wave both the function $u(\vf,x) =  \breve{u}(\vf- \pi( x))$ and the function of time $u(\omega t,x) =  \breve{u}(\omega t- \pi( x))$.
We define the translation operator
\begin{equation}\label{def:vec.tau}
	\tau_\vs : h(x) \mapsto h(x + \vs)\,,\;\;\vs\in \R^{d}\,.
\end{equation}
Then, quasi-periodic traveling waves are also characterized by the relation
\begin{equation}\label{condtraembedd}
	v(\vphi-\pi(\vs),x)=
	( \tau_\vs \circ v)(\vphi,x)=v(\vphi,x+\vs)\,,  
	\quad \forall \,\vphi \in \T^\nu\,,\; \vs\in \R^{d}\,, \; x\in\T^{d}\,.
\end{equation}
Expanding in Fourier the equivalence in \eqref{condtraembedd}, we obtain that a quasi-periodic traveling wave has the form
\begin{equation}\label{QPT.forma}
	u(\vf,x) = \sum_{\begin{subarray}{c}
	(\ell, j) \in \Z^\nu \times \Z^d \\
 \pi^\top(\ell) + j = 0
 \end{subarray}} \whu(\ell,j) e^{\im(\ell\cdot \vf + j\cdot x)}\,.
\end{equation}
We define the subspace, for $E= \C^m$ or $\R^m$,
\begin{equation}\label{SV}
	S_{\pi}:=\big\{v(\varphi, x)\in L^2(\T^{\nu+d},E) \, : \, v(\varphi, x)=V(\varphi- \pi( x)), \ V(\theta)\in L^2(\T^\nu, E)  \big\}\,.
\end{equation}
With abuse of notation, we denote by $S_{\pi}$ also the subspace of $L^2(\T^{\nu+d}, \R)$ 
of scalar functions $v(\varphi, x)$ of the form 
$v(\varphi, x)=v(\varphi- \pi( x)),\, V(\theta)\in L^2(\T^\nu, \R)$). 
We  note that
\begin{equation}\label{travelequi}
	\eqref{condtraembedd} \quad \Leftrightarrow \quad v\in S_{\pi}
	\quad  \Leftrightarrow \quad
	\pi^\top \pa_{\vphi}v+\nabla v=0\,,\;\;\forall(\vphi,x)\in \T^{\nu+d}\,.
\end{equation}

For any $\tK\geq 1 $, we define the smoothing operators on quasi-periodic traveling waves of the form \eqref{QPT.forma} as
\begin{equation}\label{def:smoothings}
	\begin{aligned}
		(\Pi_{\tK} u)(\vf,x) :=  \sum_{\begin{subarray}{c}
		\braket{\ell}\leq \tK\,,\, j \in \Z^d \\
		  \pi^\top(\ell) + j = 0
		  \end{subarray}} \whu(\ell,j) e^{\im (\ell\cdot\vf + j \cdot x)} \,, \quad \
 		\Pi^\perp_{\tK} := {\rm Id} - \Pi_{\tK} \,.
	\end{aligned}
\end{equation}
Note that, if $u(\vf,x)$ is a quasi-periodic traveling wave, then both $(\Pi_{\tK} u)(\vf,x)$ and $(\Pi_{\tK}^\perp u)(\vf,x)$ are quasi-periodic traveling waves as well.

\begin{lem} {\bf (Smoothing).} \label{lemma:smoothing}
	The smoothing operators $\Pi_{\tK}, \Pi_{\tK}^\perp$ satisfy 
	the smoothing estimates
	\begin{align}
		\| \Pi_{\tK} u \|_{s}^{\Lip(\gamma)} 
		& \leq \tK^a \| u \|_{s-a}^{\Lip(\gamma)}\, , \quad 0 \leq a \leq s \,,
		\label{p2-proi} \\
		\| \Pi_{\tK}^\bot u \|_{s}^{\Lip(\gamma)}
		& \leq \tK^{-a} \| u \|_{s + a}^{\Lip(\gamma)}\, , \quad  a \geq 0 \,.
		\label{p3-proi}
	\end{align}
\end{lem}

\subsection{Matrix representation of linear operators}
Let
${\mathcal R}  : L^2(\T^d) \to L^2(\T^d)$ be a linear operator. It can be represented as
\begin{equation}\label{matriciale 1}
	{\mathcal R} u (x) := \sum_{j, j' \in \Z^d} {\mathcal R}_j^{j'}\widehat u(j') e^{\im j \cdot x} \,, 
	\quad  \text{for} \quad  u (x) = \sum_{j \in \Z^d} \widehat u(j) e^{\im j \cdot x} \,,
\end{equation}
where, for $j, j' \in \Z^d$, the matrix element ${\mathcal R}_j^{j'}$ is defined by 
\begin{equation}\label{rappresentazione blocchi 3 per 3}
	{\mathcal R}_j^{j'} :=  \frac{1}{(2\pi)^d} \int_{\T^d} 
	{\mathcal R}[e^{\im j' \cdot x}] e^{- \im j \cdot x} \wrt x\,. 
\end{equation}

We also consider smooth $\vphi$-dependent families of linear operators 
$\T^\nu \to {\mathcal B} (L^2(\T^d))$, $\vphi \mapsto {\mathcal R}(\vphi)$, 
which we write in Fourier series with respect to $\vphi$ as 
\begin{equation}\label{matrix representation 1}
	{\mathcal R}(\vphi) = \sum_{\ell \in \Z^\nu} \widehat{\mathcal R}(\ell) e^{\im \ell \cdot \vphi}, \quad \widehat{\mathcal R}(\ell) := \frac{1}{(2 \pi)^\nu} \int_{\T^\nu} {\mathcal R}(\vphi) e^{- \im \ell \cdot \vphi}\, \wrt \vphi \,, \quad \ell \in \Z^\nu\,. 
\end{equation}
According to \eqref{rappresentazione blocchi 3 per 3}, for any $\ell \in \Z^d$, the linear operator $\widehat{\mathcal R}(\ell) \in {\mathcal B} (L^2(\T^d))$ is identified 
with the matrix $(\widehat{\mathcal R}(\ell)_j^{j'})_{j, j' \in \Z^d}$.
A map $\T^\nu \to {\mathcal B} (L^2(\T^d))$, $\vphi \mapsto {\mathcal R}(\vphi)$ 
can be also regarded as a linear operator 
$L^2(\T^{\nu+d}) \to L^2(\T^{\nu+d })$ by 
\begin{equation} \label{amatriciana}
	{\mathcal R} u(\vphi, x) := \sum_{\begin{subarray}{c}
			\ell, \ell' \in \Z^\nu \\
			j, j' \in \Z^d
	\end{subarray}} \widehat{\mathcal R}(\ell - \ell')_j^{j '} \widehat u(\ell', j') e^{\im (\ell \cdot \vphi + j \cdot x)}, \quad \forall u \in L^2(\T^{\nu + d})\,. 
\end{equation}
If the operator ${\mathcal R}$ is invariant on the space of functions with zero average in $x$, we identify ${\mathcal R}$ with the matrix 
\begin{equation}
	\Big( \widehat{\mathcal R}(\ell)_j^{j'} \Big)_{
			j , j' \in \Z^d \setminus \{ 0 \}, \,
			\ell \in \Z^\nu } \,.
\end{equation}
	Let ${\mathcal R}$ be a linear operator as in 
	\eqref{matriciale 1}-\eqref{amatriciana}. We define ${\mathcal D}_{\mathcal R}$ as the operator
	\begin{equation}\label{diagonal.op.matrix}
		{\mathcal D}_{\mathcal R} := {\rm diag}_{j \in \Z^2} \widehat{\mathcal R}(0)_j^j\,, \quad (\cD_{\cR})(\ell)_j^{j'} := \begin{cases}
			\widehat{\mathcal R}(\ell)_j^{j'} & \quad \text{if} \quad j=j'\,, \ \ell=0\,, \\
			0 & \text{otherwise}\,.
		\end{cases}\,.
	\end{equation}
	In particular, we say that $\cR$
	is a \emph{diagonal operator} if $\cR \equiv \cD_{\cR}$.

For the purpose of the Normal Form methods for the linearized operator in Section \ref{ridusezione}, it is convenient to introduce the following norms. These norms take into account both the order and the off-diagonal decay of the matrix elements representing any linear operator on $L^2(\T^{\nu+d})$.
\begin{defn} \label{block norm}
	{\bf (Matrix decay norm and the class $\OpM^m_s$).}
	Let $m \in \R$, $s \geq s_0$ and $\cR$ be an operator represented by the matrix in \eqref{amatriciana}. We say that ${\mathcal R}$ belongs to the class $\OpM^m_s$ if 
	\begin{equation} \label{def decay norm}
		| \cR |_{m, s} := \sup_{j' \in \Z^d} 
		\Big( \sum_{(\ell,j) \in \Z^{\nu+d}} 
		\langle \ell, j-j' \rangle^{2s} | \widehat \cR(\ell)_j^{j'}|^2  \Big)^{\frac12} \langle j' \rangle^{- m} < \infty \,.
	\end{equation}
	If the operator $\cR = \cR(\omega)$ is Lipschitz with respect to the parameter $\omega \in \Lambda \subseteq  \R^{\nu}$, 
	we define 
	\begin{equation} \label{def decay norm parametri}
		\begin{aligned}
			& | {\mathcal R} |_{m, s}^{{\rm Lip}(\gamma)} := |{\mathcal R}|_{m, s}^{\rm sup} + \gamma |{\mathcal R}|^{\rm lip}_{m, s - 1}\,, \\
			& |{\mathcal R}|_{m, s}^{\rm sup} := \sup_{\omega \in \Lambda} |{\mathcal R}(\omega)|_{m, s} \,, \quad |{\mathcal R}|_{m, s - 1}^{\rm lip} := \sup_{\begin{subarray}{c}
					\omega_1, \omega_2 \in \Lambda \\
					\omega_1 \neq \omega_2
			\end{subarray}} \dfrac{|{\mathcal R}(\omega_1) - {\mathcal R}(\omega_2)|_{m, s - 1}}{|\omega_1 - \omega_2|} \,.
		\end{aligned}
	\end{equation}
\end{defn}

It readily follows that 
\begin{equation}\label{prop elementari}
	\begin{aligned}
		& m  \leq m' \Longrightarrow \OpM^m_s \subseteq \OpM^{m'}_s \quad \text{and} \quad |\cdot |_{m', s}^{{\rm Lip}(\gamma)} \leq |\cdot |_{m, s}^{{\rm Lip}(\gamma)}, \\
		& s \leq s' \Longrightarrow\OpM^m_{s'} \subseteq \OpM^m_s \quad \text{and} \quad | \cdot |_{m, s}^{{\rm Lip}(\gamma)} \leq |\cdot|_{m, s'}^{{\rm Lip}(\gamma)}\,. 
	\end{aligned}
\end{equation}
We now state some 
standard properties of the decay norms 
that are needed for the reducibility scheme 
of Section \ref{ridusezione}. If $a \in H^s$, $s \geq s_0$, then the multiplication operator ${\mathcal M}_a : u \mapsto a u$ satisfies 
\begin{equation}\label{prop multiplication decay}
	{\mathcal M}_a \in\OpM_s^0 \quad \text{and} \quad |{\mathcal M}_a|_{0, s}^{{\rm Lip}(\gamma)} \lesssim \| a \|_s^{{\rm Lip}(\gamma)}\,. 
\end{equation}

\begin{lem}\label{proprieta standard norma decay}
	$(i)$ Let $s \geq s_0$ and ${\mathcal R} \in \OpM^0_s$. 
	If $\| u \|_s^{{\rm Lip}(\gamma)} < \infty$, then 
	$$
	\| {\mathcal R} u \|_s^{{\rm Lip}(\gamma)} \lesssim_{s} |{\mathcal R}|_{0, s_0}^{{\rm Lip}(\gamma)} \| u \|_s^{{\rm Lip}(\gamma)} +  |{\mathcal R}|_{0, s}^{{\rm Lip}(\gamma)} \| u \|_{s_0}^{{\rm Lip}(\gamma)}\,;
	$$
	$(ii)$ Let $s \geq s_0$, $m, m' \in \R$, and let ${\mathcal R} \in\OpM^m_s$, ${\mathcal Q} \in \OpM^{m'}_{s + |m|}$. 
	Then ${\mathcal R} {\mathcal Q} \in \OpM^{m + m'}_s$ and 
	$$
	|{\mathcal R}{\mathcal Q}|_{m + m', s}^{{\rm Lip}(\gamma)} \lesssim_{s, m} |{\mathcal R}|_{m, s}^{{\rm Lip}(\gamma)} |{\mathcal Q}|_{m', s_0 + |m|}^{{\rm Lip}(\gamma)} + |{\mathcal R}|_{m, s_0}^{{\rm Lip}(\gamma)} |{\mathcal Q}|_{m', s + |m|}^{{\rm Lip}(\gamma)}\ \,;
	$$
$(iii)$ Let $s \geq s_0$, $m\geq 0$ and ${\mathcal R} \in \OpM^{-m}_s$. 
Then, for any integer $n \geq 1$, ${\mathcal R}^n \in \OpM^{-m}_s$ and there exist constants $C(s_0,m), C(s, m) > 0$, independent of $n$, such that 
\begin{equation}\label{stima.potenza}
	\begin{aligned}
					& |{\mathcal R}^n|_{-m, s_0}^{{\rm Lip}(\gamma)} \leq C(s_0, m)^{n - 1} \big(|{\mathcal R}|_{-m, s_0}^{{\rm Lip}(\gamma)}\big)^{n} \,, \\
		& |{\mathcal R}^n|_{-m, s}^{{\rm Lip}(\gamma)} \lesssim \, \big(C(s, m)|{\mathcal R}|_{-m, s_0}^{{\rm Lip}(\gamma)}\big)^{n - 1} |{\mathcal R}|_{-m, s}^{{\rm Lip}(\gamma)}\,;
	\end{aligned}
\end{equation}
	$(iv)$ Let $s \geq s_0$, $m \geq 0$ and ${\mathcal R} \in \OpM^{- m}_s$.
%
Then there exists $\delta(s,m) \in (0, 1)$ small enough such that, 
if $|{\mathcal R}|_{-m, s_0}^{{\rm Lip}(\gamma)} \leq \delta(s,m)$,
then  the map $\Phi := {\rm exp}({\mathcal R}) \in \OpM^{0}_s$  is invertible and satisfies the estimates 
$$
|\Phi^{\pm 1}|_{0, s}^{\Lip(\gamma)} \lesssim_s 1 + |{\mathcal R}|_{-m, s}^{\Lip(\gamma)}\,; 
$$

	\noindent
	$(v)$ Let $s \geq s_0$, $m \in \R$ and ${\mathcal R} \in \OpM^m_s$. Let ${\mathcal D}_{\mathcal R}$ be the diagonal operator as in \eqref{diagonal.op.matrix}
	Then ${\mathcal D}_{\mathcal R} \in \OpM^m_s$ and $|{\mathcal D}_{\mathcal R}|_{m, s}^{{\rm Lip}(\gamma)} \lesssim |{\mathcal R}|_{m, s_0}^{{\rm Lip}(\gamma)}$ for any $s \geq s_0$.  
	As a consequence, 
	\[
	| \widehat{\mathcal R}(0)_j^j |^{{\rm Lip}(\gamma)} \lesssim \langle j \rangle^m|{\mathcal R}|_{s_0}^{{\rm Lip}(\gamma)}\,.
	\] 
\end{lem}

\begin{proof}
	Items $(i), (ii)$ are proved in Lemma 2.6 in \cite{FM23}. 
	To prove item $(iii)$ we need the following preliminary estimates:
	\begin{equation}\label{stima.potenza.zero}
		\begin{aligned}
			& |{\mathcal R}^n|_{0, s_0}^{{\rm Lip}(\gamma)} \leq \tC(s_0)^{n - 1} \big(|{\mathcal R}|_{0, s_0}^{{\rm Lip}(\gamma)}\big)^{n} \,, \\
			& |{\mathcal R}^n|_{0, s}^{{\rm Lip}(\gamma)} \leq  \big(\tC(s_0)|{\mathcal R}|_{0, s_0}^{{\rm Lip}(\gamma)}\big)^{n - 1} |{\mathcal R}|_{0, s}^{{\rm Lip}(\gamma)}\,.
		\end{aligned}
	\end{equation}
	One can easily prove them by an induction argument and by using item $(ii)$ (see also Lemma 2.6-$(iii)$ in \cite{FM23}). We now prove \eqref{stima.potenza}.
	 By item $(ii)$, \eqref{stima.potenza.zero} and \eqref{prop elementari}, we compute, for any $n\geq 1$,
	\begin{equation*}
		\begin{aligned}
				|\cR^{n-1} \cR |_{-m, s_0}^{{\rm Lip}(\gamma)} &\lesssim_{m,s_0}
			|\cR^{n-1}|_{0, s_0}^{{\rm Lip}(\gamma)} |\cR|_{-m, s_0}^{{\rm Lip}(\gamma)}  \\ 
			& \lesssim_{m,s_0}   (\tC(s_0))^{n-2}(|\cR|^{{\rm Lip}(\gamma)}_{0,s_0})^{n-1}|\cR|_{-m, s_0}^{{\rm Lip}(\gamma)} \\
			& \leq c(m,s_0) (\tC(s_0))^{n-2}(|\cR|^{{\rm Lip}(\gamma)}_{-m,s_0})^{n}\,, \\
			|\cR^{n-1} \cR |_{-m, s}^{{\rm Lip}(\gamma)} &\lesssim_{m,s}
			|\cR^{n-1}|_{0, s}^{{\rm Lip}(\gamma)} |\cR|_{-m, s_0}^{{\rm Lip}(\gamma)} 
			+ |\cR^{n-1}|_{0, s_0}^{{\rm Lip}(\gamma)} |\cR|_{-m, s}^{{\rm Lip}(\gamma)}
			\\
			& \lesssim_{m,s}	 
			(\tC(s)|\cR|^{{\rm Lip}(\gamma)}_{0,s_0})^{n-2}
			|\cR|^{{\rm Lip}(\gamma)}_{0,s} |\cR|_{-m, s_0}^{{\rm Lip}(\gamma)} \\
			&\quad  \, \,+ (\tC(s_0))^{n-2}(|\cR|^{{\rm Lip}(\gamma)}_{0,s_0})^{n-1}|\cR|_{-m, s}^{{\rm Lip}(\gamma)}
			\\
			&\le c(m,s) \tC(s)^{n-2} (|\cR|^{{\rm Lip}(\gamma)}_{-m,s_0})^{n-1}|\cR|_{-m, s}^{{\rm Lip}(\gamma)} \,,
		\end{aligned}
	\end{equation*}
	and the bound follows provided that $C(m,s_0)\geq \max\{ c(m,s_0), \tC(s_0)\}$ and $C(m,s)\geq \max\{ c(m,s), \tC(s)\}$.
		Item $(iv)$ follows by recalling that $\Phi^{\pm 1} =	{\rm exp}(\pm \cR) =  {\rm Id}+ \sum_{k=1}^{\infty} \frac{{\mathcal R}^{n}}{n!}$
		and by  \eqref{stima.potenza.zero}, \eqref{prop elementari}.
	The claims of item $(v)$ are a direct consequence of the definition of the matrix decay norm in
	Definition \ref{def decay norm}. 
\end{proof}

For $N > 0$, we define the operators $\Pi_N {\mathcal R}$ and $\Pi_N^\perp \cR$ by means of their matrix representation as follows: 
\begin{equation}\label{def proiettore operatori matrici}
	(\widehat{\Pi_N {\mathcal R}})(\ell)_{j}^{j'} := \begin{cases}
		\widehat{\mathcal R}(\ell)_j^{j'} & \text{if } |\ell|, |j - j'| \leq N\,, \\
		0 & \text{otherwise}\,, 
	\end{cases} \qquad   \ \Pi_N^\bot {\mathcal R} := {\mathcal R} - \Pi_N {\mathcal R}\,. 
\end{equation}


\begin{lem}\label{lemma proiettori decadimento}
	For all $s, \alpha \geq 0$, $m \in \R$, one has 
	$|\Pi_N {\mathcal R}|_{m, s + \alpha}^{{\rm Lip}(\gamma)} \leq N^\alpha |{\mathcal R}|_{m, s}^{{\rm Lip}(\gamma)}$ and $|\Pi_N^\bot {\mathcal R}|_{m, s}^{{\rm Lip}(\gamma)} \leq N^{- \alpha} |{\mathcal R}|_{m, s + \alpha}^{{\rm Lip}(\gamma)}$. 
\end{lem}
\begin{proof}
	The claims follow directly from \eqref{def decay norm} and \eqref{def proiettore operatori matrici}.
\end{proof}

We also define the projection $\Pi_0$ on the space of zero average functions as 
\begin{equation}\label{definizione proiettore media spazio tempo}
	\Pi_0 h := \frac{1}{(2 \pi)^{ d}} \int_{\T^{ d}} h(\vphi, x)\, \wrt  x \,,
	\qquad 
	\Pi_0^\bot := {\rm Id} - \Pi_0\,.
\end{equation}
	%
	%
	In particular,
	for any $m, s \geq 0$,
	\begin{equation}\label{stima Pi 0}
		|\Pi_0^\bot|_{0, s} \leq 1\,, \quad |\Pi_0|_{- m, s} \lesssim_{m} 1\,. 
	\end{equation}
	We finally mention the elementary properties of the Laplacian operator $- \Delta$ and its inverse $(- \Delta)^{- 1}$ acting on functions with zero average in $x$:
	\begin{equation}\label{Laplacia.zero}
		- \Delta u(x) = \sum_{\xi \in \Z^d \setminus \{ 0 \}} |\xi|^2 \widehat u(\xi) e^{\im x \cdot \xi} \,, \quad (- \Delta)^{- 1} u(x) = \sum_{\xi \in \Z^d \setminus \{ 0 \}} \frac{1}{|\xi|^2} \widehat u(\xi) e^{\im x \cdot \xi}\,. 
	\end{equation}
	By Definition \ref{block norm}, one easily verifies, for any $s\geq 0$,
	\begin{equation}\label{decay laplace}
		|- \Delta|_{2, s} \leq 1\,, \quad |(- \Delta)^{- 1}|_{- 2, s} \leq 1\,. 
	\end{equation}
	
	\noindent
Finally, recalling the definition of $\mathtt L$ in \eqref{lambda.m.def} we have, for any $s\geq 0$,
\begin{equation}\label{stima.mathttL}
		| \tL |_{-1, s} \leq 1\,. 
	\end{equation}
	
	\subsubsection{Real and reversible operators}\label{Reversible operators}
	
	For any function $u(\vf,x)$, 
	we write $u \in X$ when $u = \even(\vf,x)$, meaning that $u(\vf,x)=u(-\vf,-x)$,
	and $u \in Y$ when $u = \odd(\vf,x)$, meaning that $u(-\vf,-x)=-u(\vf,x)$. 
	
	\begin{defn}\label{reserv.operators.def}
		$(i)$ We say that a linear operator $\Phi$ is \emph{reversible} 
		if $\Phi : X \to Y$ and $\Phi : Y \to X$. 
		We say that $\Phi$ is \emph{reversibility preserving} 
		if $\Phi : X \to X$ and $\Phi : Y \to Y$. 
		
		\noindent
		$(ii)$ We say that an operator $\Phi : L^2(\T^d) \to L^2(\T^d)$ is \emph{real} if $\Phi(u)$ is real valued for any $u$ real valued. 
	\end{defn}
	It is convenient to reformulate real and reversibility properties of linear operators in terms of their matrix representations.
	\begin{lem}\label{lemma real rev matrici}
		A linear operator ${\mathcal R}$ is :
		
		\noindent
		$(i)$ real if and only if 
		$\widehat{\mathcal R}(\ell)_{j}^{j'} = \overline{\widehat{\mathcal R}(-\ell)_{- j}^{- j'}}$ 
		for all $\ell \in \Z^\nu$, $j, j' \in \Z^d$;
		
		\noindent
		$(ii)$ reversible if and only if 
		$\widehat{\mathcal R}(\ell)_j^{j'} = - \widehat{\mathcal R}(-\ell)_{- j}^{- j'}$ 
		for all $\ell \in \Z^\nu$, $j, j' \in \Z^d$;
		
		\noindent
		$(iii)$ reversibility-preserving if and only if 
		$\widehat{\mathcal R}(\ell)_j^{j'} =  \widehat{\mathcal R}(-\ell)_{- j}^{- j'}$ 
		for all $\ell \in \Z^\nu$, $j, j' \in \Z^d$.
	\end{lem}
	
	\subsection{Momentum preserving operators}\label{subsec:momentum}
	
	The following definition is crucial in the construction of traveling waves. 
	
	\begin{defn}	\label{def:mom.pres}
		{\bf (Momentum preserving)}	
		A  $ \vf $-dependent family of linear operators 
		$\cR(\vf) $, $ \vf \in \T^\nu $,  is  {\em momentum preserving} if
		\begin{equation}\label{eq:mp_A_tw}
			\cR(\vf - \pi(\vs) )  \circ \tau_\vs = \tau_\vs \circ \cR(\vf ) \,  , \quad 
			\forall \,\vf \in \T^\nu \, , \ \vs \in \R^d \,  ,
		\end{equation}
		where the translation operator $\tau_\vs$ is defined in \eqref{def:vec.tau}.
	\end{defn}
	
	In particular, we mention that the operators $-\Delta$ and $(-\Delta)^{-1}$ defined in \eqref{Laplacia.zero} are momentum preserving.
	
	Momentum preserving operators are closed under several operations.
	\begin{lem}\label{lem:mom_prop}
		Let $\cR(\vf), \cQ(\vf)$ be momentum preserving operators. Then:
		\begin{itemize}
			\item[(i)] {\rm (Composition)}: $\cR (\vf) \circ \cQ (\vf) $ is a momentum preserving operator;
			\item[(ii)] {\rm (Inversion)}: If $\cR(\vf)$ is invertible then $\cR(\vf)^{-1}$ is momentum preserving;
			\item[(iii)] {\rm (Flow)}: Assume that 
			\begin{equation}\label{same-CP}
				\partial_{t} \Phi^t (\vf) = \cR (\vf) \Phi^t (\vf) \, , \quad 
				\Phi^0 (\vf) = {\rm Id} \, , 
			\end{equation}
			has a unique  propagator $\Phi^t (\vf) $ for any $ t\in[0,1] $. 
			Then $\Phi^t ( \vf ) $ is  momentum preserving.
		\end{itemize}
	\end{lem}
	\begin{proof}
		Item $(i)$ follows directly by \eqref{eq:mp_A_tw}. Item $(ii)$ follows by taking the inverse,
		of \eqref{eq:mp_A_tw} and using that
		$\tau_{-\vs} = \tau_\vs^{-1} $.  
		Finally,  item $(iii)$ holds because  
		$ \tau_\vs^{-1} \Phi^t ( \vf - \pi(\vs)) \tau_\vs $ solves the same Cauchy 
		problem in 
		\eqref{same-CP}. 
	\end{proof}
	
%

	\begin{lem}
		\label{A.mom.cons}
		Let $\cR(\vf)$ be a momentum preserving linear operator
		and  $u$ a quasi-periodic traveling wave, according to Definition \ref{def2.QPT}.
		Then $\cR(\vf) u $ is a quasi-periodic traveling wave.
	\end{lem}
	
	\begin{proof}
		It follows by Definition \ref{def:mom.pres} and  by the characterization of traveling waves in \eqref{condtraembedd}.
	\end{proof}

%
	
	We now provide a characterization of 
	the momentum preserving property in Fourier space. 
	In particular, momentum preserving operators that are independent of $\vf\in\T^\nu$ are actually diagonal.
	
	\begin{lem}
		\label{lem:mom_pres}
		A $ \vf $-dependent family of operators $ \cR(\vf) $, with $ \vf \in \T^\nu $, 
		is momentum preserving  if and only if 
		the matrix elements of $\cR(\vf)$, defined by \eqref{amatriciana},  fulfill
		\begin{equation}\label{momentum}
			\wh\cR(\ell)_{j}^{j'} \neq 0 \quad \Rightarrow  \quad \pi^\top( \ell) + j-j' = 0 \, , \quad
			\forall\, \ell \in\Z^\nu , \ \ j,j'\in\Z^d \, . 
		\end{equation}
		As a consequence we have that, for any momentum preserving operator $\cR(\vf)$, the operator $\widehat{\mathcal R}(0)$ satisfies 
		\begin{equation}\label{diagonal mom pres}
			\wh\cR(0)_{j}^{j'} \neq 0 \quad  \Rightarrow \quad j=j' \,,
		\end{equation}
		that is, $\widehat\cR(0)$ is a time-independent diagonal operator (recall \eqref{diagonal.op.matrix}).
	\end{lem}
	\begin{proof}
		By \eqref{amatriciana} we have, for any function $ u (x)$, 
		$$
		\tau_\vs ( \cR (\vf) u )
		= \sum_{j,j'\in\Z^d}\sum_{\ell\in\Z^\nu} 
		\wh\cR(\ell)_j^{j'} e^{\im j \cdot\vs} u_{j'} e^{\im (\ell\cdot \vf + j\cdot x )} 
		$$
		and 
		$$
		\cR (\vf -\pi(\vs))  [\tau_\vs u] 
		= \sum_{j,j'\in\Z^d}\sum_{\ell \in\Z^\nu} 
		\wh\cR(\ell)_j^{j'} e^{- \im \ell \cdot \pi(\vs)} e^{\im j' \cdot \vs} u_{j'} 
		e^{\im (\ell\cdot \vf + j\cdot x )}  \, .
		$$
		Therefore, using that $\ell\cdot \pi(\vs) = \pi^\top(\ell)\cdot \vs$, we obtain that \eqref{eq:mp_A_tw} is equivalent to \eqref{momentum}. The claim in \eqref{diagonal mom pres} follows immediately by \eqref{momentum}: indeed, assuming that the matrix elements $\big( \wh\cR(0)_{j}^{j'}\big)_{j,j'\in\Z^d, \, \ell \in \Z^\nu}$ of $\widehat \cR (0)$ are zero outside the diagonal, by using the momentum condition $ \pi^\top( \ell) + j-j' = 0 $, for $\ell = 0$.
	\end{proof}

	\subsection{Conjugations with change of variables}
	
	Here we recall some results from \cite{BM20} regarding certain properties of linear operators with matrix decay that are conjugated by a change of variables
	\begin{equation}\label{diffeo.A}
		\cA=\cA(\vf) : u \mapsto \cA(\vf) u\,, \quad (\cA (\vf)u)(x) := u(x+\balpha(\vf,x))\,,
	\end{equation}
	where $\balpha : \T^\nu \times \T^d \to \R^d$. We have the following lemmata.
	\begin{lem}\label{prop.base.diffeo}
		Let $\|  \balpha\|_{s_0}^{\Lip(\gamma)}\leq \delta (s_0)$ sufficiently small. The composition operator $\cA$ in \eqref{diffeo.A} satisfies the tame estimates, for any $s\geq s_0$,
		\begin{equation}
			\| \cA u \|_{s}^{\Lip(\gamma)} \lesssim_{s} \| u\|_{s}^{\Lip(\gamma)} + \| \balpha\|_{s}^{\Lip(\gamma)} \| u \|_{s_0}^{\Lip(\gamma)} \,.
		\end{equation}
		Moreover, it is an invertible map, with inverse given by
		\begin{equation}\label{diffeo.A-1}
			\cA^{-1}= \cA(\vf)^{-1} : u \mapsto \cA(\vf)^{-1} u \,, \quad (\cA(\vf)^{-1} u)(y) := u(y+\breve{\balpha}(\vf,y))\,,
		\end{equation}
		for some function $\breve{\balpha}(\vf,y)$ such that $x:=y+\breve{\balpha}(\vf,y)$ is the inverse diffeomorphism of the map $y:=x+\balpha(\vf,x)$, with estimates, for any $s\geq s_0$,
		\begin{equation}
			\begin{aligned}
				& \| \breve{\balpha}\|_{s}^{\Lip(\gamma)}\lesssim_{s} \| \balpha\|_{s}^{\Lip(\gamma)} \,, \quad 	\| \cA^{-1} u \|_{s}^{\Lip(\gamma)} \lesssim_{s} \| u\|_{s}^{\Lip(\gamma)} + \| \balpha\|_{s}^{\Lip(\gamma)} \| u \|_{s_0}^{\Lip(\gamma)} \,, \\
				& \| (\cA^{\pm 1} -{\rm Id})u \|_{s}^{\Lip(\gamma)} \lesssim_{s}  \| \balpha \|_{s_0+1}^{\Lip(\gamma)} \| u\|_{s+1}^{\Lip(\gamma)} + \| \balpha\|_{s+1}^{\Lip(\gamma)} \| u \|_{s_0+1}^{\Lip(\gamma)} \,.
			\end{aligned}
		\end{equation}
	\end{lem}
	\begin{proof}
		We refer to Lemma 2.3-$(ii)$, $(iii)$ in \cite{BM20} for the proof.
	\end{proof}

	\begin{lem}\label{lem:dMP}
		If $ \balpha (\vf, x ) $ is a quasi-periodic traveling wave,  then 
		the operators $ \cA(\vf)$ and $\cA(\vf)^{-1}$ defined in \eqref{diffeo.A}-\eqref{diffeo.A-1} are momentum preserving.
	\end{lem}
	
	\begin{proof}
		We have 
		\begin{equation}
			\begin{aligned}
				\cA (\vf - \pi(\vs)) [\tau_\vs  u] & = 
				u(x+ \balpha (\vf - \pi(\vs),x) + \vs) \\
				&= u(x+ \vs + \balpha(\vf,x+ \vs)) =
				\tau_\vs \big( \cA (\vf )   u\big) \,.
			\end{aligned}
		\end{equation}
		This proves that $\cA(\vf)$ is momentum preserving. By Lemma \ref{lem:mom_prop}-$(ii)$, we have that $\cA(\vf)^{-1}$ is momentum preserving as well.
	\end{proof}

	\begin{lem}{\bf(Conjugation of the Laplacian).}\label{lemma.diffeo.Laplaciano}
		Let $S>s_0$ and $\balpha\in H^{S+2}(\T^{\nu+d})$, with $\| \balpha\|_{s_0+\mu}^{\Lip(\gamma)}\leq \delta $ for some $\delta=\delta(d,\nu)\in (0,1)$ small enough and for some $\mu=\mu(d,\nu)>0$. The following hold:
		\\[1mm]
		\noindent $(i)$
		The operator $-\Delta$ acting on functions with zero average in $x$, as defined in \eqref{Laplacia.zero}, is conjugated via \eqref{diffeo.A}-\eqref{diffeo.A-1} to the operator $\cP_{-\Delta}$ acting on functions with zero average in $x$ defined as
		\begin{equation}\label{Delta.conjugaged}
			\cP_{-\Delta} := \Pi_0^\perp \cA^{-1} (-\Delta) \cA \Pi_0^\perp = -\Delta + \cP_{2}\,,
		\end{equation}
		where $\cP_{2} \in \OpM^2_s$ satisfies, for any $s_0\leq s\leq S$,
		\begin{equation}\label{stima.P2}
			|\cP_{2} |_{2,s}^{\Lip(\gamma)}  \lesssim_{s} \|\balpha \|_{s+\mu}^{\Lip(\gamma)}\,.
		\end{equation}
		Furthermore, if $\balpha(\vf,x)$ is a quasi-periodic traveling wave, then $\cP_{-\Delta}$ and $\cP_{2}$ are momentum preserving, according to Definition \ref{def:mom.pres};
		\\[1mm]
		\noindent $(ii)$ The operator $\cP_{-\Delta}$ in \eqref{Delta.conjugaged} is invertible on the functions of zero average in $x$, with inverse given by\begin{equation}\label{Delta.inverse.conjugaged}
			\cP_{-\Delta}^{-1} := \Pi_0^\perp  \cA^{-1} (-\Delta)^{-1} \cA \Pi_0^\perp =  (-\Delta)^{-1}+ \cP_{-2}\,,
		\end{equation}
		where the operator $(-\Delta)^{-1}$ acting on functions with zero average in $x$ is defined in \eqref{Laplacia.zero}, whereas $\cP_{-2} \in \OpM^{-2}$ satisfies, for any $s_0\leq s\leq S$,
		\begin{equation}\label{stima.P-2}
			|\cP_{-2} |_{-2,s}^{\Lip(\gamma)} \lesssim_{s} \|\balpha \|_{s+\mu}^{\Lip(\gamma)}\,.
		\end{equation}
		Furthermore, if $\balpha(\vf,x)$ is a quasi-periodic traveling wave, then $\cP_{-\Delta}^{-1}$ and $\cP_{-2}$ are momentum preserving, according to Definition \ref{def:mom.pres}.
	\end{lem}
	
	\begin{proof}
		We start with item $(i)$. We compute explicitly
		\begin{equation}
			\begin{aligned}
				\cA^{-1} (-\Delta) \cA = -\Delta +  \sum_{i,j=1}^{d} a_{ij}(\vf,y) \pa_{y_i} \pa_{y_j} + \sum_{j=1}^{d} b_{j} (\vf,y) \pa_{y_j}  \,,
			\end{aligned}
		\end{equation}
		where, for $i,j=1,...,d$,
		\begin{equation}\label{coeff.explicit.high.dim}
			\begin{aligned}
				a_{jj}(\vf,y) & := - \cA^{-1} \big\{ | \nabla(x_j + \alpha_{j}(\vf,x)) |^2  - 1 \big\}\,, \quad i= j \,, \\
				a_{i,j}(\vf,y) & := - 2  \cA^{-1}\big\{ \nabla(x_i + \alpha_{i}(\vf,x)) \cdot \nabla(x_j + \alpha_{i}(\vf,x)) \big\} \,, \quad i\neq j \,,\\
				b_{j}(\vf,y) & := - \cA^{-1} \big\{ \Delta \alpha_{j}(\vf,x) \big\}\,, \quad \balpha(\vf,x) = (\alpha_{1}(\vf,x),...,\alpha_{d}(\vf,x))\,.
			\end{aligned}
		\end{equation}
		We conclude that the operator $\cP_{2}$ in \eqref{Delta.conjugaged} is given by
		\begin{equation}
			\cP_{2} := \Pi_0^\perp \Big( \sum_{i,j=1}^{d} a_{ij}(\vf,y) \pa_{y_i} \pa_{y_j} + \sum_{j=1}^{d} b_{j} (\vf,y) \pa_{y_j} \Big) \,,
		\end{equation}
		and the estimate \eqref{stima.P2} follows by the explicit formulae of the coefficients in \eqref{coeff.explicit.high.dim}, Lemma \ref{proprieta standard norma decay}-$(ii)$ and 	\eqref{stima Pi 0}. Moreover, if $\balpha(\vf,x)$ is a traveling wave, and since the operators $-\Delta$ and $\Pi_0^\perp$ are momentum preserving, it follows from Lemma 	  \ref{lem:dMP} that the operators $\cP_{-\Delta}$ and $\cP_{2}$ are momentum preserving as well.
		\\[1mm]
		We now prove item $(ii)$. First, note that
		\begin{equation}
			\begin{aligned}
				\cP_{-\Delta} \cP_{-\Delta}^{-1} & = \Pi_0^\perp  \cA^{-1} (-\Delta) \cA  \Pi_0^\perp  \cA^{-1} (-\Delta)^{-1} \cA \Pi_0^\perp \\
				& = \Pi_0^\perp -  \Pi_0^\perp  \cA^{-1} (-\Delta) \cA  \Pi_0  \cA^{-1} (-\Delta)^{-1} \cA \Pi_0^\perp \\
				& = \Pi_0^\perp -  \Pi_0^\perp  \cA^{-1} (-\Delta)  \Pi_0  \cA^{-1} (-\Delta)^{-1} \cA \Pi_0^\perp = \Pi_0^\perp \,,
			\end{aligned}
		\end{equation}
		since $\cA\Pi_0 = \Pi_0$ and  $(-\Delta)\Pi_0 = 0$, so that $\cP_{-\Delta}^{-1}$ defined in \eqref{Delta.inverse.conjugaged} is indeed the inverse of $\cP_{-\Delta}$ on the space of function with zero average with respect to $x$. Next, we write
		\begin{equation}
			\cP_{-\Delta} = -\Delta + \cP_{2} = (-\Delta) \big( {\rm Id} + F \big)\,, \quad F := (-\Delta)^{-1} \cP_{2}\,,
		\end{equation}
		where, by Lemma \ref{proprieta standard norma decay}-$(ii)$ and \eqref{decay laplace}, we have $F\in \OpM_{s}^0$ and, for any $s_0 \leq s \leq S$,
		\begin{equation}
			| F|_{0,s}^{\Lip(\gamma)} \lesssim_{s} | \cP_{2}|_{2,s+2}^{\Lip(\gamma)} \lesssim_{s} \| \balpha\|_{s+\mu}^{\Lip(\gamma)}\,.
		\end{equation}
		Then, for $\| \balpha\|_{s_0 +\mu}^{\Lip(\gamma)}$ sufficiently small, the operator ${\rm Id}+ F \in \OpM_{s}^0$ is invertible by a standard Neumann series argument, with inverse $({\rm Id}+ F)^{-1} \in \OpM_{s}^0$. We obtain that
		\begin{equation}
			\cP_{-\Delta}^{-1} = ({\rm Id} + F)^{-1} (-\Delta)^{-1} = (-\Delta)^{-1} + \cP_{-2} \,,
		\end{equation}
		where, by Lemma \ref{proprieta standard norma decay} and \eqref{decay laplace},
		\begin{equation}
			\cP_{-2} := \big[ (\rm Id + F)^{-1} - {\rm Id} \big] (-\Delta)^{-1} \in \OpM_{s}^{-2}\,,
		\end{equation}
		and, for any $s_0\leq s \leq S$, satisfies the estimate \eqref{stima.P-2}. Moreover, if $\balpha(\vf,x)$ is a traveling wave, then $\cP_{-\Delta}^{-1}$ and $\cP_{-2}$ are momentum preserving by item $(i)$, Lemma \ref{lem:mom_prop} and  the fact that the operator $(-\Delta)^{-1}$ is momentum preserving.
	\end{proof}

	\section{Construction of the approximate solution}
	
	The first step in the search for a solution of \eqref{QP.Eq.to.solve} is to find
	an approximate solution of size $O(\lambda^{\theta})$, with $\alpha - 1 < \theta < 1$ to be determined. By rescaling the unknown $v \rightsquigarrow \lambda^{\theta} v $ in \eqref{QP.Eq.to.solve}, we look for quasi-periodic traveling wave solutions $v(\vf,x)$ of size $O(1)$ as zeroes of the nonlinear map ${\mathcal F}(v) \equiv{\mathcal F}(v; \omega, \lambda)$ of the form
	\begin{equation}\label{internal.rescaled}
		\begin{aligned}
			&  {\mathcal F}(v) := \lambda \,\omega \cdot \pa_{\vf} v + \beta \,\tL v +  \lambda^{\theta} {\mathcal N}(v,v) - \lambda^{\alpha-\theta} f(\vf,x) =0\,, \\
			&   {\mathcal N}(v_1,v_2) : = \fB[v_1] \cdot \nabla v_2\,.
		\end{aligned}
	\end{equation}
	If we want to construct \emph{reversible} solutions,  we require the forcing $f(\vf,x)$ to satisfy the symmetry condition
	\begin{equation}\label{force.even}
		f(\vf,x) = \even(\vf,x)\,,
	\end{equation}
	where $ \even(\vf,x)$ denotes a function which is even in both variables $\vf$ and $x$.
	Given $v(\vphi, x)$ a quasi-periodic traveling wave, we denote by ${\mathcal L} \equiv {\mathcal L}(v) := \di_{v}{\mathcal F}(v)$ the linearized operator near $v$:	\begin{equation}\label{prima forma linearizzato beta plane}
		\begin{aligned}
			& {\mathcal L}  = \lambda\, \omega \cdot \partial_\vphi + \beta \, \tL+ \lambda^{\theta} \big( {\bf a}_0(\vphi, x) \cdot \nabla + {\mathcal E}_0 \big)\,, \\
			&  {\bf a}_0(\vphi, x) : =    \fB\big[ v \big] (\vphi, x) \,, \quad \cE_{0}[h] :=  \nabla  v \cdot \fB [h] \,,
		\end{aligned}
	\end{equation}
	where $\tL$ is defined in \eqref{lambda.m.def} and the Biot-Savart operator $  \fB:= \nabla^\perp (-\Delta)^{-1} $ is as in \eqref{biot-savart}.
	
	From now on, the dimension $d=2$ is fixed once and for all, so that all the definitions and properties in Section \ref{sez:functional} will be applied with $d=2$.
	
	First we establish some quantitative estimates on ${\mathcal F}$ that we shall use in the sequel and its behaviour when acting on quasi-periodic traveling waves.
	\begin{lem}\label{stime tame cal F}
		Let $s_0>0 $ be as in \eqref{definizione s0}.
	The following hold:
		\\[1mm]
		\noindent $(i)$ Assume that $s > s_0$, $v(\,\cdot\,; \omega) \in H^{s + 1}_0(\T^{\nu + 2})$ and $\| v \|_{s_0 + 1}^{\Lip(\gamma)} \lesssim 1$. Then the operator $\mathcal F(v)$ in \eqref{internal.rescaled} satisfies the following estimate:
		\begin{equation}\label{stima cal F w}
			\| {\mathcal F}(v) \|_s^{\Lip(\gamma)} \lesssim_s \lambda \big(1 + \| v \|_{s + 1}^{\Lip(\gamma)} \big)\,. 
		\end{equation}
		Moreover, if $v$ is a quasi-periodic traveling wave, then $\cF(v)$ is a quasi-periodic traveling wave. In addition, assuming \eqref{force.even}, if $v(\vf,x)=\odd(\vf,x)$, then $\cF(v)(\vf,x)=\even(\vf,x)$;
		\\[1mm]
		\noindent
		$(ii)$ Assume that $\| v \|_{s_0 + 1}^{\Lip(\gamma)} \lesssim 1$. Then, for any $b\geq 1$, $h(\,\cdot\,;\omega) \in H^{s_0 + b}_0(\T^{\nu + 2})$ and $\tK\in\N$,
		\begin{equation}\label{stime.per.mozart}
			\begin{aligned}
					& \| {\mathcal L} h \|_{s_0}^{\Lip(\gamma)} \lesssim \lambda \| h \|_{s_0 + 1}^{\Lip(\gamma)}\,, \\
					& \| [\cL,\Pi_{\tK}^\perp] h \|_{s_0}^{\Lip(\gamma)} \lesssim \lambda^{\theta} \tK^{1-b}\big( \| h \|_{s_0+b}^{\Lip(\gamma)} +  \| v \|_{s_0+b}^{\Lip(\gamma)}   \| h \|_{s_0+1}^{\Lip(\gamma)}   \big)  \,, \\
					& \|  [\Pi_{\tK},\cL] h \|_{s_0+b}^{\Lip(\gamma)} \lesssim \lambda^{\theta} \tK \big(  \| h \|_{s_0+b}^{\Lip(\gamma)}  +  \| v \|_{s_0+b+1}^{\Lip(\gamma)}   \| h \|_{s_0+1}^{\Lip(\gamma)}  \big) \,.
			\end{aligned}
		\end{equation}
		Moreover, if $h$ is a quasi-periodic traveling wave, then $\cL h$ is a quasi-periodic traveling wave. In addition,  if $h(\vf,x)=\odd(\vf,x)$, then $\cL h(\vf,x)=\even(\vf,x)$;
		\\[1mm]
		\noindent
		$(iii)$ Let $v(\,\cdot\,;\omega), \,h(\,\cdot\,;\omega) \in H_{0}^{s + 1}(\T^{\nu + 2})$. Then 
		$$
		Q (v, h) := {\mathcal F}(v + h) - {\mathcal F}(v) - {\mathcal L} h = \lambda^\theta \cN(h,h)
		$$
		satisfies
		\begin{equation}\label{stima Q v h}
			\int_{\T^2} Q(v,h)(\vf,x) \wrt x = 0 \,, \quad  
			  \| Q(v, h) \|_{s}^{\Lip(\gamma)} \lesssim \lambda^{\theta} \| h \|_{s_0 + 1}^{\Lip(\gamma)} \| h \|_{s + 1}^{\Lip(\gamma)}\ \,. 
		\end{equation}
			Moreover, if $v$  and $h$ are quasi-periodic traveling waves, then $Q(v,h)$ is a quasi-periodic traveling wave. In addition, assuming \eqref{force.even}, if $v(\vf,x),h(\vf,x)=\odd(\vf,x)$, then $Q(v,h)(\vf,x)=\even(\vf,x)$.
	\end{lem}
	\begin{proof}
		The claimed estimates follow by \eqref{prima forma linearizzato beta plane}, using that $\alpha-\theta<1$, $Q(v, h) = \lambda^{\theta}{\mathcal N}(h, h)$, by a repeated applications of the interpolation estimate \eqref{lemma:LS norms}, and by the fact that $\Pi_{\tK}$ and  $\Pi_{\tK}^\perp$ commute with $\lambda\,\omega\cdot\pa_{\vf} + \beta\,\tL$.
		Moreover, by \eqref{internal.rescaled} and using that ${\rm div} \fB [h]= 0$ for any sufficiently smooth function $h(\vf , x)$, by \eqref{biot-savart}, we have
		 \begin{equation}
		 	\int_{\T^2} Q(v,h)(\vf,x) \wrt x = \lambda^{\theta} \int_{\T^2} \fB[h(\vf,x)] \cdot \nabla h(\vf,x) \wrt x = - \lambda^\theta \int_{\T^2} {\rm div} \fB[h(\vf,x)] h(\vf,x) \wrt x = 0 \,.
		 \end{equation} 
		The claims on the quasi-periodic traveling waves follow by \eqref{internal.rescaled}, \eqref{prima forma linearizzato beta plane}, the fact that $f(\vf,x)$ is a quasi-periodic function, and Lemmata \ref{lem:mom_prop}, \ref{A.mom.cons}. Finally, assuming \eqref{force.even}, the claims on the parities follow by \eqref{internal.rescaled}, \eqref{lambda.m.def}, \eqref{biot-savart}, \eqref{prima forma linearizzato beta plane} and Lemma \ref{lemma real rev matrici}.
	\end{proof}
	
	For $N\in\N_0$, we construct an approximate solution of \eqref{internal.rescaled} of the form 
	\begin{equation}\label{approx.sum}
		\begin{aligned}
			v_{{\rm app},N}(\vf,x)& = v_0(\vf,x)+v_1(\vf,x)+...+v_N(\vf,x) \\
			& = v_{{\rm app},N-1}(\vf,x) + v_{N}(\vf,x)\,, \quad v_{{\rm app},-1}:= 0 \,,
		\end{aligned}
	\end{equation}
	in such a way that the function $v_{{\rm app}, N}(\vf,x)$ of size $O(1)$ solves \eqref{internal.rescaled} up to a progressively smaller error term which, for $N\in \N$ sufficiently large, becomes perturbatively small for $\lambda\gg 1$ big enough, see \eqref{stima.approx.qN} and \eqref{condi.per.NM.dopo} in Proposition \ref{lemma.approx}. 
	Moreover, both the approximate solution $v_{{\rm app},N}(\vf,x)$ that we are going to construct and the final error term are quasi-periodic traveling waves.

	For fixed $\gamma\in (0,1)$ and $\tau > \nu-1$, recalling \eqref{anello}, we define the set of the Diophantine frequency vectors
	\begin{equation}\label{DC.2gamma}
		\tD\tC(\gamma,\tau) := \Big\{  \omega\in \Omega :   \ |\omega\cdot \ell | \geq \gamma \braket{\ell}^{-\tau} \ \forall\,\ell\in\Z^{\nu}\setminus\{0\}  \Big\}  \,.
	\end{equation}
	It is well known that the Lebesgue measure $|\Omega \setminus \tD\tC(\gamma,\tau) | = O(\gamma)$. Moreover, we consider the following non-resonance condition, for $\tau > \nu + 1$,
	\begin{equation}\label{Omega.gamma.SA}
	\begin{aligned}
			\Omega_{\gamma} := \Big\{  \omega & \in \tD\tC(\gamma,\tau) \,: \, | \lambda \,\omega\cdot \ell + \beta \,\tL(j)  |  \geq \lambda \tfrac{\gamma}{ \braket{\ell}^{\tau}}\,, 
			  \forall\,(\ell,j)\in\Z^{\nu+2}\setminus\{0\}\,, \  \pi^\top(\ell) + j=0   \Big\}\,.
	\end{aligned}
	\end{equation} 
	At each step of the iteration, we solve  linear equations of the following form.
	
	\noindent

	\begin{lem}\label{lemma.equa.linear.approx}
		Let $s\geq 0$, $\gamma\in (0,1)$ and $\tau > \nu + 1$, Let $g (\,\cdot\,;\omega,\lambda) \in H_0^{s+2\tau+1}(\T^{\nu+2})$ be a quasi-periodic traveling wave.  Then, for any $\omega\in \Omega_{\gamma}$ as in \eqref{Omega.gamma.SA}, the linear equation
		\begin{equation}\label{equa.linear.approx}
			\lambda \, \omega\cdot \pa_{\vf} w(\vf,x) + \beta\,\tL w(\vf,x) + g(\vf,x) = 0
		\end{equation}
		is solved by the quasi-periodic traveling wave $w(\,\cdot\,;\omega,\lambda) \in H_0^s(\T^{\nu+2})$, defined as
		\begin{equation}\label{solu.linear.approx}
			\begin{aligned}
				w(\vf,x) &\, = w(\vf,x;\omega,\lambda)  := - (\lambda \, \omega \cdot \pa_{\vf} + \beta\,\tL )^{-1} g(\vf,x) \\
				&   := -  \sum_{(\ell,j)\in \Z^{\nu+2}, j\neq 0 \atop\pi^\top(\ell)+j =0} \frac{1}{\im \big( \lambda\,\omega\cdot\ell + \beta \,\tL (j)\big)} \wh{g}(\ell,j) e^{\im(\ell\cdot\vf+j\cdot x)}\,,
			\end{aligned}
		\end{equation}
		with estimates
		\begin{equation}\label{stima.linear.approx}
		\begin{aligned}
 & 	\| w \|_{s}^{\Lip(\gamma)} \lesssim_{s} \lambda^{-1} \gamma^{-1} \| g \|_{s+2\tau+1}^{\Lip(\gamma)} \,, \\
 & \inf_{\omega \in \Omega_\gamma} \| w(\cdot; \omega) \|_s \geq \tfrac12 \min\{\lambda^{-1}, |\beta|^{-1}\} \inf_{\omega \in \Omega_\gamma}\| g(\cdot; \omega) \|_{s-1}  \,. 
\end{aligned}
\end{equation}
		Furthermore, we have the measure estimate $|\tD\tC(\gamma,\tau) \setminus\Omega_{\gamma}|\lesssim \gamma $, where $\tD\tC(\gamma, \tau)$ is defined in \eqref{DC.2gamma}.\\
		In addition, if $g(\vf,x)=\even(\vf,x)$, then $w(\vf,x)=\odd(\vf,x)$.
	\end{lem}
	\begin{proof}
		We write $w(\vf,x) = \sum_{(\ell, j) \in \Z^{\nu+2}, \, j\neq 0} \whw(\ell,j) e^{\im(\ell\cdot \vf + j\cdot x)}$. Then, expanding the equation \eqref{equa.linear.approx} in Fourier, we get
		\begin{equation}
			\im\,\big( \lambda\,\omega\cdot \ell + \beta\, \tL(j) \big) \whw(\ell,j) + \whg(\ell,j)= 0\,.
		\end{equation}
		Since $g(\vf,x)$ is a quasi-periodic traveling wave of the form \eqref{QPT.forma}, we obtain
		\begin{equation}
			\whw(\ell,j):= \begin{cases}
				\im\,\big( \lambda\,\omega\cdot \ell + \beta\, \tL(j) \big)^{-1} \whg(\ell,j) & \text{if } \ \pi^\top(\ell) + j = 0 \,, \\
				0 & \text{otherwise}\,. 
			\end{cases}
		\end{equation}
		We conclude that $w(\vf,x)$ in \eqref{solu.linear.approx} is indeed a solution of \eqref{equa.linear.approx} and it is a quasi-periodic traveling wave. 
		The upper bound on $\|w \|_{s}^{\Lip(\gamma)}$ in \eqref{stima.linear.approx} follows by Definition \ref{def:Lip F uniform} and since $\omega\in\Omega_{\gamma}$ as in \eqref{Omega.gamma.SA}, whereas the lower bound follows by 
		 the upper bound on the denominator, having $|\omega|\leq 2$ for any $\omega\in\Omega_{\gamma} \subset \tD\tC(\gamma,\tau)$
		(see \eqref{DC.2gamma}-\eqref{Omega.gamma.SA}),
		\begin{equation}
			 |\lambda\,\omega\cdot \ell +\beta \tL(j) | \leq \lambda | \omega| |\ell| + |\beta| \leq 2 \max\{\lambda,|\beta|\} |\ell| \leq 2 \max\{\lambda,|\beta|\} \braket{\ell,j}\,.
		\end{equation}
		We now prove the measure estimate. For $(\ell,j)\in \Z^{\nu+2}\setminus\{0\}$, with $\pi^\top(\ell) +j=0$, we define the resonant set
		\begin{equation}
			R_{\ell,j} := \Big\{  \omega \in \tD\tC(\gamma,\tau) \, : \, | \lambda \,\omega\cdot \ell + \beta\,\tL(j) | < \lambda \gamma \braket{\ell}^{-\tau} \Big\} \,.
		\end{equation}
		Note that $\ell \neq0$. Indeed if $\ell = 0$, then also $j = 0$ (by the relation $\pi^\top(\ell) +j=0$) but this is not possible since $(\ell, j ) \neq (0, 0)$. 
Hence, let 
$$
\omega = \frac{\ell}{|\ell|} s + v, \quad v \cdot \ell = 0
$$
and set 
$$
z(s):=\lambda\,\omega\cdot\ell + \beta\,\tL(j) = \lambda |\ell| s + \beta \tL(j)\,. 
$$ Since $|\partial_s z(s)| \geq \lambda |\ell|$, one gets that 
$$
\Big| \Big\{ s :  \omega = \frac{\ell}{|\ell|} s + v \in R_{\ell, j} \Big\} \Big|  \lesssim \lambda \gamma \langle \ell \rangle^{- (\tau + 1)}
$$
and hence by Fubini, one obtains that $|R_{\ell, j}| \lesssim \lambda \gamma \langle \ell \rangle^{- (\tau + 1)}$. We then compute
		\begin{equation}
			\begin{aligned}
				| \tD\tC(\gamma,\tau) \setminus \Omega_{\gamma} | & =  \Big|  \bigcup_{(\ell,j)\in\Z^{\nu+2}\setminus\{0\} \atop \pi^\top(\ell)+j=0}  R_{\ell,j} \Big|  =  \Big|  \bigcup_{\begin{subarray}{c}
				\ell\in\Z^\nu\setminus\{0\} \\
				|j| \lesssim |\ell|
				\end{subarray}}  R_{\ell,j } \Big| \\
				& \lesssim \sum_{\begin{subarray}{c}
				\ell \neq 0 \\
				|j| \lesssim |\ell|
				\end{subarray}} \frac{\lambda \gamma}{\langle \ell \rangle^{\tau + 1}} \lesssim \sum_{\ell \neq 0} \frac{\lambda \gamma |\ell|^2}{\langle \ell \rangle^{\tau + 1}} \lesssim \sum_{\ell \neq 0} \frac{\lambda \gamma}{\langle \ell \rangle^{\tau - 1}}   \,,
			\end{aligned}
		\end{equation}
		where the last inequality holds since $\tau > \nu + 1$. \\
		Finally, if we assume $g(\vf,x)=\even(\vf,x)$, then we have that $w(\vf,x)=\odd(\vf,x)$, using that $\lambda\,\omega\cdot\pa_{\vf}+\beta\,\tL$ is reversible, by \eqref{lambda.m.def}, Lemma \ref{lemma real rev matrici} and Definition \ref{reserv.operators.def}.
		This concludes the proof.
	\end{proof}

			We are now ready to construct explicitly the approximate solution of \eqref{internal.rescaled}.

				\begin{prop}\label{lemma.approx}
				{\bf (Approximate solution).}
				Let $N\in\N_0$. Let $s\geq s_0$, $\gamma\in (0,1)$ and $\tau > \nu + 1$. For any $\omega\in\Omega_{\gamma}$, there exists
				a quasi-periodic traveling wave $v_{{\rm app},N}(\vf,x) = v_{{\rm app},N}(\vf,x ; \omega,\lambda) \in H_0^s(\T^{\nu+2})$ of the form \eqref{approx.sum} satisfying
				\begin{equation}\label{internal.approx}
					\cF(v_{{\rm app},N}(\vf,x)) = q_{N}(\vf,x)\,,
				\end{equation}
				where the remainder $q_{N}(\vf,x) = q_{N}(\vf,x,\omega,\lambda) \in H_0^s(\T^{\nu+2})$ is a quasi-periodic trave\-ling wave as well. Moreover, there exists $\bar{\lambda}=\bar{\lambda}(s,N,\tau,\beta)\gg 1$ large enough such that, defining the constants
				\begin{equation}\label{costanti.kappabetaN}
					\begin{aligned}
						&
						\kappa_{N-\frac12}(\tau) := \kappa_{N}(\tau) -1 
						\,, \quad  \kappa_{N}(\tau) := 2(N+1)(\tau+1) \,,
					\end{aligned}
				\end{equation}
				and assuming, for a given $\tc \in (0,\frac13(2-\alpha))$,
				\begin{equation}\label{parametri.per.vapp}
					\lambda \geq \bar{\lambda}  \geq\max\{ 1 ,|\beta|\}\,, \quad \gamma = \lambda^{-\tc} \,, 
					\quad \theta:= \alpha-1 +\tc \,, 
				\end{equation}
				the following estimates hold, for any $s\geq s_0$:
			\begin{align}
				& \| v_{{\rm app},N} -v_{{\rm app},N-1} \|_s^{\Lip(\gamma)} \lesssim_{s,N} \lambda^{N(\alpha  -2(1-\tc) )} \| f \|_{s+\kappa_{N-\frac12}(\tau)} \,; \label{stima.approx.diffN} \\
				& 	\| v_{{\rm app},N} \|_s^{\Lip(\gamma)} \lesssim_{s,N} \| f \|_{s+\kappa_{N-\frac12}(\tau)}\,, \quad  \inf_{\omega \in \Omega_\gamma} \| v_{{\rm app},N}(\cdot; \omega) \|_s \gtrsim_{s, N} \lambda^{-\tc} \| f \|_{s-1} \label{stima.approx.vN}\,; \\
				& \| q_{N} \|_{s}^{\Lip(\gamma)} \lesssim_{s,N} \lambda^{(N+1)(\alpha-2(1-\tc))+1-\tc}  \| f\|_{s+\kappa_{N}(\tau)}\,.
				\label{stima.approx.qN}
			\end{align}
			In particular, assuming the symmetry condition \eqref{force.even} on the forcing term $f(\vf, x)$ in \eqref{QP.Eq.to.solve}, we have that $v_{{\rm app},N}(\vf,x)= \odd(\vf,x)$ and $q_N(\vf,x)=\even(\vf,x)$. 
			Finally, suppose that
				\begin{equation}\label{condi.per.NM.dopo}
				N >\frac{\alpha-(1-\tc)}{2(1-\tc)-\alpha}\,.
			\end{equation}
			Then, $v_{{\rm app},N}(\vf,x)$ approximately solves \eqref{internal.rescaled}, namely the term on the right-hand side of \eqref{stima.approx.qN} becomes perturbatively small for sufficiently large $\lambda\gg 1$.
			\end{prop}
			
			\begin{proof}
				We argue by induction on $N\in\N_0$. Let $N=0$. Let $v_{{\rm app},0}(\vf,x) = v_0(\vf,x)$ be a solution of the equation
				\begin{equation}\label{equa.v0}
					(\lambda\,\omega\cdot \pa_{\vf} + \beta \, \tL) v_{0}(\vf,x) - \lambda^{\alpha-\theta} f(\vf,x) = 0\,.
				\end{equation}
				By Lemma \ref{lemma.equa.linear.approx}, $v_0(\vf,x)$ is a well-defined quasi-periodic traveling wave, with estimates, recalling \eqref{parametri.per.vapp},
				\begin{equation}\label{stima.v0}
				\begin{aligned}
						&\| v_{0} \|_{s}^{\Lip(\gamma)} \lesssim_{s} \lambda^{-1} \gamma^{-1} \| \lambda^{\alpha-\theta} f \|_{s+2\tau+1} = \gamma^{-1}\lambda^{-\tc} \| f\|_{s+2\tau+1} =  \| f\|_{s+2\tau+1} \,, \\
						&\inf_{\omega \in \Omega_\gamma}\| v_{0}(\cdot; \omega) \|_{s}\geq \frac12  \min\{\lambda^{-1}, |\beta|^{-1} \} \, \| \lambda^{\alpha-\theta} f \|_{s-1} \geq \frac12  \lambda^{-\tc} \| f \|_{s-1}\,,
				\end{aligned}
				\end{equation}
				which proves \eqref{stima.approx.diffN} with $N=0$, as well as \eqref{stima.approx.vN}, having $\kappa_{-\frac12}(\tau)=2\tau+1$. Using \eqref{equa.v0}, we have that \eqref{internal.approx} holds at $N=0$ with $q_0(\vf,x)$ defined by
				\begin{equation}\label{equa.q0}
					q_0(\vf,x) := \lambda^{\theta}\cN(v_0(\vf,x),v_0(\vf,x)) \,.
				\end{equation}
				By \eqref{internal.rescaled}, \eqref{biot-savart}, Lemma \ref{p1-pr}, Lemma \ref{stime tame cal F}-$(iii)$, estimate \eqref{stima.v0} and \eqref{parametri.per.vapp}, 
				 we obtain that \eqref{equa.q0} has zero average in space and satisfies \eqref{stima.approx.qN} with $N=0$, having $\kappa_{0}(\tau)= 2(\tau+1)>0$. The claim that $v_0(\vf,x)=\odd(\vf,x)$ follows by assuming \eqref{force.even} and by  Lemma \ref{lemma.equa.linear.approx}, whereas $q_0(\vf,x)=\even(\vf,x)$ follows by \eqref{equa.q0}, \eqref{internal.rescaled} and Lemma \ref{stime tame cal F}-$(iii)$. \\
				We now assume by induction that \eqref{internal.approx} holds at $N\in\N_0$ with estimates \eqref{stima.approx.diffN}-\eqref{stima.approx.vN} and we prove the claim at $N+1$. We insert \eqref{approx.sum} into \eqref{internal.approx}, both with $N\rightsquigarrow N+1$. Since $v_{{\rm app},N}$ satisfies \eqref{internal.approx}, we look for $v_{N+1}$ satisfying
				\begin{equation}\label{jigsaw1}
					\begin{aligned}
						\big(\lambda\, \omega \cdot \pa_{\vf}  + \beta \, \tL \big)v_{N+1}  &+ \lambda^{\theta}\big(\cN(v_{{\rm app},N} + v_{N+1} ,v_{{\rm app},N} + v_{N+1} ) \\
						& - \cN(v_{{\rm app},N},v_{{\rm app},N}) \big)+q_{N}(\vf,x)  = q_{N+1}(\vf,x)\,,
					\end{aligned}
				\end{equation}
				with $q_{N+1}(\vf,x)$ to be determined. In particular, we choose $v_{N+1}(\vf,x)$ to be a solution to 
				\begin{equation}\label{jigsaw2}
					\big(\lambda\, \omega \cdot \pa_{\vf}  + \beta \, \tL \big)v_{N+1}(\vf,x)+q_{N}(\vf,x) =0\,.
				\end{equation}
				Then, by Lemma \ref{lemma.equa.linear.approx}, estimate \eqref{stima.approx.qN} and \eqref{parametri.per.vapp}, we obtain, for any $s\geq s_0$
				\begin{equation}\label{jigsaw3}
					\begin{aligned}
						\| v_{N+1} \|_{s}^{\Lip(\gamma)} & \lesssim_{s} \lambda^{-1} \gamma^{-1} \| q_{N} \|_{s+2\tau + 1}^{\Lip(\gamma)}  \\
						& \lesssim_{s} \lambda^{-1+\tc}  \lambda^{(N+1)(\alpha-2(1-\tc)) +1-\tc}  \| f \|_{s+ \kappa_{N}(\tau)+2\tau +1}^{\Lip(\gamma)} \,,
					\end{aligned}
				\end{equation}
				which proves \eqref{stima.approx.diffN} at the step $N+1$, setting $\kappa_{N+1-\frac12}(\tau):= \kappa_{N}(\tau)+2\tau+1$. \\
				We now prove \eqref{stima.approx.vN} at the step $N+1$. By \eqref{approx.sum}, estimate \eqref{stima.approx.diffN} at the step $N+1$, the induction assumption on \eqref{stima.approx.vN}, \eqref{parametri.per.vapp} and assuming  $\lambda \gg 1$ large enough, we have
				\begin{equation}
					\begin{aligned}
						\| v_{{\rm app},N+1} \|_{s}^{\Lip(\gamma)} & \leq \| v_{{\rm app},N} \|_{s}^{\Lip(\gamma)} +  \| v_{{\rm app},N+1} - v_{{\rm app},N} \|_{s}^{\Lip(\gamma)} \\
						& \lesssim_{s,N}  \big( 1 + \lambda^{(N+1)(\alpha-2(1-\tc))}  \big) \| f \|_{s+\kappa_{N-\frac12}(\tau)}^{\Lip(\gamma)} \lesssim_{s,N} \| f \|_{s+\kappa_{N-\frac12}(\tau)}^{\Lip(\gamma)}\,, 
					\end{aligned}
				\end{equation}
				which proved the upper bound in \eqref{stima.approx.vN} at the step $N+1$, whereas
				\begin{equation}
					\begin{aligned}
						\inf_{\omega \in \Omega_\gamma}\| v_{{\rm app},N+1}(\cdot; \omega) \|_{s} & \geq \inf_{\omega \in \Omega_\gamma} \| v_{{\rm app},N}(\cdot; \omega) \|_{s} - \| v_{{\rm app},N+1} - v_{{\rm app},N}\|_{s}^{\Lip(\gamma)}  \\
						& \geq C_{s, N}\lambda^{-\tc} \| f \|_{s-1} - K_{s, N} \lambda^{(N+1)(\alpha-2(1-\tc))} \|f\|_{s+\kappa_{N+1-\frac12}} \\
						&   \geq \frac{C_{s, N}}{2} \lambda^{-\tc} \| f \|_{s-1} \,,
					\end{aligned}
				\end{equation}
				(for some positive constants $C_{s, N}, K_{s, N} > 0$), assuming $\lambda \geq \bar\lambda \gg 1$ large enough such that
				\begin{equation}\label{condizione.sanremo}
					 \lambda^{(N+1)(\alpha-2(1-\tc)) + \tc}\leq \frac{C_{s,N}}{2 K_{s,N}}\frac{ \| f \|_{s-1} }{ \| f\|_{s+\kappa_{N+1-\frac12}}  }\,,
				\end{equation}
				which shows the lower bound in \eqref{stima.approx.vN} at the step $N+1$, as claimed. Note that we can impose \eqref{condizione.sanremo} for $\lambda\gg1$ because $\tc <\frac13(2-\alpha)$ implies $2(1-\tc)-\alpha> \tc > \frac{\tc}{N+1}$.
				
				 Furthermore, since $q_{N}$ and $v_{{\rm app},N}$ are quasi-periodic traveling waves by induction assumption, by Lemma \ref{lemma.equa.linear.approx} we have that $v_{N+1}$ and $v_{{\rm app},N+1}$ are quasi-periodic traveling waves as well.  \\
				It remains to define $q_{N+1}(\vf,x)$ and to prove \eqref{stima.approx.qN} at the step $N+1$. By inserting \eqref{jigsaw2} into \eqref{jigsaw1}, the identity holds by defining
				\begin{equation}\label{jigsaw4}
					\begin{aligned}
						q_{N+1}& :=\lambda^{\theta}\big(\cN(v_{{\rm app},N} + v_{N+1} ,v_{{\rm app},N} + v_{N+1} )  - \cN(v_{{\rm app},N},v_{{\rm app},N}) \big) \\
						&\, = \lambda^{\alpha-1+\tc} \big(  \cN(v_{{\rm app},N} , v_{N+1} ) + \cN( v_{N+1} ,v_{{\rm app},N}  ) + \cN( v_{N+1} , v_{N+1} )  \big)\,,
					\end{aligned}
				\end{equation}
				where we used \eqref{parametri.per.vapp} and the bilinearity of $\cN(\,\cdot\, , \,\cdot\,)$. By Lemma \ref{stime tame cal F}-$(iii)$, we deduce that $q_{N+1}(\vf,x)$ has zero average in space. By \eqref{internal.rescaled}, \eqref{biot-savart}, Lemma \ref{lemma:LS norms}, and estimates \eqref{stima.approx.vN}, \eqref{jigsaw3}, 
				 we estimate \eqref{jigsaw4} by
				\begin{equation}
					\begin{aligned}
						\| q_{N+1} \|_{s}^{\Lip(\gamma)} & \lesssim_{s} \lambda^{\alpha-1+\tc} \Big(  \| v_{{\rm app},N} \|_{s-1}^{\Lip(\gamma)}   \| v_{N+1} \|_{s_0+1}^{\Lip(\gamma)}  +   \| v_{{\rm app},N} \|_{s_0-1}^{\Lip(\gamma)}    \| v_{N+1} \|_{s+1}^{\Lip(\gamma)}  \\
						&  \qquad \qquad \quad \  \ \| v_{N+1} \|_{s-1}^{\Lip(\gamma)}   \| v_{{\rm app},N} \|_{s_0+1}^{\Lip(\gamma)} +   \| v_{N+1} \|_{s_0-1}^{\Lip(\gamma)}   \| v_{{\rm app},N} \|_{s+1}^{\Lip(\gamma)} \\
						&  \qquad \qquad \quad \ \  \| v_{N+1} \|_{s-1}^{\Lip(\gamma)}    \| v_{N+1} \|_{s_0+1}^{\Lip(\gamma)}  +   \| v_{N+1} \|_{s_0-1}^{\Lip(\gamma)}   \| v_{N+1} \|_{s+1}^{\Lip(\gamma)}  \Big) \\
						&\lesssim_{s,N} \lambda^{(N+2)(\alpha-2(1-\tc)) + 1-\tc}   \big( 1+ \lambda^{(N+1)(\alpha-2(1-\tc))}  \big) \| f \|_{s+\kappa_{N+1-\frac12}(\tau)+1}  \\
						&\lesssim_{s,N} \lambda^{(N+2)(\alpha-2(1-\tc))+1-\tc}  \| f \|_{s+\kappa_{N+1-\frac12}(\tau)+1} \,,
					\end{aligned}
				\end{equation}
				which yields \eqref{stima.approx.qN} at the step $N+1$, with  $\kappa_{N+1}(\tau):= \kappa_{N+1-\frac12}(\tau)+1$. Furthermore, since $v_{{\rm app},N}$ is a quasi-periodic traveling wave by induction assumption and we showed that $v_{N+1}$ is a quasi-periodic traveling wave, it follows that $q_{N+1}$ in \eqref{jigsaw4} is a quasi-periodic traveling wave as well. Finally, the claim that $v_{{\rm app},N+1}(\vf,x)=\odd(\vf,x)$ follows from the fact that $q_N(\vf,x)=\even(\vf,x)$, $v_{{\rm app},N}(\vf,x)=\odd(\vf,x)$ by induction assumption and by Lemma \ref{lemma.equa.linear.approx}, whereas the $q_{N+1}(\vf,x)=\even(\vf,x)$ follows by \eqref{jigsaw4}, \eqref{internal.rescaled}, \eqref{biot-savart} and Lemma \ref{lemma real rev matrici}.  This concludes the proof of the claim and of the proposition.
			\end{proof}

\section{The linearized operator}\label{sez:linear}
Once the approximate solution $v_{{\rm app},N}(\vf,x)$ with respect to $\lambda\gg1$ (as constructed in Proposition \ref{lemma.approx}, see \eqref{stima.approx.vN} and \eqref{stima.approx.qN}), is obtained, we proceed to study the linearization of equation \eqref{internal.rescaled} at any approximate solution of the Nash-Moser iteration (see Section \ref{sezione:NASH}). This solution takes the form
\begin{equation}\label{eq:approx-sol}
 w= v_{{\rm app},N}+  g \in {\mathcal C}^\infty(\T^\nu \times \T^2) \,,
\end{equation}
satisfying the ansatz
\begin{equation}\label{ansatz}
	\|  w \|_{s_0 + \sigma}^{\Lip(\gamma)} \leq C_0, \quad \text{for some constants } \quad \sigma \,,\, C_0 \gg 0 \quad \text{large enough.}
\end{equation}
Recalling \eqref{prima forma linearizzato beta plane}, the linearized operator  is given by
\begin{equation}\label{operatore linearizzato}
	\cL := \di_{v} \cF(  w) := \lambda \, \omega\cdot \pa_{\vf}   + \beta\,\tL +  \lambda^{\theta}\ba_{0}\cdot \nabla + \lambda^{\theta} {\cE_{0}}\,,
\end{equation}
where
\begin{equation}\label{def coefficienti operatori linearized}
	\ba_{0}:=    \fB\big[  w \big] \,, \quad \cE_{0}[h] :=  \nabla  w \cdot \fB [h]
\end{equation}
and
\begin{equation}\label{theta.def.ridu}
	\theta := \alpha-1+\tc \,, \quad \text{for some} \ \ \tc \in (0,\tfrac13(2-\alpha))\,.
\end{equation}
The range for the parameter $\tc>0$ ensures that, for any fixed $\alpha\in(1,2)$,
\begin{equation}\label{small.theta}
	\theta-1 < \theta - 1 + \tc < 0 \,.
\end{equation}
The parameter $\tc$ will be fixed arbitrarily small enough in Proposition \ref{iterazione-non-lineare}, independently of the step of the Nash-Moser iteration. In particular, it will guarantee that
\begin{equation}\label{piccolo.ansatx}
	\lambda^{\theta-1} \gamma^{-1} \ll 1 \,.
\end{equation}
By the ansatz \eqref{ansatz} and by the definitions \eqref{def coefficienti operatori linearized}, we get that 
\begin{equation}\label{stima bf a1}
	\| \ba_{0} \|_{s} \lesssim_s   \|  w \|_{s-1} \,, \quad \forall s \geq s_0 \,,
\end{equation}
and, for any $s\geq s_0$, $\cE_{0} \in \OpM_{s}^{- 1}$, with the estimate
\begin{equation}\label{stima cal R linearized}
	|\cE_{0}|_{- 1, s, \alpha} \lesssim_{s, \alpha} \|  w \|_{s+1} \,, \quad \forall s \geq s_0\,, \quad \alpha \in \N\,. 
\end{equation}
By the definition of $\ba_{0}$ in \eqref{def coefficienti operatori linearized} and the definition of the Biot-Savart operator $\fB$ in \eqref{biot-savart}, we clearly have that
\begin{equation}\label{divergenza.zero.a1}
	\braket{\ba_{0}}_{x} :=\frac{1}{(2\pi)^2}\int_{\T^2}\ba_{0}(\vf,x) \wrt x = 0\,, \quad {\rm div}(\ba_{0}):= \nabla\cdot \ba_{0} = 0 \,.
\end{equation}
Moreover, using also that ${\rm div} (\nabla^\perp  h) = 0$ for any $h$, the operators $\ba_{0} \cdot \nabla$, $\cE_{0}$ and ${\mathcal L}$ leave invariant the subspace of zero average function, with
\begin{equation}\label{invarianze cal L}
	\begin{aligned}
		& [\Pi_0^\bot , \ba_{0} \cdot \nabla ] = 0\,, \quad [\Pi_0^\bot , \cE_{0} ] = 0\,, \\
		&  \ba_{0} \cdot \nabla \Pi_0 = \Pi_0 \ba_{0} \cdot \nabla = 0\,, \quad \cE_{0} \Pi_0 = \Pi_0 \cE_{0} = 0 \,,\\
		&[\Pi_0^\bot , {\mathcal L}] = 0\,, \quad \Pi_0 {\mathcal L}= {\mathcal L} \Pi_0 = 0\,, 
	\end{aligned}
\end{equation}
implying that
\begin{equation}\label{invarianze.2}
	\ba_{0} \cdot \nabla = \Pi_0^\bot \ba_{0} \cdot \nabla \Pi_0^\bot\,, \quad \cE_{0} = \Pi_0^\bot \cE_{0} \Pi_0^\bot\,, \quad 	{\mathcal L} = \Pi_0^\bot {\mathcal L} \Pi_0^\bot \,.
\end{equation}
We always work on the space of zero average functions and we shall preserve this invariance along the whole paper.

\section{Normal form reduction}\label{ridusezione}

We now present the step-by-step scheme of reduction for the operator ${\mathcal L}$ defined in \eqref{operatore linearizzato}.  

\subsection{Reduction of the transport}

We consider the composition operator
\begin{equation}\label{diffeo.maps}
	(\cB u)(\vf,x) := h (\vf, x+ \bbeta(\vf,x)) \,, \quad (\cB^{-1} h ) (\vf,y) := h (\vf, y + \breve{\bbeta}(\vf,y))\,,
\end{equation}
induced by a $\vf$-dependent family of diffeomorphisms of the torus  
\begin{equation}
	\T^2 \to \T^2, \quad  x \mapsto y := x + \bbeta (\vf,x) \,,
\end{equation}
with inverse $y \mapsto x = y + \breve{\bbeta}(\vf,y)$ for $\| \bbeta \|_{s}$ sufficiently small and where $\bbeta(\vf,x)$, $\breve{\bbeta}(\vf,y)$ are quasi-periodic traveling wave functions.

We need to reduce to constant coefficients the highest order operator
which is a transport operator of the form 
\begin{equation}\label{varelambda}
	\begin{aligned}
		\cT & := \lambda \, \omega \cdot \pa_{\vf} + \lambda^{\theta} \ba_{0} \cdot \nabla = \lambda \, \wt\cT \,, \\
		\wt\cT & :=  \omega \cdot \partial_\vphi + {\bf b}(\vphi, x)\cdot \nabla\,, \quad {\bf b}(\vphi, x) := \varepsilon {\bf a}_0(\vphi, x)\,, \quad \varepsilon := \lambda^{\theta-1}\,. 
	\end{aligned}
\end{equation}
Note that, for $\lambda \gg 1$, by\eqref{theta.def.ridu}-\eqref{small.theta}, one has that $\varepsilon = \lambda^{\theta-1} \ll 1$ and the quasi-periodic traveling wave function ${\bf b}(\vphi, x)$ satisfies the estimate 
\begin{equation}\label{ridu trasporto}
\begin{aligned}
	 &	\| {\bf b} \|_s \lesssim_s \lambda^{\theta-1} \|  w \|_{s-1}\,, \quad \forall s \geq s_0\,, \\
		& \| \Delta_{12} {\bf b} \|_{s_1} \lesssim_{s_1} \lambda^{\theta-1} \| w_1 - w_2 \|_{s_1-1} \,, \quad s_1 \geq s_0\,,
\end{aligned}
\end{equation}
where $ w$ is given in \eqref{eq:approx-sol}, the notation $\Delta_{12}$ is given in \eqref{deltaunodue}, and $w_1, w_2$ 
satisfy the ansatz \eqref{ansatz}.
We consider the set of  Diophantine non-resonance conditions as defined in \eqref{DC.2gamma}.
We have the following result.
\begin{prop}\label{proposizione trasporto}
	{\bf (Straightening of the transport operator $\mathcal T$).} Let $\gamma\in (0,1)$ and $\tau>\nu-1$.
	There exists  $\sigma := \sigma (\tau, \nu) > 0$ large enough  such that, for any 
	$S > s_0 + \sigma$, there exist $\delta := \delta(S, \tau, \nu) \in (0, 1)$ small enough and $\tau_{1}=\tau_{1}(\tau,\nu)>0$ such that, if \eqref{ansatz}, \eqref{ridu trasporto} hold  and
	\begin{equation}\label{condizione piccolezza rid trasporto}
		N_0^{\tau_{1}}\varepsilon \gamma^{- 1} = N_0^{\tau_{1}} \lambda^{\theta-1} \gamma^{-1} \leq \delta \,,
	\end{equation} 
	is fulfilled, then the following holds. 
	There exists an invertible diffeomorphism 
	$\T^2 \to \T^2$, $x \mapsto x + \bbeta(\vphi, x; \omega)$ 
	with inverse $y \mapsto y + \breve \bbeta(\vphi, y; \omega)$,
	defined for all $\omega \in \tD\tC(2\gamma, \tau)$, with the set given in \eqref{DC.2gamma},
	satisfying, for any $s_0\leq s \leq S - \sigma$,
	\begin{equation}\label{stima alpha trasporto}
		\| \bbeta \|_s^{{\rm Lip}(\gamma)}, \| \breve \bbeta\|_{s}^{{\rm Lip}(\gamma)} \lesssim_{s} N_0^{2\tau+1} \varepsilon \gamma^{- 1}  \|  w \|_{s + \sigma}^{\Lip(\gamma)} \,,
	\end{equation}
	such that
	one gets the conjugation 
	\begin{equation}\label{coniugazione nel teo trasporto}
		{\mathcal B}^{- 1} {\mathcal T} {\mathcal B} = \lambda\, \omega \cdot \partial_\vphi\,,
	\end{equation}
	where the invertible maps $\cB$, $\cB^{-1}$ are defined in \eqref{diffeo.maps}, satisfying the estimates , for any $s\in[s_0,S-\sigma]$,
	\begin{equation}\label{stima tame cambio variabile rid trasporto}
		\begin{aligned}
			& \| {\mathcal B}^{\pm 1}  h\|_s^{{\rm Lip}(\gamma)} \lesssim_{s} \| h \|_s^{{\rm Lip}(\gamma)} + \|  w \|_{s + \sigma}^{\Lip(\gamma)} \| h \|_{s_0}^{\Lip(\gamma)}\,, \\
			&   \| ({\mathcal B}^{\pm 1} - {\rm Id})  h\|_s^{{\rm Lip}(\gamma)}
			\lesssim_{s} N_0^{2\tau+1} \varepsilon \gamma^{- 1} \big(  \|  w \|_{s_0 + \sigma}^{{\rm Lip}(\gamma)}  \| h \|_{s + 1}^{{\rm Lip}(\gamma)} +  \|  w \|_{s + \sigma }^{{\rm Lip}(\gamma)}  \| h \|_{s_0 + 1}^{{\rm Lip}(\gamma)} \big) \,. 
		\end{aligned}
	\end{equation}
	Let $s_1 \geq s_0$ and assume that $w_1, w_2$ 
	satisfy \eqref{ansatz} with $ \sigma_{0}  \geq s_1 + \sigma $. 
	Then, for any $ \omega \in \tD\tC(2\gamma, \tau)$, recalling the notation $\Delta_{12}$ in \eqref{deltaunodue}, one has 
	\begin{equation}\label{stime delta 12 prop trasporto}
		\begin{aligned}
			& \| \Delta_{12} \bbeta \|_{s_1} \,,\, \| \Delta_{12} \breve \bbeta \|_{s_1} \lesssim_{s_1} N_0^{2\tau +1}\varepsilon \gamma^{- 1} \| w_1 - w_2 \|_{s_1 + \sigma}\,, \\
			& \| \Delta_{12} {\mathcal B}^{\pm 1} h \|_{s_1}
			\lesssim_{s_1} N_0^{2\tau +1} \varepsilon \gamma^{- 1} \| w_1 -w_2 \|_{s_1 +\sigma} \| h \|_{s_1 + 1}\,. 
		\end{aligned}
	\end{equation}
	Furthermore, $\bbeta,\breve{\bbeta}$ are quasi-periodic traveling waves, and the related maps ${\mathcal B}, {\mathcal B}^{- 1}$ are momentum preserving (see Definition \ref{def:mom.pres}). In addition, if we assume $ w(\vf,x)= \odd(\vf,x)$, then $\bbeta,\breve{\bbeta}$ are $\odd(\vphi,x)$ and the related maps ${\mathcal B}, {\mathcal B}^{- 1}$ are reversibility preserving (see Definition \ref{reserv.operators.def}) .
\end{prop}

In order to prove Proposition \ref{proposizione trasporto}, we follow \cite{BM20} and \cite{BFM21}. First we show the following iterative lemma. We fix the constants
\begin{equation}\label{constantine}
	\begin{aligned}
		& N_0 >0 \,, \quad N_{-1}:= 1\,, \quad \chi := 3/2\,, \quad N_{n} := N_0^{\chi^n}\,, \quad n \in \N_0\,, \\
		& \fa:= 6 \tau +8\,, \quad \fb:= \fa +1 \,.
	\end{aligned}
\end{equation} 

\begin{lem}\label{lemma.iterativo.trasporto}
	Let $\gamma\in (0,1)$ and $\tau>\nu-1$.
	There exists  $\sigma := \sigma (\tau, \nu) > 0$ large enough  such that, for any 
	$S > s_0 + \sigma$, there exist $\delta := \delta(S, \tau, \nu) \in (0, 1)$ small enough, $N_0=N_0(S,\tau,\nu)>0$ large enough and $\tau_{1}=\tau_{1}(\tau,\nu)>0$ such that, if \eqref{ansatz}, \eqref{ridu trasporto} and \eqref{condizione piccolezza rid trasporto} hold for $\sigma_{0}\ge s_0 +\sigma$, the following holds for $n\geq 0$:
	\\[1mm]
	There exists a linear operator
	\begin{equation}\label{robertino1}
		\begin{aligned}
			\cT_{n} := \lambda\,\wt\cT_{n}\,, \quad  \wt\cT_{n} := \omega\cdot \pa_{\vf} + \tm_{n} \cdot \nabla + \bb_{n}(\vf,x) \cdot \nabla \,,
		\end{aligned}
	\end{equation}
	defined for all $\omega\in \tD\tC(\gamma,\tau)$, where $\bb_{n}(\vf,x)$ is a quasi-periodic traveling wave, with estimates, for any $s \in [s_0,S-\sigma]$,
	\begin{equation}\label{dormiente1}
		\| \bb_{n} \|_{s}^{\Lip(\gamma)} \leq C(s,\fb) N_{n-1}^{-\fa} \| \bb \|_{s+\sigma} \,, \quad  \| \bb_{n} \|_{s+\fb}^{\Lip(\gamma)} \leq C(s,\fb) N_{n-1} \| \bb \|_{s+\sigma}\,,
	\end{equation}
	for some constant $C(s,\fb)>0$ monotone increasing with $[s_0,S-\sigma]$ and $\tm_{n}$ is a real constant satisfying
	\begin{equation}\label{dormiente2}
		| \tm_{n} |^{\Lip(\gamma)} \leq 2 \| \bb\|_{s_0}^{\Lip(\gamma)} \,, \quad |\tm_{n} - \tm_{n-1} |^{\Lip(\gamma)} \leq C(s_0,\fb) N_{n-2}^{-\fa} \| \bb \|_{s_0+\fb}^{\Lip(\gamma)} \,, \ \ n\geq 2\,.
	\end{equation}
	If $n=0$, we set $\cO_{0}^{\gamma}:= \tD\tC(\gamma,\tau) $ and, for $n\geq 1$, we define
	\begin{equation}\label{On.gamma}
		\begin{aligned}
			\cO_{n}^{\gamma}  := \Big\{   \omega \in \cO_{n-1} \, : \,  |\omega\cdot \ell + \tm_{n-1}\cdot j | & \geq {\gamma}\,{\braket{\ell}^{-\tau}} \,, \ \  \forall\,(\ell,j)\in\Z^{\nu+2} \setminus\{0\}\,, 
			\\
			& \quad \ \ | \ell | \leq N_{n-1} \,, \ \ \pi^\top(\ell) + j = 0 \Big\} \,.
		\end{aligned}
	\end{equation}
	For $n\geq 1$, there exists an invertible diffeomorphism of the torus $\T^2\to \T^2$, $x \mapsto x + \bbeta_{n-1}(\vf,x)$, with inverse $y \mapsto y + \breve{\bbeta}_{n-1}(\vf,x)$, such that, for any $s\in [s_0,S - \sigma]$,
	\begin{equation}\label{dormiente3}
		\begin{aligned}
			& \| \bbeta_{n-1} \|_{s}^{\Lip(\gamma)}\,, \ \| \breve{\bbeta}_{n-1} \|_{s}^{\Lip(\gamma)} \lesssim_{s} N_{n-1}^{2\tau+1} N_{n-1}^{-\fa} \varepsilon \gamma^{-1} \|  w \|_{s+\sigma}^{\Lip(\gamma)} \,, \\ 
			& \| \bbeta_{n-1} \|_{s+\fb}^{\Lip(\gamma)} \,, \  \| \breve{\bbeta}_{n-1} \|_{s+\fb}^{\Lip(\gamma)} \lesssim_{s} N_{n-1}^{2\tau+1} N_{n-1} \varepsilon \gamma^{-1} \|   w \|_{s+\sigma}^{\Lip(\gamma)} \,.
		\end{aligned}
	\end{equation}
	The operators
	\begin{equation}\label{diffeo.maps.n-1}
		(\cB_{n-1} u)(\vf,x) := h (\vf, x+ \bbeta_{n-1}(\vf,x)) \,, \quad (\cB_{n-1}^{-1} h ) (\vf,y) := h (\vf, y + \breve{\bbeta}_{n-1}(\vf,y))\,,
	\end{equation}
	satisfy, for any $\omega \in \cO_{n}^{\gamma}$, the conjugation
	\begin{equation}\label{coniugio.trasporto}
		\cT_{n} = \cB_{n-1}^{-1} \cT_{n-1} \cB_{n-1} \,.
	\end{equation}
	Furthermore, $\bbeta_{n-1},\breve{\bbeta}_{n-1}$ are  quasi-periodic traveling waves, and the related maps ${\mathcal B}_{n-1}, {\mathcal B}_{n-1}^{- 1}$ are momentum preserving. 
		In addition, if we assume $ w(\vf,x)=\odd(\vf,x)$, then $\bbeta_{n-1},\breve{\bbeta}_{n-1}$ are $\odd(\vphi,x)$, the related maps ${\mathcal B}_{n-1}, {\mathcal B}_{n-1}^{- 1}$ are reversibility preser\-ving, and $\bb_{n}$ is $\even(\vf,x)$.
\end{lem}

\begin{proof}
	We argue by induction on $n\in\N_0$. The claimed statements for $n=0$ follow directly by setting $\tm_{0}:= 0$ and $\bb_{0}(\vf,x):=\bb(\vf,x)$, with $\bb(\vf,x)$ defined in \eqref{varelambda} with estimate \eqref{ridu trasporto}. \\
	We now assume that the claimed properties hold for some $n\geq 0$ and we prove them at the step $n+1$. We look for a diffeomorphism of the torus $\T^2 \to \T^2$, $x \mapsto x + \bbeta_{n}(\vf,x)$, with inverse $y \mapsto y + \breve{\bbeta}_{n}(\vf,x)$, such that, defining the operators
		\begin{equation}\label{diffeo.maps.n}
		(\cB_{n} u)(\vf,x) := h (\vf, x+ \bbeta_{n}(\vf,x)) \,, \quad (\cB_{n}^{-1} h ) (\vf,y) := h (\vf, y + \breve{\bbeta}_{n}(\vf,y))\,,
	\end{equation}
	the operator $\cB_{n}^{-1} \cT_{n} \cB_{n}= \lambda \, \cB_{n}^{-1} \wt\cT_{n} \cB_{n}$ has the desired properties. We compute
	\begin{equation}
		\begin{aligned}
			\cB_{n}^{-1}   \wt\cT_{n} \cB_{n} & = \omega \cdot \pa_{\vf} + \tm_{n} \cdot \nabla + \big\{  \cB_{n}^{-1} \big( \omega\cdot \pa_{\vf} \bbeta_{n}  + \tm_{n} \cdot \nabla \bbeta_{n}  + \bb_{n} + \bb_{n} \cdot \nabla \bbeta_{n}  \big)  \big\} \cdot \nabla \\
			& = \omega \cdot \pa_{\vf} + \tm_{n} \cdot \nabla + \big\{  \cB_{n}^{-1} \big( \omega\cdot \pa_{\vf} \bbeta_{n}  + \tm_{n} \cdot \nabla \bbeta_{n}  + \Pi_{N_n} \bb_{n}  \big)  \big\} \cdot \nabla + \bb_{n+1} \cdot \nabla \,,
		\end{aligned}
	\end{equation}
	where
	\begin{equation}\label{neo2}
		\bb_{n+1} := \cB_{n}^{-1} \bg_{n} \,, \quad \bg_{n}:= \Pi_{N_n}^\perp \bb_{n} + \bb_{n} \cdot \nabla \bbeta_{n} \,,
	\end{equation}
	and the projectors $\Pi_{N_n}$, $\Pi_{N_n}^\perp$ are defined in \eqref{def:smoothings}. For any $\omega\in \cO_{n+1}^\gamma$, we solve the homological equation
	\begin{equation}\label{neo1}
		\omega\cdot \pa_{\vf} \bbeta_{n}  + \tm_{n} \cdot \nabla \bbeta_{n}  + \Pi_{N_n} \bb_{n} = \braket{\bb_{n}}_{\vf,x} \,,
	\end{equation}
 where
 \begin{equation}
 	\braket{\bb_{n}}_{\vf,x} := \frac{1}{(2\pi)^{\nu+2}} \int_{\T^{\nu+2}} \bb_{n}(\vf,x) \wrt 	\vf \wrt x \in \R^{2} \,.
 \end{equation}
 Explicitly, the solution of the equation \eqref{neo1} is given by
 \begin{equation}\label{neo3}
 	\begin{aligned}
 		\bbeta_{n}(\vf,x) & := -( \omega\cdot \pa_{\vf} + \tm_{n} \cdot \nabla)^{-1} \big[ \Pi_{N_n}\bb_{n} - \braket{\bb_{n}}_{\vf,x}  \big] \\
 		& := - \sum_{(\ell, j) \in \Z^{\nu+2} \setminus \{0\}, \atop |\ell|\leq N_n, \, \pi^\top(\ell) + j = 0} \frac{1}{\im(\omega\cdot \ell + \tm_{n}\cdot j)} \wh\bb_{n}(\ell,j) e^{\im(\ell\cdot \vf + j\cdot x)} \,.
 	\end{aligned}
 \end{equation}
 We define
 \begin{equation}\label{neo4}
 \begin{aligned}
 		\cT_{n+1} & := \lambda \, \wt\cT_{n+1}\,, \quad  \wt\cT_{n+1} := \omega\cdot \pa_{\vf} + \tm_{n+1} \cdot \nabla + \bb_{n+1} \cdot \nabla \,, \\
 		\tm_{n+1} & := \tm_{n} + \braket{\bb_{n}}_{\vf,x}  \,.
 \end{aligned}
 \end{equation}
 Therefore, for any $\omega \in \cO_{n+1}^\gamma$, we have that $\wt\cT_{n+1}= \cB_{n}^{-1} \wt\cT_{n} \cB_{n}$, so that \eqref{coniugio.trasporto} holds at the step $n+1$. Moreover, note that $\bbeta_{n}(\vf,x;\omega)$ in \eqref{neo3} admits a Lipschitz extension with respect to $\omega$ to $\tD\tC(\gamma,\tau)$, so that $\bb_{n+1}$ in \eqref{neo2} and $\cT_{n+1}$ \eqref{neo3} can be well defined for all $\omega \in \tD\tC(\gamma,\tau)$. By \eqref{neo3} and Lemma \ref{lemma:smoothing}, we have, for any $s\geq 0$,
 \begin{equation}\label{ciao}
 	\begin{aligned}
 		\| \bbeta_{n} \|_{s}^{\Lip(\gamma)} & \lesssim_{s} \| \Pi_{N_n} \bb_{n} \|_{s+2\tau+1}^{\Lip(\gamma)}  \lesssim_{s} N_n^{2\tau +1}  \gamma^{-1} \|  \bb_{n}\|_{s}^{\Lip(\gamma)} \,, \\
 		\| \nabla \bbeta_{n} \|_{s}^{\Lip(\gamma)} & \lesssim_{s} \| \Pi_{N_n} \bb_{n} \|_{s+2\tau+2}^{\Lip(\gamma)}  \lesssim_{s} N_n^{2\tau +2}  \gamma^{-1} \|  \bb_{n}\|_{s}^{\Lip(\gamma)} \,.
 	\end{aligned}
 \end{equation}
 The latter estimate, together with the induction estimate \eqref{dormiente1} on $\bb_{n}$, imply, for any $s \in [s_0,S-\sigma]$, 
 \begin{equation}\label{neo5}
 	\begin{aligned}
 		\| \bbeta_{n+1} \|_{s}^{\Lip(\gamma)} \,, \| \breve{\bbeta}_{n+1} \|_{s}^{\Lip(\gamma)} & \lesssim_{s} N_n^{2\tau+1} N_{n-1}^{-\fa} \varepsilon \gamma^{-1} \|  w \|_{s+\sigma}^{\Lip(\gamma)} \,, \\
 			\| \bbeta_{n+1} \|_{s+\fb}^{\Lip(\gamma)} \,, \| \breve{\bbeta}_{n+1} \|_{s+\fb}^{\Lip(\gamma)} & \lesssim_{s} N_n^{2\tau+1} N_{n-1} \varepsilon \gamma^{-1} \|  w \|_{s+\sigma}^{\Lip(\gamma)} \,,
 	\end{aligned}
 \end{equation}
 which are the estimates \eqref{dormiente3} at the step $n+1$. Note that, by the definition of the constant $\fa$ in \eqref{constantine} and the ansatz \eqref{ansatz}, from \eqref{neo5} we deduce, with $s=s_0$,
 \begin{equation}
 	\| \bbeta_{n+1}\|_{s_0}^{\Lip(\gamma)} \,, \| \breve{\bbeta}_{n+1}\|_{s_0}^{\Lip(\gamma)} \lesssim N_0^{2\tau+1} \varepsilon \gamma^{-1}\,.
 \end{equation}
 Together with the smallness condition \eqref{condizione piccolezza rid trasporto}, Lemma \ref{prop.base.diffeo} implies that, for any $s_0 \leq s \leq S-\sigma+\fb$,
 \begin{equation}\label{ciao2}
 	\| \cB_{n}^{\pm1} h\|_{s}^{\Lip(\gamma)} \lesssim_{s} \| h \|_{s}^{\Lip(\gamma)} + N_n^{2\tau+1} \gamma^{-1} \| \bb_{n} \|_{s}^{\Lip(\gamma)} \| h \|_{s_0}^{\Lip(\gamma)}\,.
 \end{equation}
 Now we estimate $\bb_{n+1}$ in \eqref{neo2}.  First, we estimate $\bg_{n}$. By \eqref{constantine}, \eqref{dormiente1}, the ansatz \eqref{ansatz} and the smallness condition \eqref{condizione piccolezza rid trasporto}, we have
 \begin{equation}\label{neo6}
 	N_n^{2\tau+2} \gamma^{-1} \| \bb_{n} \|_{s_0}^{\Lip(\gamma)} \leq 1\,. 
 \end{equation}
 By Lemma \ref{lemma:smoothing}, Lemma \ref{lemma:LS norms} and estimates \eqref{ciao}, \eqref{neo6}, we have, for any $s \in [s_0,S-\sigma]$,
 \begin{equation}\label{ciao3}
 	\begin{aligned}
 		& \| \bg_{n} \|_{s_0}^{\Lip(\gamma)} \lesssim \| \bb_{n}\|_{s_0}^{\Lip(\gamma)} \,, \\
 		& \| \bg_{n}  \|_{s}^{\Lip(\gamma)} \lesssim_{s} N_n^{-\fb} \| \bb_{n} \|_{s+\fb}^{\Lip(\gamma)} + N_{n-1}^{2\tau+2} \gamma^{-1}  \| \bb_{n} \|_{s}^{\Lip(\gamma)} \| \bb_{n} \|_{s_0}^{\Lip(\gamma)} \,, \\
 		&  \| \bg_{n} \|_{s+\fb}^{\Lip(\gamma)} \lesssim_{s}  \| \bb_{n} \|_{s+\fb}^{\Lip(\gamma)} (1 + N_{n}^{2\tau+2} \gamma^{-1}  \| \bb_{n} \|_{s_0}^{\Lip(\gamma)})  \lesssim_{s} \| \bb_{n} \|_{s+\fb}^{\Lip(\gamma)} \,.
  	\end{aligned}
 \end{equation}
 Therefore, estimates \eqref{ciao2}-\eqref{ciao3} imply that, for any $s \in [s_0,S-\sigma]$,
 \begin{equation}
 	\begin{aligned}
 		&  \| \bb_{n+1} \|_{s}^{\Lip(\gamma)} \lesssim_{s} N_n^{-\fb}  \| \bb_{n} \|_{s+\fb}^{\Lip(\gamma)} + N_n^{2\tau+2} \gamma^{-1}  \| \bb_{n} \|_{s}^{\Lip(\gamma)}  \| \bb_{n} \|_{s_0}^{\Lip(\gamma)} \,, \\
 		&  \| \bb_{n} \|_{s+\fb}^{\Lip(\gamma)} \lesssim_{s}  \| \bb_{n} \|_{s+\fb}^{\Lip(\gamma)} \,.
 	\end{aligned}
 \end{equation}
 Using the definition of the constants $\fa$, $\fb$ in \eqref{constantine} and the induction assuptions on the estimates \eqref{dormiente1}, by a standard argument we deduce the estimates \eqref{dormiente1} at the step $n+1$. The estimates \eqref{dormiente2} follow by the definition of $\tm_{n+1}$ in \eqref{neo4}, the induction assumption on \eqref{dormiente1} and by a telescopic argument. \\
 Finally, by \eqref{neo3}, since $\bb_{n}(\vf,x)$ is a quasi-periodic traveling wave by induction assumption, we have that $\bbeta_{n}(\vf,x)$, $\breve{\bbeta}_{n}(\vf,y)$ are quasi-periodic traveling waves by Lemma \ref{lem:dMP}, which also imply that $\cB_{n}^{\pm 1}$ are momentum preserving operators. Recalling \eqref{neo2}, we deduce that also $\bb_{n+1}(\vf,y)$ is a quasi-periodic traveling wave.
 
 \noindent
 Finally if $ w$ is ${\rm odd}(\vf, x)$, by \eqref{neo3}, since $\bb_{n}(\vf,x)$ is $\even(\vf,x)$, we have that $\bbeta_{n}(\vf,x)$ and $\breve{\bbeta}_{n}(\vf,y)$ are $\odd(\vf,y)$ and hence $\cB_{n}^{\pm 1}$ are reversibility preseving operators. Recalling \eqref{neo2}, we deduce that also $\bb_{n+1}(\vf,y)$ is $\even(\vf,y)$. This concludes the proof.
\end{proof}

We then define
\begin{equation}\label{compo.mappe.diffeo}
	\wt\cB_{n} := \cB_{0} \circ \cB_{1} \circ ... \circ \cB_{n} \,, \quad \text{with inverse} \quad \wt\cB_{n}^{-1} = \cB_{n}^{-1} \circ ... \circ \cB_{1}^{-1} \circ \cB_{0}^{-1} \,.
\end{equation}

\begin{lem}\label{lemma.limite.trasporto}
	Let $\gamma\in (0,1)$ and $\tau>\nu-1$.
	There exists  $\sigma := \sigma (\tau, \nu) > 0$ large enough  such that, for any 
	$S > s_0 + \sigma$, there exist $\delta := \delta(S, \tau, \nu) \in (0, 1)$ small enough, $N_0=N_0(S,\tau,\nu)>0$ large enough and $\tau_{1}=\tau_{1}(\tau,\nu)>0$ such that, if \eqref{ansatz}, \eqref{ridu trasporto} and \eqref{condizione piccolezza rid trasporto} hold for $\sigma_{0}\ge s_0 +\sigma$, the following hold:
	\\[1mm]
	\noindent $(i)$ For any $n \in \N_0$, we have
	\begin{equation}
			(\wt\cB_{n} u)(\vf,x) := h (\vf, x+ \balpha_{n}(\vf,x)) \,, \quad (\wt\cB_{n}^{-1} h ) (\vf,y) := h (\vf, y + \breve{\balpha}_{n}(\vf,y))\,,
	\end{equation}
	for some functions $\balpha_{n}(\vf,x)$, $\breve{\balpha}_{n}(\vf,y)$ such that, for any $s\in [s_0,S-\sigma]$,
	\begin{equation}
		\begin{aligned}
			& \| \balpha_{0}\|_{s}^{\Lip(\gamma)}\,, \ \|  \breve{\balpha}_{0}\|_{s}^{\Lip(\gamma)} \lesssim_{s} N_0^{2\tau+1} \varepsilon\gamma^{-1} \|  w \|_{s+\sigma}^{\Lip(\gamma)} \,, \\
			&\| \balpha_{n} - \balpha_{n-1} \|_{s}^{\Lip(\gamma)}\,, \ \|  \breve{\balpha}_{n}- \breve{\balpha}_{n-1} \|_{s}^{\Lip(\gamma)} \lesssim_{s} N_n^{2\tau+1} N_{n-1}^{-\fa} \varepsilon\gamma^{-1} \|  w \|_{s+\sigma}^{\Lip(\gamma)} \,, \quad n\geq1\,.
		\end{aligned}
	\end{equation}
	As a consequence, for any $s \in [s_0,S-\sigma]$,
	\begin{equation}\label{robertino5}
		\| \balpha_{n}\|_{s}^{\Lip(\gamma)} \lesssim_{s} N_0^{2\tau+1} \varepsilon\gamma^{-1} \|  w \|_{s+\sigma}^{\Lip(\gamma)} \,.
	\end{equation}
		Furthermore, $\balpha_{n}$, $\breve{\balpha}_{n}$ are both quasi-periodic traveling waves. In addition, if we assume $\bb_{0}(\vf,x)=\even(\vf,x)$, we have that $\balpha_{n}(\vf,x)$ is $\odd(\vf,x)$, $\breve{\balpha}_{n}(\vf,y)$ is $\odd(\vf,y)$;
	\\[1mm]
	\noindent $(ii)$ For any $s \in [s_0,S-\sigma]$, the sequence $(\balpha_{n}(\vf,x))_{n\in\N_0}$ (resp. $(\breve{\balpha}_{n}(\vf,y))_{n\in\N_0}$) is a Cauchy sequence with respect to the norm $\| \,\cdot \,\|_{s}^{\Lip(\gamma)}$ and it converges to some limit $\bbeta(\vf,x)$ (resp. $\breve{\bbeta}(\vf,y)$), with estimates, for any $n\in\N_0$,
	\begin{equation}
		\begin{aligned}
			& \| \bbeta - \balpha_{n} \|_{s}^{\Lip(\gamma)} \,, \ \| \breve{\bbeta} - \breve{\balpha}_{n}\|_{s}^{\Lip(\gamma)} \lesssim_{s} N_{n+1}^{2\tau+1} N_{n}^{1-\fa} \varepsilon \gamma^{-1} \|  w \|_{s+\sigma}^{\Lip(\gamma)} \,, \\
			& \| \bbeta \|_{s}^{\Lip(\gamma)}\,, \ \| \breve{\bbeta} \|_{s}^{\Lip(\gamma)}\lesssim_{s} N_0^{2\tau+1} \varepsilon \gamma^{-1} \|  w \|_{s+\sigma} \,.
		\end{aligned}
	\end{equation}
			Furthermore, $\bbeta$, $\breve{\bbeta}$ are both quasi-periodic traveling waves. In addition, if we assume $\bb_{0}(\vf,x)=\even(\vf,x)$, we have that $\bbeta(\vf,x)$ is $\odd(\vf,x)$, $\breve{\bbeta}(\vf,y)$ is $\odd(\vf,y)$;
	\\[1mm]
	\noindent $(iii)$ Define the operators
	\begin{equation}
			(\cB u)(\vf,x) := h (\vf, x+ \bbeta(\vf,x)) \,, \quad (\cB^{-1} h ) (\vf,y) := h (\vf, y + \breve{\bbeta}(\vf,y))\,.
	\end{equation}
	Then, for any $s \in [s_0,S-\sigma]$, the sequences $(\wt\cB_{n}^{\pm1})_{n\in\N_0}$ converge strongly in $H^s(\T^{\nu+2})$ to $\cB^{\pm 1}$, namely $\lim_{n\to + \infty}\|(\cB^{\pm 1}- \wt\cB_{n}^{\pm 1})h \|_{s} =0$ for any $h \in H^s(\T^{\nu+2})$.
	Furthermore, $\cB^{\pm1}$ are momentum preserving. In addition, if we assume $ w(\vf,x)=\odd(\vf,x)$, they are reversibility preserving.
 \end{lem}
 \begin{proof}
 	See the proof of Lemma 4.4 in \cite{BM20}. The property of the quasi-periodic traveling wave functions and momentum preserving operators follow from Lemma \ref{lemma.iterativo.trasporto}, since both the properties are closed under limit.
 \end{proof}
 
 \begin{lem}\label{lemma.robertino}
 	The sequence $(\tm_{n})_{n\in\N_0}$ satisfies the bound
 	\begin{equation}\label{stima.mn.superpiccoli}
 		\sup_{\omega \in \cO_{n}^\gamma} |\tm_{n}(\omega)| \lesssim \varepsilon N_{n-1}^{-\fa} \,.
 	\end{equation}
 	Furthermore, we have that $\tD\tC(2\gamma,\tau) \subseteq \cap_{n \geq 0} \cO_{n}^\gamma$ (recalling \eqref{DC.2gamma}, \eqref{On.gamma}).
 \end{lem}
 \begin{proof}
 	We start with the proof of \eqref{stima.mn.superpiccoli}. By \eqref{compo.mappe.diffeo}, using \eqref{robertino1} and \eqref{coniugio.trasporto}, we have
 	\begin{equation}\label{robertino2}
 		\omega\cdot \pa_{\vf} + \tm_{n} \cdot \nabla + \bb_{n} \cdot \nabla = \wt\cT_{n}  = \wt\cB_{n-1}^{-1} \wt\cT_{0} \wt\cB_{n-1} \quad \forall\,\omega \in \cO_{n}^\gamma\,,
  	\end{equation}
  	and, by Lemma \eqref{lemma.limite.trasporto}-$(i)$, we compute explicitly
  	\begin{equation}\label{robertino3}
  		\wt\cB_{n-1}^{-1} \wt\cT_{0} \wt\cB_{n-1} = \omega\cdot \pa_{\vf} + \wt\cB_{n-1}^{-1} \big( \omega\cdot \pa_{\vf} \balpha_{n-1} + \bb + \bb\cdot \nabla\balpha_{n-1} \big) \cdot \nabla \,.
  	\end{equation}
  	Since, clearly, $\wt\cB_{n-1} [c]= c$ for any constant $c\in\R$, we have, by \eqref{robertino2}, \eqref{robertino3},
  	\begin{equation}\label{robertino4}
  		\tm_{n} + \wt\cB_{n-1} \bb_{n} = \omega\cdot\pa_{\vf} \balpha_{n-1}+ \bb + \bb\cdot \balpha_{n-1} \,.
  	\end{equation}
  	We note that, by \eqref{divergenza.zero.a1},  \eqref{varelambda},
  	\begin{equation}\label{robertino1000}
  		\begin{aligned}
  		&	\int_{\T^{\nu+2}} \omega\cdot \pa_{\vf}\balpha_{n-1} \wrt \vf\wrt x = 0 \,, \quad  \int_{\T^{\nu+2}} \bb(\vf,x)  \wrt \vf \wrt x = 0\,, \\
  		& \int_{\T^{\nu+2}} \bb \cdot \nabla\balpha_{n-1} \wrt \vf \wrt x = - \int_{\T^{\nu+2}} {\rm div}(\bb) \balpha_{n-1} \wrt \vf \wrt x = 0\,.
   		\end{aligned}
  	\end{equation}
  	Taking the space-time average of the equation \eqref{robertino4} and using \eqref{robertino1000}, we deduce that
  	\begin{equation}
  		\tm_{n} = - \int_{\T^{\nu+2}} \wt\cB_{n-1} \bb_{n} \wrt \vf\wrt x \quad \forall\,\omega\in\cO_{n}^\gamma \,.
  	\end{equation}
  	The claimed estimate \eqref{stima.mn.superpiccoli} follows by  Lemmata \ref{lemma:LS norms}, \ref{prop.base.diffeo} and estimates \eqref{dormiente1}, \eqref{robertino5}, \eqref{ansatz} and \eqref{condizione piccolezza rid trasporto}. \\
  	We now prove the claim on the inclusion.  We argue by induction, that is, we show that $\tD\tC(2\gamma,\tau) \subseteq \cO_{n}^\gamma$ for any $n\in\N_0$. For $n=0$, we have that $\cO_{n}^\gamma = \tD\tC(\gamma,\tau)$, so the claim is trivially satisfied. Now we assume that $\tD\tC(2\gamma,\tau) \subseteq \cO_{n}^\gamma$ for some $n\geq 0$ and we want to prove the inclusion with $n+1$. Let $\omega\in \tD\tC(2\gamma,\tau)$. By induction hypothesis, $\omega \in \cO_{n}^\gamma$. Therefore, by \eqref{stima.mn.superpiccoli}, we have $|\tm_{n}| \lesssim \varepsilon  N_{n-1}^{-\fa}$. Hence, for any $(\ell,j)\in \Z^{\nu+2}\setminus\{0\}$, with $|\ell| \leq N_{n}$ and $\pi^\top(\ell)+j=0$ (which implies $|j|\lesssim N_{n} $), we have that
  	\begin{equation}
  		\begin{aligned}
  			|\omega\cdot \ell +\tm_{n}(\omega)\cdot j| & \geq |\omega\cdot \ell| - |\tm_{n}(\omega)\cdot j| \\
  			& \geq \frac{2\gamma}{\braket{\ell}^{\tau}} - \varepsilon C N_{n} N_{n-1}^{-\fa} \geq \frac{\gamma}{\braket{\ell}^\tau} \,,
  		\end{aligned}
  	\end{equation} 
  	for some constant $C>0$, provided $C N_{n}^{1+\tau} N_{n-1}^{-\fa} \varepsilon \gamma^{-1}\leq 1$. This holds for any $n\geq 0$, provided $C N_0^{1+\tau}\varepsilon\gamma^{-1} \leq 1$, which is satisfied by the smallness condition \eqref{condizione piccolezza rid trasporto} and by \eqref{constantine}. Thus, by the definition of $\cO_{n+1}^\gamma$ (see \eqref{On.gamma} with $n+1$), we have $\omega\in\cO_{n+1}^\gamma$. The proof is concluded.
 \end{proof}
 
 \begin{proof}[Proof of Proposition \ref{proposizione trasporto}]
 	For any $\omega\in \tD\tC(2\gamma,\tau)$, by Lemma \ref{lemma.robertino}, $\tm_{n}\to 0$ as $n\to\infty$. By \eqref{robertino5}, \eqref{dormiente1} and Lemma \ref{lemma:LS norms}, we have $\| \wt\cB_{n-1} \bb_{n} \|_{s_0} \lesssim \| \bb_{n}\| \to 0$ as $n\to \infty$. Moreover, 
 	\begin{equation}
 		\| \pa_{\vf} \balpha_{n-1} - \pa_{\vf}\bbeta\|_{s_0}\,, \| \nabla\balpha_{n-1}-\nabla \bbeta \|_{s_0} \leq \| \balpha_{n-1}-\bbeta \|_{s_0+1} \to 0 \quad \text{as}  \ \ n\to \infty \,.
 	\end{equation}
 	Hence, passing to the limit in the norm $\| \,\cdot \,\|_{s_0}$ in identity \eqref{robertino4}, we obtain the identity
 	\begin{equation}\label{robertino1001}
 		\omega\cdot \pa_{\vf} \bbeta + \bb(\vf,x) + \bb(\vf,x) \cdot \bbeta =0  
 	\end{equation}
 	in $H^{s_0}(\T^{\nu+1})$ and therefore pointwise for all $(\vf,x) \in \T^{\nu+2}$, for any $\omega\in\tD\tC(2\gamma,\tau)$. As a consequence,
 	\begin{equation}
 		\cB^{-1} \wt\cT\cB = \omega\cdot \pa_{\vf} + \big\{ \cB^{-1} \big( \omega\cdot \pa_{\vf}\bbeta + \bb + \bb\cdot \nabla \bbeta \big)  \big\} \cdot \nabla = \omega\cdot \pa_{\vf} 
 	\end{equation}
 	for all $\omega\in \tD\tC(2\gamma,\tau)$, which shows \eqref{coniugazione nel teo trasporto}. Estimates \eqref{stima alpha trasporto}, \eqref{stima tame cambio variabile rid trasporto} follow by Lemma \ref{lemma.limite.trasporto}-$(ii)$ and Lemma \ref{lemma:LS norms}. \\
 	It remains to prove the estimates \eqref{stime delta 12 prop trasporto}. Let $\bb_{i}:= \bb_{i}(w_i)$, $i=1,2$ satisfy \eqref{divergenza.zero.a1} and assume that, for $s_1 > s_0$, $w_1,w_2$ satisfy \eqref{ansatz}, with $\sigma_0 \geq s_1 + \sigma $.  Let $\bbeta_{i}$ and $\cB_{i}$, $i=1,2$, be the corresponding functions and operators given by Lemma \eqref{lemma.limite.trasporto}-$(ii)$,$(iii)$. Then, by \eqref{robertino1001} and \eqref{coniugazione nel teo trasporto}, for $i=1,2$, we have, for any $\omega \in \tD\tC(2\gamma,\tau)$,
 	\begin{equation}\label{robertino.ciao}
 		\cB_{i}^{-1}  (\omega\cdot \pa_{\vf} + \bb_{i} \cdot \nabla) \cB_{i} = \omega\cdot \pa_{\vf}\,, \quad (\omega\cdot \pa_{\vf} + \bb_{i}\cdot \nabla)[\bbeta_{i}] + \bb_{i} = 0\,.
 	\end{equation}
 	Hence, with $\Delta_{12}\bbeta:= \bbeta_{1}(w_1)-\bbeta_{2}(w_2)$ and $\Delta_{12}\bb:= \bb_{1}(w_1)-\bb_{2}(w_2)$,
 	\begin{equation}
 		(\omega\cdot \pa_{\vf} + \bb_{1}(w_1) \cdot \nabla)[\Delta_{12}\bbeta] + \fg = 0\,,  \quad \fg:= \Delta_{12}\bb \cdot \nabla \bb_{2}(w_2) + \Delta_{12}\bb \,.
 	\end{equation}
 	By \eqref{robertino.ciao}, we have $\omega\cdot \pa_{\vf} + \bb_{1}(w_1)\cdot\nabla = \cB_{1}^{-1} \omega\cdot \pa_{\vf} \cB_{1}$, and therefore
 	\begin{equation}
 		\omega\cdot\pa_{\vf} \cB_{1}^{-1}\Delta_{12}\bbeta + \cB_{1}^{-1} \fg = 0\,.
 	\end{equation}
 	Since $\braket{\cB_{1}^{-1} \fg }_{\vf,x} = - \braket{\omega\cdot\pa_{\vf} \cB_{1}^{-1}\Delta_{12}\bbeta}_{\vf,x}=0$ for any $\omega \in\tD\tC(2\gamma,\tau)$ and $\Delta_{12}\bbeta$, $\cB_{1}^{-1}\Delta_{12}\bbeta = \odd(\vf,x)$ (by Lemma \ref{lemma.limite.trasporto}-$(ii)$,$(iii)$), one has
 	\begin{equation}
 		\Delta_{12}\bbeta = - \cB_{1} (\omega\cdot\pa_{\vf})^{-1} \cB_{1}^{-1} \fg \,, \quad \forall\,\omega\in \tD\tC(2\gamma,\tau) \,.
 	\end{equation}
 	Then, having $\| w_{i}\|_{s_1+\sigma}\leq 1$, by \eqref{stima tame cambio variabile rid trasporto}, \eqref{stima alpha trasporto} and Lemma \ref{lemma:LS norms}, recalling also that $\bb_{i} = \varepsilon \fB[w_i]$ by \eqref{varelambda}, \eqref{def coefficienti operatori linearized}, $i=1,2$, we get the estimate in \eqref{stime delta 12 prop trasporto} for $\Delta_{12}\bbeta$. The corresponding estimate for $\Delta_{12}\breve{\bbeta}$ can be proved by using that $\breve{\bbeta}_{i}=-\cB_{i}^{-1}\bbeta_{i}$ and the mean value theorem. Finally, the estimates for $\Delta_{12}\cB^{\pm1}$ follow by the estimates for $\Delta_{12}\bbeta$, $\Delta_{12}\breve{\bbeta}$, the mean value theorem and Lemma \ref{lemma:LS norms}. The proof is concluded.
 \end{proof}

We now compute the complete conjugation of the operator  $\cL: H_0^{s+1}(\T^{\nu+2})\to H_0^s(\T^{\nu+2})$ in \eqref{operatore linearizzato} by means of  the transformation in \eqref{diffeo.maps}, with $\bbeta(\vf,x)$ as in Proposition \ref{proposizione trasporto}.

\begin{prop}\label{prop coniugio cal L L1}
	Let $S>s_0+\sigma_{1}$, for some $\sigma_{1}=\sigma_{1}(\tau,\nu) \gg \sigma$ (where $\sigma $ is provided by Proposition \ref{proposizione trasporto}). There exists $\delta(S,\tau,\nu)\in(0,1)$ small enough such that, if \eqref{ansatz} and \eqref{condizione piccolezza rid trasporto} are fulfilled, the following holds:
	\\[1mm]
	\noindent $(i)$ Let $\cB_{\perp}:= \Pi_0^{\perp} \cB \Pi_0^{\perp}$, where $\Pi_0^{\perp}$ is defined in \eqref{definizione proiettore media spazio tempo} and the invertible map $\cB$, with inverse $\cB^{-1}$, is constructed in Proposition \ref{proposizione trasporto}. Then, for any $s_{0}  \leq s\leq S-\sigma$, with $\sigma$ as in Proposition \eqref{proposizione trasporto}, the operator $\cB_{\perp} : H^s_0(\T^{\nu + 2}) \to H^s_0(\T^{\nu + 2})$ is invertible with bounded inverse given by $\cB_{\perp}^{-1}:= \Pi_0^\perp \cB^{-1}\Pi_0^\perp : H_0^{s}(\T^{\nu+2})\to H_{0}^{s}(\T^{\nu+2})$; 
	\\[1mm]
	\noindent  $(ii)$  For any $\omega\in \tD\tC(2\gamma,\tau)$, as  in \eqref{DC.2gamma}, one has
	\begin{equation}\label{cL1}
		\cL_{1} : =\cB_{\perp}^{-1} \cL\, \cB_{\perp}  = \lambda\, \omega\cdot \pa_{\vf}  + \beta  \,\tL+ \cE_{1}\,,
	\end{equation}
	where, for any $s_0 \leq s \leq S - \sigma_1$,
	the operator $\cE_{1}\in \OpM_{s}^{-1}$ satisfies the estimate 
	\begin{equation}\label{stima mathcal E1}
		| \cE_{1} |_{-1,s}^{\Lip(\gamma)} \lesssim_{s} \lambda^{\theta}  \|  w \|_{s + \sigma_1}^{{\rm Lip}(\gamma)}\,.
	\end{equation}
		Let $s_1 \geq s_0$ and assume that $w_1, w_2$ 
	satisfy \eqref{ansatz} with $ \sigma_{0}  \geq s_1 + \sigma_{1} $. 
	Then for any $ \omega \in \tD\tC(2\gamma, \tau)$ one has 
	\begin{equation}\label{stime.delta12.fine.trasporto}
		\begin{aligned}
			& | \Delta_{12} \cE_{1} |_{-1,s_1}
			\lesssim_{s_1}  \lambda^{\theta} \| w_1 -w_2 \|_{s_1 +\sigma_{1}}\,.
		\end{aligned}
	\end{equation}
	Furthermore, the operators $\cL_{1}$ and $\cE_{1}$ are  momentum preserving and, if we assume that $ w(\vf,x)= \odd(\vf,x)$, they are reversibility preserving.
\end{prop}
\begin{proof}
	Item $(i)$ is proved in Lemma 4.2 in \cite{FrMo}. We now prove item $(ii)$.
	By \eqref{operatore linearizzato}, \eqref{varelambda}, we write $\cL = \cT + \beta\,\tL + \lambda^{\theta} \cE_{0}$. We have to compute
	\begin{equation}
		\begin{aligned}
			\cB_{\perp}^{-1} \cL \,\cB_{\perp}
			&=
			\cB_{\perp}^{-1} \cT \cB_{\perp}
			+ \beta\, \cB_{\perp}^{-1} \tL\, \cB_{\perp} +\lambda^{\theta} \cB_{\perp}^{-1} \cE_{0}\, \cB_{\perp} \,.
		\end{aligned}
	\end{equation}
	By Proposition \ref{proposizione trasporto}, by item $(i)$ and by \eqref{invarianze.2}, we have
	\begin{equation}
		\cB_{\perp}^{-1} \cT \cB_{\perp} = \Pi_0^\perp \cB^{-1} \cT \cB \Pi_0^{\perp} = \Pi_0^{\perp} \,\lambda\,\omega\cdot \pa_{\vf} \Pi_0^{\perp} =  \lambda\,\omega\cdot \pa_{\vf} {\rm Id}_0 
	\end{equation}
	 where ${\rm Id}_0$ is the identity on $L_0^2(\T^{\nu+2})$.
	To compute the other terms, the main tool is Lemma \ref{lemma.diffeo.Laplaciano}-$(ii)$, which provides the structure for the conjugation of the operator $(-\Delta)^{-1}$ as in \eqref{Laplacia.zero}. Recalling by \eqref{lambda.m.def} that $\tL=(-\Delta)^{-1} \pa_{x_1}$, which is invariant on the functions with zero average, we compute
\begin{equation}\label{coniugio.lambda}
	\begin{aligned}
		\cB_{\perp}^{-1} \tL\, \cB_{\perp} & = \Pi_0^{\perp}\big( \cB^{-1} (-\Delta)^{-1} \cB \big) \Pi_0^{\perp} \big( \cB^{-1} \pa_{x_1} \cB \big)\Pi_0^\perp \\
		& = \big( (-\Delta)^{-1} + \cP_{-2}  \big) \Pi_0^\bot \{ \cB^{-1}( 1+ \bbeta_{x_1} )\} \pa_{y_1}  \\
		& = \tL +   (-\Delta)^{-1} \Pi_0^\perp \{\cB^{-1}(\bbeta_{x_1})\} \pa_{y_1} +  \cP_{-2}\, \Pi_0^\perp \{\cB^{-1}( 1+ \bbeta_{x_1} )\} \pa_{y_1}     \,.
	\end{aligned}
\end{equation}
	In order to conjugate $\cE_{0}$ we note that, by its definition in \eqref{def coefficienti operatori linearized} and by \eqref{biot-savart}, we have
	\begin{equation}
		\cE_{0} = \nabla  w \cdot \fB =  w_{x_1} (-\Delta)^{-1} \pa_{x_2} -  w_{x_2} (-\Delta)^{-1} \pa_{x_1}\,.
	\end{equation}
	Using the same computation in \eqref{coniugio.lambda}, we deduce that $\cB_{\perp}^{-1}  (-\Delta)^{-1} \pa_{x_m}  \cB_{\perp} \in \OpM_{s}^{-1}$, $m=1,2$, and consequently, using also \eqref{invarianze cal L}, \eqref{invarianze.2},
	\begin{equation}\label{coniugio.E0}
		\begin{aligned}
			\cB_{\perp}^{-1} \cE_{0}\,\cB_{\perp}  & = \Pi_0^\perp \{\cB^{-1} ( w_{x_1})\} \cB_{\perp}^{-1} (-\Delta)^{-1} \pa_{x_2} \cB_{\perp} -  \Pi_0^\perp \{\cB^{-1} ( w_{x_2})\} \cB_{\perp}^{-1} (-\Delta)^{-1} \pa_{x_1} \cB_{\perp} \\ 
			& =  \Pi_0^\perp \{\cB^{-1}(\nabla  w)\} \cdot \cB_{\perp}^{-1} (-\Delta)^{-1} \nabla^\perp \cB_{\perp}
		\end{aligned}
	\end{equation}
	is in $\OpM_{s}^{-1}$. We conclude that $\cL_{1}:= \cB_{\perp}^{-1} \cL\,\cB_{\perp}$ has the form \eqref{cL1}, where $\cE_{1}\in \OpM_{s}^{-1}$, recalling \eqref{coniugio.lambda}, \eqref{coniugio.E0}, is given by
	\begin{equation}\label{E1.explicit}
	\begin{aligned}
		\cE_{1} & :=  \beta \big(  (-\Delta)^{-1} \Pi_0^\perp \{\cB^{-1}(\bbeta_{x_1})\} \pa_{y_1} +  \cP_{-2}\, \Pi_0^\perp \{\cB^{-1}( 1+ \bbeta_{x_1} )\} \pa_{y_1}    \big) \\
		&\quad  + \lambda^{\theta} \Pi_0^\perp \{\cB^{-1}(\nabla  w)\} \cdot \cB_{\perp}^{-1} (-\Delta)^{-1} \nabla^\perp \cB_{\perp} \,.
	\end{aligned}
	\end{equation}
	The estimate \eqref{stima mathcal E1} follows by \eqref{E1.explicit}, Lemma \ref{lemma:LS norms}, Lemma \ref{proprieta standard norma decay}-$(ii)$, estimate \eqref{stima Pi 0}, Lemma \ref{lemma.diffeo.Laplaciano} and Proposition \ref{proposizione trasporto} with the smallness condition \eqref{condizione piccolezza rid trasporto}. The estimate \eqref{stime.delta12.fine.trasporto} follows by similar arguments from the explicit expression \eqref{E1.explicit} and we omit the details. Finally, by Lemma \ref{lemma real rev matrici}, \ref{lem:mom_prop} and Proposition \ref{proposizione trasporto}, we obtain that $\cL_{1}$ is real and momentum preserving. If $ w$ is ${\rm odd}(\vf, x)$, then ${\mathcal L}_1$ is also reversible. This concludes the proof.
\end{proof}

\subsection{Reduction to perturbative of the large remainder }

The next goal is to conjugate the operator $\cL_{1}$ in \eqref{cL1} through a series of transformation in order to reduce both the size and the order of the remainder $\cE_{1}\in  \OpM_{s}^{-1}$. 
We shall prove the following Proposition.
\begin{prop}\label{prop normal form lower orders}
	Let $M\in\ \N$, $M > \frac{1-\tc}{2(1-\tc) - \alpha}$. There exist $\sigma_M= \sigma_{M}(\tau,\nu,M) \gg 0$ large enough such that, if $S>s_0 +\sigma_{M}$, there is $\delta \equiv \delta(S, M) \ll 1$ small enough such that, if $\lambda^{\theta-1} \gamma^{- 1} \leq \delta$ and if \eqref{ansatz} is fullfilled with $\sigma_{0} = \sigma_M$, then the following hold. For any $\omega \in \tD\tC(2\gamma, \tau)$, there exists a real, invertible operator ${\bf \Phi}_M \in \OpM_s^0$, $s_0 \leq s \leq S - \sigma_M$ satisfying 
	\begin{equation}\label{stima bf Phi M}
		\begin{aligned}
			& |{\bf \Phi}_M^{\pm 1}|_{0, s}^{\Lip(\gamma)} \lesssim_{s, M} 1 + \| {w} \|_{s + \sigma_M}^{\Lip(\gamma)} \,, \quad \forall s_0 \leq s \leq S - \sigma_M\,, \\
			&  |{\bf \Phi}_M^{\pm 1}-{\rm Id}|_{-M, s}^{\Lip(\gamma)} \lesssim_{s, M} \| {w} \|_{s + \sigma_M}, \,
		\end{aligned}
	\end{equation}
	such that for any $\omega \in {\mathtt D}{\mathtt C}(2\gamma, \tau)$
	\begin{equation}\label{forma cal LM finale}
		{\mathcal L}_M := {\bf \Phi}_M^{- 1} {\mathcal L}_1 {\bf \Phi}_M =  \lambda \, \omega \cdot \partial_\vphi + \beta\,\tL + {\mathcal Z}_M + {\mathcal E}_M\,,
	\end{equation}
	where ${\mathcal Z}_M \in \OpM^{- 1}_s$, $s \geq 0$, is a diagonal operator ${\mathcal Z}_M := {\rm diag}_{j \neq 0} z_M(j)$ and ${\mathcal E}_M \in \OpM^{- M}_s$, $s_0 \leq s \leq S - \sigma_M$ satisfy the estimates 
	\begin{equation}\label{stime induttive cal EM cal ZM finale}
		\begin{aligned}
			& |{\mathcal Z}_M|_{- 1, s}^{\Lip(\gamma)} \lesssim_{M} \lambda^{\theta}\,, \quad \forall s \geq s_0\,, \quad  \sup_{j \neq 0} |j| |z_M(j)|^{\Lip(\gamma)} \lesssim_M \lambda^{\theta}\,, \\
			&  |{\mathcal E}_M|_{- M, s}^{\Lip(\gamma)} \lesssim_{s, M} \big( \lambda^{\theta-1}\gamma^{-1} \big)^{M-1} \lambda^{\theta}  \| {w} \|_{s + \sigma_M}^{\Lip(\gamma)} \,, \quad \forall s_0 \leq s \leq S - \sigma_M\,. 
		\end{aligned}
	\end{equation}
	%
	%
	%
	Let $s_1 \geq s_0$ and assume that $ w_1, w_2$ satisfy \eqref{ansatz}  with $ \sigma_{0}  \geq s_1 + \sigma_{M} $. 
	Then, for any $\omega \in \tD\tC(2\gamma,\tau)$, one has
	\begin{equation}\label{stime.delta12 induttive cal EM cal ZM finale}
		\begin{aligned}
			& |\Delta_{12} \b\Phi_{M}^{\pm 1} |_{-M,s_1} \lesssim_{s_1, M} \|  w_1 -  w_2 \|_{s_1 + \sigma_{M}}\,,\\
			& | \Delta_{12} \cE_{M}|_{-M,s_1}\lesssim_{s_1, M}  \big( \lambda^{\theta-1}\gamma^{-1} \big)^{M-1} \lambda^{\theta}  \|   w_1 -  w_2  \|_{s_1 + \sigma_{M}}\,.
		\end{aligned}
	\end{equation}
	Furthermore, the operators $\b\Phi_{M}^{\pm 1}$ and $\cL_{M}$ are real and momentum preserving. In addition, if we assume that $ w(\vf,x)=\odd(\vf,x)$, then the operators $\b\Phi_{M}^{\pm 1}$ are reversibility preserving and $\cL_{M}$ is reversible.
\end{prop}

Proposition \ref{prop normal form lower orders} is proved as a consequence of the following iterative procedure.

\begin{lem}\label{lemma.large.to.pert}
	Let $M\in\ \N$, $M > \frac{1-\tc}{2(1-\tc) - \alpha} $. There exist $\sigma_1 < \sigma_2 < \ldots < \sigma_M$ with $\sigma_{i}:=\sigma_{i}(\tau,\nu)>0$ large enough such that, for any $S>s_0 +\sigma_{M}$, there exists $\delta \equiv \delta(S, M) \ll 1$ small enough such that, if $\lambda^{\theta-1} \gamma^{- 1} \leq \delta$ and if \eqref{ansatz} is fullfilled with $\sigma_{0} = \sigma_M$, then the following holds. For any $m = 1, \ldots, M$, there exists a real operator ${\mathcal L}_m$ of the form
	\begin{equation}\label{op mathcal Lm}
		{\mathcal L}_m = \lambda \, \omega \cdot \partial_\vphi + \beta \,\tL + {\mathcal Z}_m + {\mathcal E}_m\,,
	\end{equation}
	where,  for any $s_0 \leq s \leq S - \sigma_m$, 
${\mathcal Z}_m \in \OpM^{- 1}_s$ is the diagonal operator ${\mathcal Z}_m = {\rm diag}_{j \neq 0} z_m(j)$ , ${\mathcal E}_m \in \OpM^{- m}_s$ and they satisfy the estimates
	\begin{equation}\label{stime induttive cal Em cal Zm}
	\begin{aligned}
		& |{\mathcal Z}_m|_{- 1, s}^{\Lip(\gamma)} \lesssim_m \lambda^{\theta} \,,  \quad  \sup_{j \neq 0} |j| |z_m(j)|^{\Lip(\gamma)} \lesssim_m \lambda^{\theta}\,,\\
		&  |{\mathcal E}_m|_{- m, s}^{\Lip(\gamma)} \lesssim_{s, m} \big( \lambda^{\theta-1} \gamma^{- 1} \big)^{m - 1} \lambda^{\theta} \| {w} \|_{s + \sigma_m}^{\Lip(\gamma)} \,.
	\end{aligned}
\end{equation}
	There exist real operators $\{ \Phi_{m} \}_{m=1}^{M - 1}$, with $\Phi_{m}:= {\rm exp} (\cX_{m})\in \OpM_{s}^{0}$, where
	 $$\cX_{m} := \cX_{m}(\vf) = \sum_{\ell\in\Z^{\nu}\setminus\{0\}} \wh\cX_{m}(\ell)e^{\im\,\ell\cdot\vf} \in \OpM_{s}^{-m}\,,$$ 
	satisfying the estimates 
	\begin{equation}\label{stime cal Xm Phim nel lemma}
		\begin{aligned}
			& |\cX_m|_{- m, s}^{\Lip(\gamma)} \lesssim_{s, m} \big(\lambda^{\theta-1} \gamma^{- 1} \big)^m	\| {w} \|_{s + \sigma_m}^{\Lip(\gamma)}, \quad \forall s_0 \leq s \leq S - \sigma_m\,, \\
			& |\Phi_m^{\pm 1}|_{0, s}^{\Lip(\gamma)} \lesssim_{s, m} 1 + \| {w} \|_{s + \sigma_{m}}^{\Lip(\gamma)}, \quad \forall s_0 \leq s \leq S - \sigma_m\,. 
		\end{aligned}
	\end{equation}
	Moreover, for any $m = 2, \ldots, M - 1$ and any $\omega\in\tD\tC(2\gamma,\tau)$, we have that 
	\begin{equation}\label{operator.pertubative}
		\cL_{m} := \Phi_{m-1}^{-1} \cL_{m-1} \Phi_{m-1} \,.
	\end{equation}
	Let $s_1 \geq s_0$ and assume that $w_1,w_2$ satisfy \eqref{ansatz}  with $ \sigma_{0}  \geq s_1 + \sigma_{m} $. 
	Then, for any $\omega \in \tD\tC(2\gamma,\tau)$, we have
	\begin{equation}\label{stime induttive delta12 cal Em cal Phim}
		\begin{aligned}
			& |\Delta_{12} \Phi_{m}^{\pm 1} |_{-m ,s_1} \lesssim_{s_1, m} \|  w_1 -  w_2 \|_{s_1 + \sigma_{m}} \,,\\
			& | \Delta_{12} \cE_{m}|_{-m,s_1}\lesssim_{s_1, m} \big( \lambda^{\theta-1} \gamma^{- 1} \big)^{m - 1} \lambda^{\theta} \|  w_1 -  w_2 \|_{s_1 + \sigma_{m}}\,.
		\end{aligned}
	\end{equation}
	Furthermore, the operators $\b\Phi_{m}^{\pm 1}$ and $\cL_{m}$ are real and momentum preserving. In addition, if we assume that $ w(\vf,x)=\odd(\vf,x)$, then the operators $\b\Phi_{m}^{\pm 1}$ are reversibility preserving and $\cL_{m}$ is reversible.
\end{lem}
\begin{proof}
	We proceed by induction.
	The operator $\cL_{1}$ in \eqref{cL1} is of the form \eqref{op mathcal Lm} with $\cZ_{1}=0$.
	By \eqref{stima mathcal E1} and \eqref{stime.delta12.fine.trasporto} we have that $\cE_{1}$ satisfies \eqref{stime induttive cal Em cal Zm} and \eqref{stime induttive delta12 cal Em cal Phim}.
	Finally, by Proposition \ref{prop coniugio cal L L1} we have that $\cL_{1}$ is reversible and momentum preserving.

	We now assume that the claimed statements hold for $m = 1, \ldots, M - 1$ and we prove them for $m + 1$. 
	We look for a trasformation of the form 
	\begin{equation}
		\Phi_{m}:= {\rm exp}(\cX_m) \in \OpM_{s}^{0}\,, \quad \cX_{m} := \cX_{m}(\vf) = \sum_{\ell\in\Z^\nu\setminus\{0\}} \wh\cX_{m}(\ell) e^{\im\,\ell\cdot\vf} \in \OpM_{s}^{-m}\,,
	\end{equation} 
	where $\wh\cX_{m}$ has to be determined. 
	To compute ${\mathcal L}_{m + 1} = \Phi_m^{- 1} {\mathcal L}_m \Phi_m$, we look separately at the conjugation of the three terms appearing in \eqref{op mathcal Lm}. By the standard Lie expansion, we get
	\begin{equation}\label{coniugio a pezzi cal L1 Phi1 m}
		\begin{aligned}
			\Phi_m^{- 1} \lambda\, \omega \cdot \partial_\vphi \Phi_m & = \lambda\, \omega \cdot \partial_\vphi + \lambda \, \omega \cdot \partial_\vphi \cX_m(\vphi) + {\mathcal Q}_m^{(1)}(\vphi) \,, \\
			{\mathcal Q}_m^{(1)}(\vphi) & = \int_0^1 (1 - \tau ) {\rm exp}(- \tau \cX_m) [\lambda\, \omega \cdot \partial_\vphi \cX_m(\vphi), \cX_m(\vphi)]  {\rm exp}( \tau \cX_m) \wrt \tau\,, \\
			\Phi_m^{- 1} ( \beta \tL) \Phi_m & =  \beta \tL + {\mathcal Q}_m^{(2)}\,, \\
			{\mathcal Q}_m^{(2)}(\vphi) & :=  \int_0^1 {\rm exp}(- \tau \cX_m(\vphi)) [\beta \tL, \cX_m(\vphi)] {\rm exp}(\tau \cX_m(\vphi)) \wrt \tau\,, \\
			\Phi_m^{- 1} {\mathcal Z}_m \Phi_m & = {\mathcal Z}_m + {\mathcal Q}_m^{(3)}\,, \\
			{\mathcal Q}_m^{(3)}(\vphi) & := \int_0^1 {\rm exp}(- \tau \cX_m(\vphi)) [{\mathcal Z}_m, \cX_m(\vphi)] {\rm exp}(\tau \cX_m(\vphi)) \wrt \tau \,, \\
			\Phi_m^{- 1} {\mathcal E}_m \Phi_m & = {\mathcal E}_m  + {\mathcal Q}_m^{(4)}\,, \\
			{\mathcal Q}_m^{(4)}(\vphi) & :=  \int_0^1 {\rm exp}(- \tau \cX_m(\vphi)) [{\mathcal E}_m(\vphi), \cX_m(\vphi)] {\rm exp}(\tau \cX_m(\vphi)) \wrt \tau\,.\\
		\end{aligned}
	\end{equation}
	The terms of order $- m$ in the expansion of ${\mathcal L}_{m + 1}= \Phi_m^{- 1} {\mathcal L}_m \Phi_m$ that we want to reduce are then given by $\lambda\, \omega \cdot \partial_\vphi \cX_m(\vphi) + {\mathcal E}_m(\vphi)$. Hence, we solve the homological equation 
	\begin{equation}\label{equazione omologica step - 1 m}
		\lambda \, \omega \cdot \partial_\vphi \cX_m(\vphi) + {\mathcal E}_m(\vphi) = \widehat{\mathcal E}_m(0)\,, \quad \widehat{\mathcal E}_m(0) := \frac{1}{(2 \pi)^\nu} \int_{\T^\nu} {\mathcal E}_m(\vphi) \wrt \vphi \,,
	\end{equation}
	by defining, for any $\omega\in \tD\tC(2\gamma,\tau)$
	\begin{equation}\label{def cal X 1 m}
		{\mathcal X}_m(\vphi) = - \sum_{\ell \neq 0} \frac{\widehat{\mathcal E}_m(\ell)}{\im \lambda\, \omega \cdot \ell} e^{\im \ell \cdot \vphi}\,.
	\end{equation}
	Hence the matrix elements of ${\mathcal X}_m$ are given by 
	$$
	\widehat{\mathcal X}_m(\ell)_j^{j'} = \begin{cases}
	- \dfrac{\widehat{\mathcal E}_m(\ell)_j^{j'}}{\im \lambda \,\omega \cdot \ell}\,,  &\ell \neq 0, \quad j, j' \in \Z^2 \setminus \{ 0 \}, \quad \pi^\top(\ell) + j - j' = 0\,, \\
	0  &\text{otherwise.}
	\end{cases}
	$$
	Then for $\omega \in \tD\tC(2\gamma,\tau)$, one has that 
	$$
	|\widehat{\mathcal X}_m(\ell)_j^{j'}| \lesssim \gamma^{- 1} \lambda^{- 1} \langle \ell \rangle^\tau |\widehat{\mathcal E}_m(\ell)_j^{j'}|
	$$
	and if $\omega_1, \omega_2 \in \tD\tC(2\gamma,\tau)$, one has 
	$$
	\begin{aligned}
	|\widehat{\mathcal X}_m(\ell; \omega_1)_j^{j'} - \widehat{\mathcal X}_m(\ell; \omega_2)_j^{j'}| & \lesssim  \langle \ell \rangle^{2 \tau + 1} \gamma^{- 2} \lambda^{- 1} |\widehat{\mathcal E}_m(\ell; \omega_2)_j^{j'}| |\omega_1 - \omega_2|  \\
	& \quad + \langle \ell \rangle^\tau \gamma^{- 1} \lambda^{- 1} |\widehat{\mathcal E}_m(\ell; \omega_1)_j^{j'} - \widehat{\mathcal E}_m(\ell; \omega_2)_j^{j'}|\,.
	\end{aligned}
	$$
	By taking $S > s_0 + \sigma_{m} + 2 \tau + 1$, the latter two estimates, together with the Definition \ref{block norm} imply that $ {\mathcal X}_m \in \OpM_s^{- m}$ for any $s_0 \leq s \leq S - \sigma_{m} - 2 \tau - 1$ and 
	\begin{equation}\label{prop + stime cal X1 m}
		\begin{aligned}
			& |{\mathcal X}_m|_{- m,  s}^{\Lip(\gamma)} \lesssim_s \gamma^{- 1} \lambda^{- 1} |{\mathcal E}_m|_{- m, s + 2 \tau + 1}^{\Lip(\gamma)} \\
			& \qquad \qquad  \stackrel{\eqref{stime induttive cal Em cal Zm}}{\lesssim_{s, m}} \lambda^{- 1} \gamma^{- 1} \big( \lambda^{\theta-1} \gamma^{- 1} \big)^{m - 1} \lambda^{\theta} \| {w} \|_{s + \sigma_m + 2 \tau + 1}^{\Lip(\gamma)} \\
			& \qquad \qquad \lesssim_{s, m}   \big( \lambda^{\theta-1} \gamma^{- 1} \big)^{m } \| {w} \|_{s + \sigma_m + 2 \tau + 1}^{\Lip(\gamma)}\,.
		\end{aligned}
	\end{equation}
	Together with Lemma \ref{proprieta standard norma decay}-$(iv)$ and using that $\lambda^{\theta-1} \gamma^{- 1} \ll 1$, we have that \eqref{prop + stime cal X1 m} implies, for any $s_0 \leq s \leq S - \sigma_m - 2 \tau - 1$,
	\begin{equation}\label{Prop Phi 1 Phi 1 inv m}
		\sup_{\tau \in [- 1, 1]}|{\rm exp}(\tau \cX_m)|_{0, s}^{\Lip(\gamma)}  \lesssim_{s, m} 1 + \| {w} \|_{s + \sigma_m + 2 \tau + 1}^{\Lip(\gamma)} \,. 
	\end{equation}
	By \eqref{coniugio a pezzi cal L1 Phi1 m}, we obtain that
	\begin{equation}\label{forma finale cal L2 m}
		\begin{aligned}
			{\mathcal L}_{m + 1} & = \Phi_m^{- 1} {\mathcal L}_m \Phi_m = \lambda\,\omega \cdot \partial_\vphi + \beta\, \tL + {\mathcal Z}_{m + 1}  + {\mathcal E}_{m + 1} \,, \\
			{\mathcal Z}_{m + 1} & := {\mathcal Z}_m + \widehat{\mathcal E}_m(0)\,, \\
			{\mathcal E}_{m + 1}  & := {\mathcal Q}_m^{(1)} +  {\mathcal Q}_m^{(2)} +   {\mathcal Q}_m^{(3)} +  {\mathcal Q}_m^{(4)}\,.  
		\end{aligned}
	\end{equation}
	{\sc Properties and estimates of ${\mathcal Z}_{m + 1}$.} 
	Since ${\mathcal E}_m$ is momentum preserving, by Lemma \ref{lem:mom_pres} we have that the time-independent operator $\widehat{\mathcal E}(0)$ is diagonal and hence also
	the operator ${\mathcal Z}_{m+1} = {\rm diag}_{j \neq 0} z_{m+1}(j)$ is a diagonal operator with $z_{m + 1}(j) :=  z_m(j) + \widehat{\mathcal E}_m(0)_j^{j}$, $j \in \Z^2 \setminus \{ 0 \}$. Furthermore, by Lemma \ref{proprieta standard norma decay}-$(v)$ and by the induction estimates \eqref{stime induttive cal Em cal Zm} (using also the ansatz \eqref{ansatz}), we get that  the operator ${\mathcal Z}_{m + 1} \in \OpM^{- 1}_s$ for any $s \geq s_0$ and it satisfies, using also $\lambda^{\theta-1} \gamma^{- 1} \ll 1$ by \eqref{piccolo.ansatx},
	\begin{equation}\label{stima cal Z1 m}
		\begin{aligned}
			|{\mathcal Z}_{m + 1}|_{- 1, s}^{\Lip(\gamma)} &\lesssim |{\mathcal Z}_m|_{- 1, s}^{\Lip(\gamma)} + |\widehat{\mathcal E}_m(0)|_{- m, s} \lesssim |{\mathcal Z}_m|_{- 1, s}^{\Lip(\gamma)}  +  |{\mathcal E}_m|_{- m, s_0}^{\Lip(\gamma)} \\
			&  \lesssim_m \lambda^{\theta} + \big( \lambda^{\theta-1} \gamma^{- 1}\big)^{m - 1} \lambda^{\theta} \lesssim_m \lambda^{\theta}\,, 
		\end{aligned}
	\end{equation}
	which implies that
	\[
	\sup_{j \neq 0} |j| |z_{m + 1}(j)|^{\Lip(\gamma)} \lesssim_m \lambda^{\theta}\,. 
	\]
	
	\medskip
	
	\noindent
	{\sc Properties and estimates of ${\mathcal E}_{m + 1}$.} We now estimate the remainder ${\mathcal E}_{m + 1}$ in \eqref{forma finale cal L2 m}. We have to analyze the four terms ${\mathcal Q}_m^{(1)}, {\mathcal Q}_m^{(2)}, {\mathcal Q}_m^{(3)}, {\mathcal Q}_m^{(4)}$ in \eqref{coniugio a pezzi cal L1 Phi1 m}. First, by \eqref{equazione omologica step - 1 m}, we note that
	\begin{equation}\label{seconda forma cal Q 1 (1) m}
		{\mathcal Q}_m^{(1)} = \int_0^1 (1 - \tau ) {\rm exp}(- \tau \cX_m) [{\mathcal Z}_{m + 1} - {\mathcal E}_m, \cX_m]  {\rm exp}( \tau \cX_m) \wrt \tau\,.
	\end{equation}
	By the estimates \eqref{stime induttive cal Em cal Zm}, \eqref{prop + stime cal X1 m}, \eqref{stima cal Z1 m}, using also \eqref{stima.mathttL}  Lemma \ref{proprieta standard norma decay}-$(ii)$  the ansatz \eqref{ansatz}, we get that, taking  $S > s_0 + \sigma_m + m + 2 \tau + 1$, the operators $ [{\mathcal Z}_{m + 1} , \cX_m] \,,[\tL, \cX_m] \,,\, \in \OpM_s^{- m - 1},  [{\mathcal E}_m , \cX_m] \in \OpM_s^{- 2 m} {\subseteq} \OpM^{- m - 1}_s$ for any $ s_0 \leq s \leq S - \sigma_m  - m- 2 \tau - 1$ and they satisfy the following estimates, recalling also that $\lambda^{\theta-1} \gamma^{- 1} \ll 1$ by \eqref{piccolo.ansatx}:
	\begin{equation}\label{stime cal Q 1 1 3 m}
		\begin{aligned}
			| [{\mathcal Z}_{m + 1} , \cX_m]|_{- m - 1, s}^{\Lip(\gamma)} 
			&\lesssim_{s, m}  \lambda^{\theta} \big( \lambda^{\theta-1} \gamma^{- 1} \big)^m  \| {w} \|_{s + \sigma_m + m + 2 \tau + 1}^{\Lip(\gamma)}\,, \\
			|[{\mathcal E}_m , \cX_m]|_{- m - 1 , s}^{\Lip(\gamma)} 
			&\leq |[{\mathcal E}_m , \cX_m]|_{- 2m , s}^{\Lip(\gamma)}  \\
			&\lesssim_{s, m} \lambda^{\theta} \big( \lambda^{\theta-1} \gamma^{- 1}\big)^{m - 1} \big( \lambda^{\theta-1} \gamma^{- 1}\big)^{m} \| {w} \|_{s + \sigma_m + m + 2 \tau + 1}^{\Lip(\gamma)} \\
			& \lesssim_{s, m} \lambda^{\theta}  \big( \lambda^{\theta-1} \gamma^{- 1}\big)^{m} \| {w} \|_{s + \sigma_m + m + 2 \tau + 1}^{\Lip(\gamma)}\,, \\
			|[\tL, \cX_m]|_{- m - 1, s}^{\Lip(\gamma)}
			& \lesssim_{s, m} \big( \lambda^{\theta-1} \gamma^{- 1}\big)^{m} \| {w} \|_{s + \sigma_m + 2 \tau + 1}^{\Lip(\gamma)} \\
			& \lesssim_{s, m} \lambda^{\theta} \big( \lambda^{\theta-1} \gamma^{- 1}\big)^{m} \| {w} \|_{s + \sigma_m + 2 \tau + 1}^{\Lip(\gamma)}\,. 
		\end{aligned}
	\end{equation}
	Together with \eqref{Prop Phi 1 Phi 1 inv m}, and Lemma \ref{proprieta standard norma decay}-$(ii)$, the latter estimates imply that, for $\sigma_{m + 1} \geq \sigma_m + m + 2 \tau + 2$, $\| {w} \|_{s_0 + \sigma_{m + 1}}^{\Lip(\gamma)} \lesssim 1$ and $S > s_0 + \sigma_{m + 1}$, the operator ${\mathcal E}_{m + 1} \in \OpM_s^{- m - 1}$ for any $s_0 \leq s \leq S - \sigma_{m + 1}$ and it satisfies
	\begin{equation}\label{stima mathcal E2 m}
		\begin{aligned}
			&|{\mathcal E}_{m + 1}|_{- m - 1, s}^{\Lip(\gamma)} \lesssim_{s, m} \lambda^{\theta} \big(  \lambda^{\theta-1} \gamma^{- 1} \big)^m \| {w} \|_{s + \sigma_{m + 1}}^{\Lip(\gamma)}\,.
		\end{aligned}
	\end{equation}
	The estimates in \eqref{stime induttive delta12 cal Em cal Phim} follow by similar arguments and we omit the details.
	By Lemma \ref{lem:mom_prop} and the inductive hypotheses on $\cL_{m}$ and $\Phi_{m}$, we obtain that $\Phi_{m+1}, {\mathcal L}_{m + 1}$ are  momentum preserving. If $ w = {\rm odd}(\vf, x)$ then by Lemma \ref{lemma real rev matrici}, $\Phi_m$ is reversibility preserving and ${\mathcal L}_{m + 1}, {\mathcal E}_{m + 1}, {\mathcal Z}_{m + 1}$ are reversible operators. 
	The proof of the claimed statements at the step $m + 1$ is then concluded. 
\end{proof}
\begin{proof}[Proof of Proposition \ref{prop normal form lower orders}.]  Let ${\bf\Phi}_{1}:={\rm Id}$ and ${\bf \Phi}_M := \Phi_{1} \circ \Phi_{2} \circ \ldots \circ \Phi_{M - 1}$ for $M\geq 2$. By Lemma \ref{lemma.large.to.pert},  Lemma \ref{proprieta standard norma decay}-$(ii)$ and the estimates \eqref{stime cal Xm Phim nel lemma}, \eqref{stime induttive delta12 cal Em cal Phim}, we obtain that $\b\Phi_{M}$ satisfies the estimates \eqref{stima bf Phi M}, \eqref{stime.delta12 induttive cal EM cal ZM finale}. Moreover, if $\omega \in {\mathtt D}{\mathtt C}(2\gamma, \tau)$, then the conjugation \eqref{forma cal LM finale} holds with ${\mathcal Z}_M$ satisfying the estimate in \eqref{stime induttive cal EM cal ZM finale}, using also Lemma \ref{proprieta standard norma decay}-$(v)$, and ${\mathcal E}_M$ satisfying the estimate , for any $s_0 \leq s \leq S - \sigma_M$,
	$$
	|{\mathcal E}_M|_{- M, s}^{\Lip(\gamma)} \lesssim_{s, M}  \big( \lambda^{\theta-1} \gamma^{- 1} \big)^{M - 1} \lambda^{\theta} \| {w} \|_{s + \sigma_M}^{\Lip(\gamma)} \,.
	$$
	Then, the desired bound on ${\mathcal E}_M$  in \eqref{stime induttive cal EM cal ZM finale}  follows since, using $M > \frac{1-\tc}{2(1-\tc) - \alpha} = \frac{1-\tc}{1-\tc -\theta} $, $\lambda \gg 1$ large enough, and recalling \eqref{theta.def.ridu}, \eqref{small.theta}, we have, if we assume $\gamma=\lambda^{-\tc}$,
	$$
	\big( \lambda^{\theta-1} \gamma^{- 1} \big)^{M - 1} \lambda^{\theta} = \lambda^{M\theta -(M-1)(1-\tc)}  < 1\,.
	$$
	Finally, by Lemma \ref{lem:mom_prop} and Proposition \ref{prop coniugio cal L L1}, we obtain that $\b\Phi_{M}$, ${\mathcal L}_M$ are momentum preserving. Moreover, if $ w = {\rm odd}(\vf, x)$, then by Lemma \ref{lemma real rev matrici}, $\b\Phi_{M}$ is reversibility preserving and ${\mathcal L}_M$ is reversible. 
	The proof of the Proposition is then concluded. 
\end{proof}

 \subsection{KAM perturbative reduction}\label{sez:iteraKAM}
 
 We are now in position to reduce the operator $\cL_{M}$ in \eqref{operator.pertubative} with a KAM reducibility iteration. To the purposes of the KAM reducibility, we define
 \begin{equation}\label{bL0_inizioKAM}
 	\bL_{0} : =\cL_{M} = \lambda\,\omega\cdot \pa_{\vf} + \bD_{0}  + \bE_{0}
 \end{equation}
 where:
 \\[1mm]
 \noindent $\bullet$ The diagonal operator $\bD_{0}$ is given by
 \begin{equation}
 	\begin{aligned}
 	 &	\bD_{0}:= \beta \,\tL + \bZ_{0} = \diag_{j\in\Z^2\setminus\{0\}} \mu_{0}(j)\,, \quad \mu_{0}(j) = \im\,\beta \tL(j) + \tz_{0}(j) \,, \\
 	& \bZ_{0} := \cZ_{M} := \diag_{j\in\Z^2\setminus\{0\}} \tz_{0}(j)\,, \quad \tz_{0} (j) := z_{M}(j) \,, \  \ j \in \Z^2\setminus\{0\}\,, 
 	\end{aligned}
  \end{equation}
  with $\tL(j)\in \R$ and $\tz_0 ( j )\in \C$ defined respectively in \eqref{lambda.m.def} and in Proposition \ref{prop normal form lower orders} (if we assume reversibility conditions, then $\tz_{0}(j)\in \im\R$);
  \\[1mm]
  \noindent $\bullet$ For any $s\in [s_0,S-\sigma_{M}]$, the operator $\bE_{0}:= \cE_{M} \in \OpM_{s}^{- M}$ satisfies the estimate (see \eqref{stime induttive cal EM cal ZM finale})
  \begin{equation}
  	| \bE_{0} |_{-M,s}^{\Lip(\gamma)} \lesssim_{s, M} \lambda^{M(\theta-1)+1} \gamma^{-(M-1)}  \| w \|_{s+\sigma_{M}} 
  \end{equation}
  and if $s_1 \geq s_0$, $ w_1, w_2$ satisfy \eqref{ansatz}  with $ \sigma_{0}  \geq s_1 + \sigma_{M} $, then (see \eqref{stime.delta12 induttive cal EM cal ZM finale})
  \begin{equation}
  |\Delta_{12} {\bf E}_0|_{- M, s_1} \lesssim_{s_1, M} \lambda^{M(\theta-1)+1} \gamma^{-(M-1)}  \|   w_1 -  w_2  \|_{s_1 + \sigma_{M}}\,.
  \end{equation}
  Note that, with $M > \frac{1-\tc}{2(1-\tc)-\alpha}$ and assuming $\gamma=\lambda^{-\tc}$, recalling also \eqref{theta.def.ridu}, \eqref{small.theta}, the operator $\bE_{0}$ has small size with respect to $\lambda \gg 1$ large enough.   
   
 Given $\tau, \tN_0 > 0$, we fix the constants 
 \begin{equation}\label{definizione.param.KAM} 
 	\begin{aligned}
 		& 
 		M :=  \max\{2 \tau, \tfrac{1-\tc}{2(1-\tc)-\alpha} \}\ + 1,  \quad 
 		\ta := {\rm max}\{3(2 \tau\! + \!M \!+ 1) + 1\,,\, \chi(\tau + \tau^2 + 2) \} \,, \quad  \tb:=\! \ta + 1 \,, 
 		 \\
 		& \tau_{1}:= 4\tau +2 +M \,, \quad   \Sigma(\tb):= \sigma_{M} +\tb\,, \quad S > s_0 + \Sigma(\tb)\,, 
 		 \\
 		&   
 		\tN_{- 1} := 1 \,, \quad 
 		\tN_n := \tN_0^{\chi^n}\,, \quad 
 		n \geq 0\,, \quad 
 		\chi := 3/2  \,, 
 	\end{aligned}
 \end{equation}
 where 
 $M$, $\sigma_{M }$ are introduced in Proposition \ref{prop normal form lower orders}.
 
 \begin{rem}\label{remark.su.param.KAM}
 	Let us describe of the parameters introduced in \eqref{definizione.param.KAM} and their role in the following Proposition \ref{prop riducibilita}. The parameter $\ta>0$ appears in the negative exponent of the first estimate in \eqref{stime cal Rn rid} and measures how the low regularity norm of the remainder is fastly decreasing at each iteration. The parameter $\tb>0$ appears in the regularity index $s+\tb$ of the second estimate in \eqref{stime cal Rn rid}, which is the high regularity norm of the remainder that is allowed to diverge. The parameter $\tau_1>0$ appears as an exponent in the smallness condition \eqref{KAM smallness condition}. The parameter $\Sigma(\tb)>0$ accounts for the loss of derivatives of the KAM iteration. 
 \end{rem}
 
 By Proposition \ref{prop normal form lower orders}, replacing $s$ by $s + \tb$ in  \eqref{stime induttive cal EM cal ZM finale} and having $ \bZ_{0} = {\rm diag}_{j \in \Z^2 \setminus \{ 0 \}} \tz_{0}(j)$  diagonal, and by the ansatz \eqref{ansatz}, one gets the initialization conditions for the KAM reducibility, for any $s_0\leq s\leq S-\Sigma(\tb)$,
 \begin{equation}\label{stime pre rid}
 \begin{aligned}
 & 	\sup_{j \in \Z^2 \setminus \{ 0 \}} |j| |\tz_{0}(j)|^{{\rm Lip}(\gamma)} \lesssim  \lambda^{\theta}\,,  \\
 &    |\bE_{0}|_{- M , s + \tb}^{{\rm Lip}(\gamma)} \lesssim_{s}   \lambda^{M(\theta-1)+1} \gamma^{-(M-1)}  \|  w \|_{s + \Sigma(\tb)}^{\Lip(\gamma)}\,, \\
  &    |\bE_{0}|_{- M , s_0 + \tb}^{{\rm Lip}(\gamma)} \lesssim_{s} \lambda^{M(\theta-1)+1} \gamma^{-(M-1)} \,, \\
  & \text{and if} \quad  w_1, w_2 \quad \text{ satisfy \eqref{ansatz}  with} \quad  \sigma_{0}  = \Sigma(\tb)  \quad  \text{then} \\
  & |\Delta_{12} {\bf E}_0 |_{- M, s_0 + \mathtt b} \lesssim  \lambda^{M(\theta-1)+1} \gamma^{-(M-1)}  \|  w_1 -  w_2 \|_{s_0 + \Sigma(\mathtt b)}\,. 
	\end{aligned}
 \end{equation}
 By the definition of $M$ in \eqref{definizione.param.KAM}, one has that $M>\frac{1-\tc}{2(1-\tc)-\alpha}> \frac{1}{1-\theta}$. We work in the regime $\lambda\gg 1$ and recalling \eqref{theta.def.ridu}, \eqref{small.theta}, we define the small parameters (recalling \eqref{varelambda})
 \begin{equation}\label{vare.vare}
 	\varepsilon:= \lambda^{\theta-1}\,, \quad \varepsilon_{M} := \lambda^{M(\theta-1)+1} = \varepsilon^{M-1} \lambda^{\theta} \,.
 \end{equation}
 
 \begin{prop}[\bf Reducibility]\label{prop riducibilita}
 	Let $S > s_0 + \Sigma(\tb)$, with the notation of \eqref{definizione.param.KAM}. There exist $\tN_{0} := \tN_{0}(S, \tau,\nu) > 0$ large enough and $\delta := \delta(S, \tau,\nu) \in (0, 1)$ small enough such that, if \eqref{ansatz} holds with $\sigma = \Sigma(\tb)$ and
 	\begin{equation}\label{KAM smallness condition}
	\tN_{0}^{\tau_1}  \varepsilon^{M}\gamma^{-M} \leq \delta \,,
\end{equation} 
 	then the following statements hold for any integer $n \geq 0$.
 	
 	\smallskip
 	\noindent 
 	${\bf (S1)}_n$ There exists a real,  momentum preserving  operator 
 	\begin{equation}\label{def.calLn calDn calQn}
 		\begin{aligned}
 			& \bL_{n} := \lambda\,\omega \cdot \partial_\vphi + \bD_{n} + \bE_{n} : H^{s + 1}_0 \to H^s_0 \,, \\
 			& \bD_{n} := \beta\, \tL + \bZ_{n}=  {\rm diag}_{j \in \Z^2 \setminus \{ 0 \}} \mu_{n}(j) \,, \\ 
 			&\bZ_{n} =  {\rm diag}_{j \in \Z^2 \setminus \{ 0 \}} \tz_{n}(j)\,, \quad \mu_{n}(j) :=  \im\, \beta \tL(j)+ \tz_{n}(j)\,,
 		\end{aligned}
 	\end{equation}
 	defined for any $\omega \in \Omega_{n}^\gamma$, where we define $\Omega_{0}^\gamma := \tD\tC(2\gamma, \tau)$ for $n=0$ and, for $n \geq 1$, 
 	\begin{equation}\label{insiemi di cantor rid}
 		\begin{aligned}
 			\Omega_{n}^\gamma& := \Big\{ \omega \in \Omega_{n - 1}^\gamma \,: \, |\im \,\lambda\,\omega \cdot \ell + \mu_{n - 1}(j) - \mu_{n - 1}(j') | \geq \frac{\lambda\, \gamma }{\langle \ell \rangle^\tau  |j'|^\tau}, \\
 			& \qquad \ \forall \,\ell \in \Z^\nu \setminus \{ 0 \} \,, \ \ 
 			j,j' \in \Z^2 \setminus \{ 0 \}\,, \ \  \pi^\top( \ell) + j-j' =0,  \ \ |\ell|\leq \tN_{n-1} \Big\}\,. 
 		\end{aligned}
 	\end{equation}
	For any $j \in \Z^2 \setminus \{ 0 \}$, the eigenvalues $\mu_{n}(j)= \mu_{n}(j;\omega)$ 
 	satisfy the estimates
 	\begin{align}
 		& |\tz_{n}(j)|^{{\rm Lip}(\gamma)} \lesssim \lambda^{\theta} |j|^{-1}\,, \quad | \tz_{n}(j)- \tz_{0}(j) |^{{\rm Lip}(\gamma)} \lesssim 
 		\varepsilon_{M} \gamma^{-(M-1)}
 		  |j|^{- M} \,,  \label{stime qn} \\
 		&| \tz_{n}(j) - \tz_{n - 1}(j)|^{{\rm Lip}(\gamma)}
 		\lesssim  \varepsilon_{M} \gamma^{-(M-1)} \, \tN_{n - 2}^{- \ta}\, |j|^{- M}  \quad \ \text{when } \quad n\geq 1 \,. \label{cal Nn - N n - 1}
 	\end{align}
 	The operator $\bE_{n}$ is real and momentum preserving, satisfying, for any $s_0\leq s\leq S-\Sigma(\tb)$,
 	\begin{equation}\label{stime cal Rn rid}
 		\begin{aligned}
 			& |\bE_{n} |_{- M, s}^{{\rm Lip}(\gamma)} 
 			\leq C_*(s)  \varepsilon_{M} \gamma^{-(M-1)}  \, \tN_{n - 1}^{- \ta} \|  w \|_{s + \Sigma(\tb)}^{\Lip(\gamma)} \,,  \\
 			& |\bE_{n}  |_{-M, s + \tb}^{{\rm Lip}(\gamma)}  \leq C_*(s) \varepsilon_{M} \gamma^{-(M-1)}   \, \tN_{n - 1} \|  w \|_{s + \Sigma(\tb)}^{\Lip(\gamma)} \,,
 		\end{aligned}
 	\end{equation}
 	for some constant $C_* (s) > 0$ large enough. 
 	\\
 	When $n \geq 1$, there exists an invertible, real, and momentum preserving map 
 	$\Phi_{n -1} = {\rm exp}(\Psi_{n - 1})$, such that, 
 	for any $\omega \in \Omega_{n}^\gamma$, 
 	\begin{equation}\label{coniugazione rid}
 		\bL_{n} = \Phi_{n - 1}^{- 1} \bL_{n-1} \Phi_{n - 1}\,.
 	\end{equation}
 	Moreover, for any $s_0\leq s\leq S-\Sigma(\tb)$, the map $\Psi_{n - 1} : H^s_0 \to H^s_0$ satisfies
 	\begin{equation}\label{stime Psi n rid}
 		\begin{aligned}
 			|\Psi_{n - 1}|_{0, s}^{{\rm Lip}(\gamma)}
 			& \lesssim_s \varepsilon^{M} \gamma^{- M} \tN_{n - 1}^{2 \tau + 1} \tN_{n - 2}^{- \ta} \|  w \|_{s + \Sigma(\tb)}^{\Lip(\gamma)}\,, \\
 			|\Psi_{n - 1} |_{0, s + \tb}^{{\rm Lip}(\gamma)} & \lesssim_s \varepsilon^{M} \gamma^{- M} \tN_{n - 1}^{2 \tau + 1}  \tN_{n - 2} \|  w \|_{s + \Sigma(\tb)}^{\Lip(\gamma)}\,.
 		\end{aligned}
 	\end{equation}
 	In addition, if we assume that $ w(\vf,x)=\odd(\vf,x)$, then the operators $\bL_{n}$, $\bE_{n}$ are reversible, the eigenvalues $\mu_{n}(j)= \mu_{n}(j;\omega)$,  for any $j \in \Z^2 \setminus \{ 0 \}$, are purely imaginary, satisfying
 	\begin{equation}\label{reversibility reality auto}
 		\begin{aligned}
 			& \mu_n(j) = - \mu_n(- j) = - \overline{\mu_n( j)} \,, \ \ \text{or equivalently}   \\
 			& \tz_{n}(j) = - \tz_{n}(- j)= - \overline{\tz_{n}( j)}\,, 
 		\end{aligned}
 	\end{equation}
 	and, for $n\geq 1$, the map $\Phi_{n - 1}$ is reversibility preserving.
 	
 	\noindent
 	${\bf (S2)}_n$ For all $ j \in \Z^2 \setminus \{ 0 \}$,  there exists a Lipschitz extension  of the eigenvalues $\mu_n(j;\,\cdot\,) :\Omega_{n}^\gamma \to  \C$ to the set $\tD\tC(2\gamma, \tau)$, denoted by
 	$ \widetilde \mu_n(j;\,\cdot\,): \tD\tC(2\gamma, \tau) \to \C$,
 	satisfying,  for $n \geq 1$, 
 	\begin{equation}\label{lambdaestesi}  
 		|\widetilde \mu_n(j) -  \widetilde \mu_{n - 1}(j) |^{{\rm Lip}(\gamma)}  \lesssim  |j|^{- M} |\bE_{n - 1}|_{-M,s_0}^{{\rm Lip}(\gamma)}  \lesssim  \varepsilon_{M} \gamma^{-(M-1)} \, \tN_{n - 2}^{- \ta}\, |j|^{- M}\,.
 	\end{equation}
	${\bf(S3)_{n}}$ Let $w_1(\omega)$, $ w_2(\omega) $ satisfying the ansatz \eqref{ansatz} be such that 
${\bf E}_0(w_1)$,  ${\bf E}_0(w_2 )$ satisfy \eqref{stime pre rid}. 
Then 
for all $\omega \in \Omega_n^{\gamma_1}(w_1) \cap \Omega_n^{\gamma_2}(w_2)$
with $\gamma_1, \gamma_2 \in [\gamma/2, 2 \gamma]$, the following estimates hold:
\begin{equation}\label{stime R nu i1 i2}
\begin{aligned}
& |\Delta_{12} {\bf E}_n |_{- M, s_0}
\leq C_*(s)   \tN_{n - 1}^{- \ta}
 \| w_1 - w_2\|_{s_0 +  \Sigma(\tb)} \,,
\\ 
& 
|\Delta_{12} {\bf E}_n |_{- M, s_0 + \tb}
\leq C_*(s)   \tN_{n - 1}^{}
 \| w_1 - w_2\|_{s_0 +  \Sigma(\tb)} \,.
 \end{aligned}
\end{equation}
Moreover for $n \geq 1$, for all $j \in \Z^2 \setminus \{ 0 \}$, 
\begin{equation}\label{r nu - 1 r nu i1 i2}
\begin{aligned}
 \big|\Delta_{12}(\mu_n(j) - \mu_{n - 1}(j)) \big|  &= \big|\Delta_{12}(\tz_n(j) - \tz_{n - 1}(j)) \big| \\
&  \lesssim   \varepsilon_{M} \gamma^{-(M-1)}   \, \tN_{n - 2}^{- \ta} \, |j|^{- M}  \| w_1 - w_2  \|_{ s_0  + \Sigma(\tb)}\,, \\
  |\Delta_{12} \mu_n(j)| = |\Delta_{12} \tz_{n}(j)| & \lesssim  \lambda^{\theta}  |j|^{- 1}  \| w_1 -  w_2  \|_{ s_0  + \Sigma(\tb)}\,. 
\end{aligned}
\end{equation}
\noindent ${\bf(S4)_{n}}$ Let $w_1$, $w_2$ be like in ${\bf(S3)_{n}}$ 
and $0 < \rho \leq \gamma/2$. 
Then:
\begin{equation}\label{inclusione cantor riducibilita S4 nu}
	K \varepsilon \rho^{- 1} \tN_{n - 1}^{\tau + \tau^2} \| w_1 - w_2 \|_{s_0 + \Sigma(\tb)} \leq 1 \quad \Longrightarrow \quad 
	\Omega_n^\gamma(w_1) \subseteq \Omega_n^{\gamma - \rho}(w_2) \, . 
\end{equation}

 \end{prop}
 
 \begin{proof}
 	{\sc Proof of ${\bf (S1)}_0-{\bf (S4)}_0$.} 
 	The claimed properties follow directly 
 	from Proposition \ref{prop normal form lower orders}, recalling \eqref{bL0_inizioKAM}, \eqref{stime pre rid}, \eqref{vare.vare} 
 	and the definition of $\Omega_{0}^\gamma := \tD\tC(2\gamma, \tau)$. 
 	
 	\smallskip
 	
 	\noindent
	By induction, we assume the the claimed properties ${\bf (S1)}_n$-${\bf (S4)}_n$ hold for some $n \geq 0$ and we prove them at the step $n + 1$. 
	
	\noindent
 	{\sc Proof of ${\bf (S1)}_{n+1}$.} 
 	Let $\Phi_n = {\rm exp}( \Psi_n)$ where $\Psi_n$ is an operator of order $- M$ to be determined. 
 	By the Lie expansion, we compute
 	\begin{equation}\label{primo coniugio Ln Psin}
 		\begin{aligned}
 			\bL_{n + 1} & = \Phi_n^{- 1}\bL_{n} \Phi_n  = \lambda \, \omega \cdot \partial_\vphi + \bD_n + \lambda\, \omega \cdot \partial_\vphi \Psi_n + [\bD_n , \Psi_n] + \Pi_{\tN_n} \bE_n + \bE_{n + 1} \,,  \\
 			\bE_{n + 1} & := \Pi_{\tN_n}^\bot \bE_n   + \int_0^1 (1 - \tau) {\rm exp}(- \tau \Psi_n) \,[\lambda \,\omega \cdot \partial_\vphi \Psi_n + [\bD_n , \Psi_n] , \Psi_n] \, {\rm exp}(\tau \Psi_n) \wrt \tau \\
 			& \quad + \int_0^1 {\rm exp}(- \tau \Psi_n) [\bE_n , \Psi_n] {\rm exp}(\tau \Psi_n) \wrt \tau \,,  
 		\end{aligned}
 	\end{equation}
 	where $\bD_{n}:=  \beta \,\tL + \bZ_{n}$ and the projectors $\Pi_{N}$, $\Pi_{N}^\bot$ are defined in  \eqref{def proiettore operatori matrici}. 
 	Our purpose is to find a map $\Psi_n$ solving the {\it homological equation} 
 	\begin{equation}\label{equazione omologica KAM}
 		\lambda\,\omega \cdot \partial_\vphi \Psi_n + [\bD_{n} , \Psi_n] + \Pi_{\tN_{n}} \bE_{n} = \widehat{\bE}_n(0) \,,
 	\end{equation}
 	where $\widehat{\bE}_n(0) := \frac{1}{(2 \pi)^\nu} \int_{\T^\nu} \bE_n(\vphi) \wrt \vphi$ is a diagonal operator by Lemma \ref{lem:mom_pres}, since $\bE_{n}(\vf)$ is a momentum preserving operator by induction assumption.
	By  \eqref{matrix representation 1} and \eqref{def.calLn calDn calQn}, the homological equation \eqref{equazione omologica KAM} is equivalent to 
 	\begin{equation}\label{eq omologica matrici}
	\begin{aligned}
 	&	\big( \im \,\lambda\, \,\omega \cdot \ell + \mu_{n}(j) - \mu_{n}(j') \big) \widehat \Psi_{n}(\ell)_j^{j'} + \widehat{\bE_{n}}(\ell)_j^{j'} = 0\,, \\
	&   \ell \in \Z^\nu\setminus\{0\}\,, \ \   |\ell| \leq \tN_n\,,\, \ \ j, j' \in \Z^2 \setminus \{ 0 \}\,, \ \ \pi^\top(\ell) + j - j' = 0\,. 
		\end{aligned}
 	\end{equation}
 	 Therefore, we define the linear operator $\Psi_{n}$ by
 	\begin{equation}\label{def Psi eq omo KAM}
 		\begin{footnotesize}
 			\widehat\Psi_{n} (\ell)_j^{j'} := \begin{cases}
 				- \dfrac{\widehat{\bE_{n}}(\ell)_j^{j'} }{ \im \,\lambda\, \omega \cdot \ell + \mu_{n}(j) - \mu_{n}(j')  }, & \begin{matrix}
 					 \ell \in \Z^{\nu}\setminus\{0\} \,, \quad j,j'\in \Z^2\setminus\{0\}\,, \\  |\ell|  \leq \tN_{n}\,, \quad \pi^\top(\ell) + j-j'=0\,,
 				\end{matrix} \\
 				0 & \text{otherwise}\,. 
 			\end{cases}
 		\end{footnotesize}
 	\end{equation}
 	which  is a solution of \eqref{equazione omologica KAM}-\eqref{eq omologica matrici}. 
 	
 	\begin{lem}\label{Lemma eq omologica riducibilita KAM}
 		The operator $\Psi_n$ in \eqref{def Psi eq omo KAM}, defined for any $\omega \in \Omega_{n + 1}^\gamma$, satisfies, for any $s_0\leq s\leq S-\Sigma(\tb)$,
 		\begin{equation}\label{stime eq omologica}
 			\begin{aligned}
 				& |\Psi_n|_{0, s}^{{\rm Lip}(\gamma)}
 				\lesssim_{s} \tN_{n}^{2 \tau + 1} \lambda^{- 1} \gamma^{- 1} |\bE_{n}|_{- M, s}^{{\rm Lip}(\gamma)}\,, \\
 				& |\Psi_n|_{0, s + \eta}^{{\rm Lip}(\gamma)} \lesssim_{s} \tN_{n}^{2 \tau + 1 + \eta} \lambda^{- 1}\gamma^{- 1} |\bE_{n}|_{- M, s}^{\Lip(\gamma)}\,, \quad \forall \, \eta > 0 \,.
 			\end{aligned}
 		\end{equation}
		Assume that $w_1\,,\, w_2$ satisfy \eqref{ansatz} with $\sigma = \Sigma(\tb)$, as in \eqref{definizione.param.KAM}. Then for any $\omega \in \Omega_{n + 1}^{\gamma_1}(w_1) \cap \Omega_{n + 1}^{\gamma_2}(w_2)$, $\gamma/2 \leq \gamma_1, \gamma_2 \leq \gamma$, one has 
		\begin{equation}\label{stime Delta 12 Psin}
		\begin{aligned}
		|\Delta_{12} \Psi_n|_{0, s_0} & \lesssim \varepsilon^{M} \gamma^{-M} \tN_{n}^{2 \tau }  \tN_{n - 1}^{- \ta} \| w_1 - w_2 \|_{s_1 + \Sigma(\tb)}  \,, \\
		|\Delta_{12} \Psi_n |_{0, s_0 + \tb} & \lesssim \varepsilon^{M} \gamma^{-M} \tN_{n}^{2 \tau } \tN_{n - 1}\| w_1 - w_2 \|_{s_1 + \Sigma(\tb)} \,, \\
		|\Delta_{12} \Psi_n|_{0, s_0 + \eta} & \lesssim \varepsilon^{M} \gamma^{-M} \tN_{n}^{2 \tau + \eta  }  \tN_{n - 1}^{- \ta} \| w_1 - w_2 \|_{s_1 + \Sigma(\tb)}  \,,  \quad \forall \, \eta > 0 \,,\\
		|\Delta_{12} \Psi_n |_{0, s_0 + \tb + \eta} & \lesssim  \varepsilon^{M} \gamma^{-M} \tN_{n}^{2 \tau + \eta }  \tN_{n - 1}\| w_1 - w_2 \|_{s_1 + \Sigma(\tb)} \,, \quad \forall \, \eta > 0\,. 
		\end{aligned}
		\end{equation}
			Moreover, $\Psi_n$ is real and momentum preserving. In addition, if we assume $ w(\vf,x)=\odd(\vf,x)$, the map $\Psi_n$ is also reversibility preserving.
 	\end{lem}
 	
 	\begin{proof}
 		To simplify notations, along this proof 
 		we drop the index $n$.
		
		\medskip
		
		\noindent
		{\bf Proof of \eqref{stime eq omologica}.}
 		 Let $\wh\Psi(\ell)_{j}^{j'}=\wh\Psi(\ell;\omega)_{j}^{j'}$ as in \eqref{def Psi eq omo KAM}, with $ \ell \in \Z^\nu$, $ j, j' \in \Z^2 \setminus \{ 0 \}$, with $0 < |\ell| \leq \tN$ and $ \pi^\top(\ell) + j - j' = 0$. For any $\omega \in \Omega_{n + 1}^\gamma$ (see \eqref{insiemi di cantor rid}), we immediately get the estimate
 		\begin{equation}\label{stima eq omo KAM 1}
 			|\widehat\Psi (\ell;\omega)_j^{j'}| \lesssim \lambda^{- 1}\gamma^{- 1} \langle \ell \rangle^\tau | j' |^\tau |\wh\bE(\ell;\omega)_j^{j'}| \lesssim \lambda^{- 1} \gamma^{- 1}  \tN^\tau  | j' |^\tau |\wh\bE (\ell;\omega)_j^{j'}| \,.
 		\end{equation}
 		 We define $\delta_{\ell j j'}(\omega) :=\im\,\lambda\, \omega \cdot \ell +\mu(j;\omega) - \mu(j';\omega)$. Let $\omega_{1}, \omega_{2}  \in \Omega_{n + 1}^\gamma$. By \eqref{def.calLn calDn calQn}, \eqref{insiemi di cantor rid}, \eqref{stime qn},  we have  
		 \begin{equation}
		 	 \begin{aligned}
		 		& \big| \big( \mu(j;\omega_1) - \mu(j';\omega_1) \big) - \big( \mu(j;\omega_2) - \mu(j';\omega_2)\big) \big|  \\
		 		&\leq \big|\tz(j; \omega_1) - \tz(j; \omega_2)\big| + \big|\tz(j'; \omega_1) - \tz(j'; \omega_2)\big| \\
		 		& \lesssim \lambda^{\theta} \gamma^{- 1} \big( |j|^{-1} + |j'|^{-1} \big) |\omega_1 - \omega_2| \lesssim  \lambda^{\theta} \gamma^{- 1}  |\omega_1 - \omega_2| \,,
		 	\end{aligned}
		 \end{equation}
		 and therefore, using that $\lambda^{\theta-1}\gamma^{-1}\ll 1 $ by \eqref{piccolo.ansatx},
		 \begin{equation}\label{calcetto1}
		 		 \begin{aligned}
		 		|\delta_{\ell j j'}(\omega_1) - \delta_{\ell j j'}(\omega_2)| & \lesssim (\lambda\, |\ell|+\lambda^{\theta}\gamma^{-1}) |\omega_1 - \omega_2|  \lesssim \lambda \, \langle \ell \rangle |\omega_1 - \omega_2|\,.
		 	\end{aligned}
		 \end{equation}
		 Then, estimate \eqref{calcetto1}, together with the fact that $\omega_1, \omega_2 \in \Omega_{n + 1}^\gamma$, implies that
 		\begin{equation}\label{calcetto2}
 			\begin{aligned}
 				\Big| \frac{1}{ \delta_{\ell j j'}(\omega_{1})} - \frac{1}{\delta_{\ell j j'}(\omega_{2})} \Big| & \leq \dfrac{|\delta_{\ell j j'}(\omega_{1}) - \delta_{\ell j j'}(\omega_{2})|}{|\delta_{\ell j j'}(\omega_{1})| |\delta_{\ell j j'}(\omega_{2})|} \\
 				& \lesssim \lambda (\lambda\gamma)^{-2} \langle \ell \rangle^{2 \tau + 1}  |j'|^{2 \tau} |\omega_{1} - \omega_{2} | \\
				&  \lesssim \lambda^{- 1} \gamma^{- 2} \langle \ell \rangle^{2 \tau + 1}  |j'|^{2 \tau} |\omega_{1} - \omega_{2} |\,. 
 			\end{aligned}
 		\end{equation}
 		Therefore, by \eqref{def Psi eq omo KAM} and \eqref{calcetto2},  for any $\omega_{1},\omega_{2}\in \Omega_{n + 1}^\gamma$ we have that
 		\begin{equation}\label{Psi lam 12 KAM}
 			\begin{aligned}
 				\big|\widehat\Psi (\ell;\omega_{1})_j^{j'} - \widehat\Psi (\ell;\omega_{2})_j^{j'}\big| & \lesssim  \langle \ell \rangle^\tau  |j'|^\tau \lambda^{- 1} \gamma^{- 1} \big|\widehat\bE(\ell;\omega_{1})_j^{j'} - \widehat\bE(\ell;\omega_{2})_j^{j'}\big| \\
 				& \ \ + \langle \ell \rangle^{2 \tau + 1}  |j'|^{2 \tau} \lambda^{- 1} \gamma^{-2}  \big|\widehat{\bE}(\ell;\omega_{2})_j^{j'}\big| |\omega_{1}-\omega_{2} | \\
				& \lesssim \tN^\tau  |j'|^\tau \lambda^{- 1} \gamma^{- 1} \big|\widehat\bE(\ell;\omega_{1})_j^{j'} - \widehat\bE(\ell;\omega_{2})_j^{j'}\big| \\
 				& \ \ + \tN^{2 \tau + 1}  |j'|^{2 \tau} \lambda^{- 1} \gamma^{-2}  \big|\widehat{\bE}(\ell;\omega_{2})_j^{j'}\big| |\omega_{1}-\omega_{2} |\,. 
 			\end{aligned}
 		\end{equation}
 		Since $M  > 2 \tau$ by \eqref{definizione.param.KAM}, recalling Definition \ref{block norm} and collecting the estimates \eqref{stima eq omo KAM 1}, \eqref{Psi lam 12 KAM}, we obtain the bounds, for any $s\in [s_0,S+\Sigma(\tb)]$,
 		$$
 		\begin{aligned}
 			& |\Psi|_{0, s}^{\rm sup} \lesssim \tN^{ \tau} \lambda^{- 1}\gamma^{- 1} |{\bE}|_{- M, s}^{\rm sup}\,, \\
 			& |\Psi|_{0, s}^{\rm lip} \lesssim  \tN^{2 \tau + 1} \lambda^{- 1}\gamma^{- 2}  |{\bE}|_{- M, s}^{\rm sup} + \tN^{\tau} \lambda^{- 1} \gamma^{- 1}  |{\bE}|_{- M, s}^{\rm lip}\,.
 		\end{aligned}
 		$$
 		Similarly, using also that $|\ell| \leq \tN$ and $\pi^\top(\ell) + j - j' = 0$ imply that $|j - j'| \lesssim |\ell| \lesssim \tN$, by analogous arguments we obtain that, for any $\eta > 0$,
 		$$
 		\begin{aligned}
 			& |\Psi|_{0, s + \eta}^{\rm sup} \lesssim \tN^{\tau + \eta} \gamma^{- 1} \lambda^{- 1} |{\bE}|_{- M, s}^{\rm sup}\,, \\
 			& |\Psi|_{0, s+ \eta}^{\rm lip} \lesssim  \tN^{2 \tau + \eta +1} \gamma^{- 2} \lambda^{- 1}  |{\bE}|_{- M, s}^{\rm sup} + \tN^{\tau + \eta} \lambda^{- 1} \gamma^{- 2} |{\bE}|_{- M, s}^{\rm sup}\,. 
 		\end{aligned}
 		$$
 		Hence, we conclude the claimed bounds in \eqref{stime eq omologica}. 
\\[1mm]
		\noindent
		{\bf Proof of \eqref{stime Delta 12 Psin}.} Let $\ell\in\Z^{\nu}$, $j, j' \in \Z^2 \setminus \{ 0\}$, with $0 < |\ell| \leq \tN$ and $\pi^\top(\ell) + j - j' = 0$. Recalling \eqref{def.calLn calDn calQn}, we set
		\begin{equation}\label{def delta i1 delta l j j'}
		\begin{aligned}
		& \delta_{\ell j j'}(w_i) := \im\, \lambda \,\omega \cdot \ell + \mu(j; w_i) - \mu(j'; w_i) \,, \quad i = 1,2\,,  \\
		& \mu(j; w_i) :=  \im\, \beta \,\tL(j)+ \tz(j; w_i)\,. 
		\end{aligned}
		\end{equation}
		By \eqref{r nu - 1 r nu i1 i2}, we get that 
		\begin{equation}\label{stima delta 12 delta l j j'}
		\begin{aligned}
		|\Delta_{12} \delta_{\ell j j'}| & \leq  |\Delta_{12} \tz(j)| + |\Delta_{12} \tz(j')| \lesssim \lambda^{\theta}\big(|j|^{- 1} + |j'|^{- 1} \big) \| w_1 - w_2 \|_{s_1 + \Sigma(\tb)} \\
		& \lesssim \lambda^{\theta}  \| w_1 - w_2 \|_{s_1 + \Sigma(\tb)} \,. 
		\end{aligned}
		\end{equation}
		Therefore, by \eqref{def Psi eq omo KAM}, \eqref{stima Delta 12 quadratico E n + 1} and using that $\lambda^{\theta-1}\gamma^{-1}\ll 1$, we get, for any  $\omega \in \Omega_{n + 1}^{\gamma_1}(w_1) \cap \Omega_{n + 1}^{\gamma_1}(w_2)$ with $\gamma/2 \leq \gamma_1, \gamma_2 \leq \gamma$,
		$$
		\begin{aligned}
		\big|\Delta_{12} \widehat{\Psi}(\ell)_j^{j'}\big| & \leq \frac{\big|\Delta_{12} \widehat\bE(\ell)_j^{j'}\big|}{|\delta_{\ell j j'}(w_1)|} + \frac{\big|\widehat\bE(\ell)_j^{j'}(w_2\big)| |\Delta_{12} \delta_{\ell j j'}|}{|\delta_{\ell j j'}(w_1)| |\delta_{\ell j j'}(w_2)|} \\
		& \lesssim \tN^\tau \lambda^{- 1} \gamma^{- 1} |j'|^\tau \big|\Delta_{12} \widehat\bE(\ell)_j^{j'}\big|  \\
		& \quad + \tN^{2 \tau} \lambda^{\theta} \lambda^{- 2} \gamma^{- 2} |j'|^{2 \tau}\big|\widehat\bE(\ell)_j^{j'}(w_2)\big|  \| w_1 - w_2 \|_{s_1 + \Sigma(\tb)} \\
		& 
		\lesssim
		 \tN^{2 \tau} \lambda^{- 1} \gamma^{- 1} |j'|^{2 \tau} \big( \big|\Delta_{12} \widehat\bE(\ell)_j^{j'}\big| + \big|\widehat\bE(\ell)_j^{j'}(w_2)\big|  \| w_1 - w_2 \|_{s_1 + \Sigma(\tb)}\big) \,.
		\end{aligned}
		$$
		Since $M  > 2 \tau$ by \eqref{definizione.param.KAM}, recalling Definition \ref{block norm} and the induction estimates \eqref{stime cal Rn rid}, \eqref{stime R nu i1 i2}, the claimed bounds follow easily.
		Finally, since ${\bE}$ is real, reversible and momentum preserving, by \eqref{def Psi eq omo KAM}, Lemma \ref{lemma real rev matrici} and Lemma \ref{lem:mom_pres} we have that $\Psi$ is real, reversibility preserving and momentum preserving.  This concludes the proof.
 	\end{proof}

 	By Lemma \ref{Lemma eq omologica riducibilita KAM}, the induction assumption on the estimates \eqref{stime cal Rn rid} and by \eqref{vare.vare}, we obtain, for any $s_0\leq s\leq S-\Sigma(\tb)$,
 	\begin{equation}\label{stime Psin neumann}
 		\begin{aligned}
 			|\Psi_n|_{0, s}^{{\rm Lip}(\gamma)} 
 			& \lesssim_{s}  \tN_n^{2 \tau + 1} \lambda^{- 1}\gamma^{- 1} |{\bE}_n |_{- M, s}^{{\rm Lip}(\gamma)}
 			\lesssim_{s } \tN_n^{2 \tau + 1} \tN_{n - 1}^{- \ta} 
 			\varepsilon^M \gamma^{- M} \|  w \|_{s + \Sigma(\tb)}^{\Lip(\gamma)} \,, \\
 			|\Psi_n|_{0, s + \tb}^{{\rm Lip}(\gamma)} & \lesssim_{s} \tN_n^{2 \tau + 1} \lambda^{- 1} \gamma^{- 1} |{\bE}_n |_{-M, s + \tb}^{{\rm Lip}(\gamma)} \lesssim_{s } \tN_n^{2 \tau + 1} \tN_{n - 1} \varepsilon^{M}\gamma^{- M} \|  w \|_{s + \Sigma(\tb)}^{\Lip(\gamma)}\,,
 		\end{aligned}
 	\end{equation} 
 	which are the estimates \eqref{stime Psi n rid} at the step $n + 1$. 
 	Moreover, setting $\eta=M$ in \eqref{stime eq omologica}, by the same arguments we also have
 	 	\begin{equation}\label{stime.Psin.neumann2}
 		\begin{aligned}
 			|\Psi_n|_{0, s + M}^{{\rm Lip}(\gamma)} \,,\, 
 			& \lesssim_{s}  \tN_n^{2 \tau + 1 + M} \gamma^{- 1} \lambda^{- 1} |{\bE}_n |_{- M, s}^{{\rm Lip}(\gamma)}
 			\\
 			& \lesssim_{s } \tN_n^{2 \tau + 1 + M} \tN_{n - 1}^{- \ta} 
 			\varepsilon^M \gamma^{- M} \|  w \|_{s + \Sigma(\tb)}^{\Lip(\gamma)} \,, \\
 			|\Psi_n|_{0, s + \tb + M}^{{\rm Lip}(\gamma)} & \lesssim_{s} \tN_n^{2 \tau + 1 + M} \lambda^{- 1}\gamma^{- 1} |{\bE}_n |_{-M, s + \tb}^{{\rm Lip}(\gamma)}  \\
 			& \lesssim_{s } \tN_n^{2 \tau + M + 1} \tN_{n - 1} \varepsilon^M \gamma^{- M} \|  w \|_{s + \Sigma(\tb)}^{\Lip(\gamma)} \,.
 		\end{aligned}
 	\end{equation}
 	In particular, by \eqref{definizione.param.KAM}, \eqref{ansatz} and by the smallness condition \eqref{KAM smallness condition}, we deduce, setting $s=s_0$ in \eqref{stime.Psin.neumann2}, for $\tN_0>0$ large enough,
 	\begin{equation} \label{2103.1}
	\begin{aligned}
 		|\Psi_n|_{0, s_0 + M}^{{\rm Lip}(\gamma)} 
		& \lesssim \tN_n^{2 \tau +1 + M} \tN_{n - 1}^{- \ta} \varepsilon^M \gamma^{- M} \|  w \|_{s_0 + \Sigma(\tb)}^{\Lip(\gamma)} \ \leq \delta < 1\,.
		\end{aligned}
 	\end{equation}
Therefore, by Lemma \ref{proprieta standard norma decay}-$(iv)$ and estimates \eqref{stime.Psin.neumann2}, \eqref{2103.1}, we have the estimates, for any $s_0 \leq s \leq S - \Sigma(\tb)$,
	\begin{equation}\label{stime Phi n inv - Id}
			\begin{aligned}
			& \sup_{\tau \in [- 1, 1]}|{\rm exp}(\tau \Psi_n)|_{0, s  + M}^{\Lip(\gamma)} \lesssim_s 1 + |\Psi_n|_{0, s + M}^{\Lip(\gamma)}\,, \quad \sup_{\tau \in [- 1, 1]}|{\rm exp}(\tau \Psi_n)|_{0, s_0  + M}^{\Lip(\gamma)} \lesssim 1\,,  \\
			& \sup_{\tau \in [- 1, 1]}|{\rm exp}(\tau \Psi_n)|_{0, s  + \tb +  M}^{\Lip(\gamma)} \lesssim_s 1 + |\Psi_n|_{0, s + \tb + M}^{\Lip(\gamma)}\,. 
		\end{aligned}
	\end{equation}
 	Then, by recalling \eqref{primo coniugio Ln Psin} and by using that $\Psi_n$ solves the equation \eqref{equazione omologica KAM}, we conclude that
 		\begin{equation}\label{2003.1}
 				\begin{aligned}
 				{\bL}_{n + 1} & \,=\lambda\, \omega \cdot \partial_\vphi + {\bD}_{n + 1} + {\bE}_{n + 1}\,,  \\
 				{\bD}_{n + 1} & :=  -\beta \tL + {\bZ}_{n + 1}\,, \qquad {\bZ}_{n + 1}  : = {\bZ}_n + \widehat{\bE}_n(0)\,,  \\
 				{\bE}_{n + 1}& \,=\Pi_{\tN_n}^\bot \bE_n   + \int_0^1 (1 - \tau) {\rm exp}(- \tau \Psi_n) \,[\widehat{\bE}_n(0) - \Pi_{\tN_n} \bE_n , \Psi_n] \, {\rm exp}(\tau \Psi_n)\wrt  \tau  \\
 				& \quad + \int_0^1 {\rm exp}(- \tau \Psi_n) [\bE_n , \Psi_n] {\rm exp}(\tau \Psi_n) \wrt \tau\,. 
 			\end{aligned}
 		\end{equation}
 	All the operators in \eqref{2003.1}
 	are defined for any $\omega \in \Omega_{n + 1}^\gamma$. 
 	Moreover, by \eqref{primo coniugio Ln Psin}, \eqref{equazione omologica KAM},
 	for $\omega \in \Omega^\gamma_{n+1}$ one has the 
 	identity
 	$\Phi_n^{-1} \bL_n \Phi_n = \bL_{n+1}$, 
 	which is \eqref{coniugazione rid} at the step $n+1$. 
 	Recalling that $\wh\bE_{n}(0)$ is a diagonal operator and by \eqref{2003.1}, we have that
 	\begin{equation}\label{frittata di maccheroni}
	\begin{aligned}
		{\bZ}_{n + 1} & = {\rm diag}_{j \in \Z^2 \setminus \{ 0 \}} \tz_{n + 1}(j)\,, \quad	\tz_{n + 1}(j)  := \tz_n(j) + \widehat{\bE}_n(0)_j^j\,, \\
		{\bD}_{n + 1} &= {\rm diag}_{j \in \Z^2 \setminus \{ 0 \}} \mu_{n + 1}(j)\,, \quad	\mu_{n + 1}(j) :=  \im \, \beta\, \tL(j)+ \tz_{n + 1}(j)\,.
	\end{aligned}
\end{equation}
 	By Lemma \ref{proprieta standard norma decay}-$(v)$
 	\begin{equation}\label{frittata.1}
 		\begin{aligned}
 			| \mu_{n + 1}(j) - \mu_n(j)|^{{\rm Lip}(\gamma)} & = | \tz_{n + 1}(j) - \tz_n(j)|^{{\rm Lip}(\gamma)}   \\
 			& \leq | \widehat{\bE}_n(0)_j^j |^{{\rm Lip}(\gamma)} \lesssim |{\bE}_n|_{- M, s_0}^{{\rm Lip}(\gamma)} \langle j \rangle^{- M}\,.
 		\end{aligned}
 	\end{equation}
 	Then, \eqref{frittata.1}, together with the estimate \eqref{stime cal Rn rid},
 	implies  \eqref{cal Nn - N n - 1} at the step $n + 1$. 
 	The estimate \eqref{stime qn} at the step $n + 1$ follows, as usual, 
 	by a telescoping argument, 
 	using the fact that $\sum_{n \geq 0} \tN_{n - 1}^{- \ta}$ is convergent 
 	since $\ta > 0$ (see \eqref{definizione.param.KAM}). 
 	Now we prove the estimates \eqref{stime cal Rn rid} at the step $n + 1$. 
 	By \eqref{2003.1}, estimates 
 	\eqref{stime Psin neumann}, 
 	\eqref{2103.1},
 	\eqref{stime Phi n inv - Id}, the induction estimates \eqref{stime cal Rn rid},
 	Lemma \ref{proprieta standard norma decay}-$(ii),(v)$ and 
 	Lemma \ref{lemma proiettori decadimento}, 
 	we get, for any $s_0\leq s \leq S-\Sigma(\tb)$,
 	\begin{equation}\label{thuram}
 			\begin{aligned}
 				& |{\bE}_{n + 1}|_{- M, s}^{{\rm Lip}(\gamma)} 
 				\lesssim_{s} \tN_n^{- \tb} |{\bE}_n|_{- M, s + \tb}^{{\rm Lip}(\gamma)} + \tN_n^{2 \tau + 1 + M}\lambda^{- 1} \gamma^{- 1}  |{\bE}_n|_{- M, s_0}^{{\rm Lip}(\gamma)} |{\bE}_n|_{- M, s}^{{\rm Lip}(\gamma)} \,, \\
 				& |{\bE}_{n + 1}|_{- M, s + \tb}^{{\rm Lip}(\gamma)} 
 				\lesssim_{s} |{\bE}_n|_{- M, s +\tb}^{{\rm Lip}(\gamma)}  \\
				& \quad \quad + \tN_{n}^{2 \tau + 1 + M} \lambda^{- 1}\gamma^{-1}\big( |{\bE}_n|_{- M, s_0}^{{\rm Lip}(\gamma)} |{\bE}_n|_{- M, s+\tb}^{{\rm Lip}(\gamma)} +  |{\bE}_n|_{- M, s}^{{\rm Lip}(\gamma)} |{\bE}_n|_{- M, s_0+\tb}^{{\rm Lip}(\gamma)} \big) \,. 
 			\end{aligned}
 	\end{equation}
 	We now verify the estimates \eqref{stime cal Rn rid} at the step $n + 1$. By \eqref{thuram}, the induction assumption on the estimate \eqref{stime cal Rn rid}, \eqref{vare.vare} and using the ansatz \eqref{ansatz} with $\sigma = \Sigma(\tb)$, we get, for any $s_0 \leq s \leq S - \Sigma(\tb)$,
	\begin{equation}\label{milan merda 0}
	\begin{aligned}
		|{\bE}_{n + 1}|_{- M, s +\tb}^{{\rm Lip}(\gamma)} & \leq C(s) C_*(s)\tN_{n - 1} \varepsilon_{M} \gamma^{-(M-1)} \|  w\|_{s+\Sigma(\tb)}^{\Lip(\gamma)} \\
		&
		+ 2C(s)C_*^2(s) C_0 \,\tN_n^{2 \tau + M + 1} \tN_{n - 1}^{- \ta}\varepsilon^{M} \gamma^{- M} \varepsilon_{M} \gamma^{-(M-1)}  \|  w \|_{s + \Sigma(\tb)}^{\Lip(\gamma)} \\
		& \leq C_*(s) \tN_n  \varepsilon_{M} \gamma^{-(M-1)} \|  w \|_{s + \Sigma(\tb)}^{\Lip(\gamma)} \,,
	\end{aligned}
\end{equation}
	provided
	\begin{equation}\label{forza inter 0}
	C(s) \tN_{n}^{-1} \tN_{n - 1} \leq \frac12 \,, \quad 2 C(s)C_*(s) C_0\, \tN_n^{2 \tau + M } \tN_{n - 1}^{ - \ta}\varepsilon^{M} \gamma^{- M} \leq \frac12\,,
	\end{equation}
	and 
	\begin{equation}\label{milan merda 1}
	\begin{aligned}
		|{\bE}_{n + 1}|_{- M, s}^{{\rm Lip}(\gamma)}
		& \leq  C(s) C_*(s) \tN_n^{- \tb} \tN_{n - 1} \varepsilon_{M} \gamma^{-(M-1)}  \|  w \|_{s+ \Sigma(\tb)}^{\Lip(\gamma)}     \\
		& + C(s) C_*^2(s) C_0\, \tN_n^{2 \tau + M + 1} \tN_{n - 1}^{- 2 \ta} \varepsilon^{-M} \gamma^{-M} \varepsilon_{M} \gamma^{-(M-1)}  \|  w \|_{s+ \Sigma(\tb)}^{\Lip(\gamma)}  \\
		& \leq C_*(s) \tN_n^{- \ta} \varepsilon_{M} \gamma^{-(M-1)}  \|  w \|_{s + \Sigma(\tb)}^{\Lip(\gamma)} \,,
	\end{aligned}
\end{equation}
	provided
	\begin{equation}\label{forza inter 1}
	C(s) \tN_n^{\ta - \tb} \tN_{n - 1}  \leq \frac12  \,, \quad C(s) C_*(s)C_0 \, \tN_n^{2 \tau + \ta + M + 1} \tN_{n - 1}^{- 2 \ta}\varepsilon^{M} \gamma^{- M} \leq \frac12  \,;
	\end{equation}
	note that the constant $C(s)$ comes from estimates \eqref{thuram}, the constant $C_*(s)$ comes from the estimates \eqref{stime cal Rn rid} and the constant $C_0$ comes from impositing the ansatz \eqref{ansatz}.
	The conditions \eqref{forza inter 0}, \eqref{forza inter 1} are verified by \eqref{definizione.param.KAM}, \eqref{KAM smallness condition}, taking $\tN_0, \lambda  \gg 0$ large enough. Therefore, the estimates \eqref{milan merda 0} and \eqref{milan merda 1} imply that the estimates \eqref{stime cal Rn rid} hold at the step $n + 1$.  Finally,	since $\Psi_n, \Phi_n, \Phi_n^{- 1}$ are real and momentum preserving by Lemma \ref{equazione omologica KAM} and ${\bD}_n, {\bE}_n$ are real and momentum preserving by the induction assumption, then ${\bD}_{n + 1}$, ${\bE}_{n + 1}$ are real and momentum preserving operators, by \eqref{2003.1} and Lemmata \ref{lemma real rev matrici}, \ref{lem:mom_prop}, \ref{lem:mom_pres}. Morover if $ w = {\rm odd}(\vf, x)$, since $\Psi_n, \Phi_n, \Phi_n^{- 1}$ are reversibility preserving by Lemma \ref{equazione omologica KAM} and ${\bD}_n, {\bE}_n$ are reversible by the induction assumption, then ${\bD}_{n + 1}$, ${\bE}_{n + 1}$ are reversible by \eqref{2003.1} and Lemma \ref{lemma real rev matrici} and hence \eqref{reversibility reality auto} is verified at the step $n + 1$.
 	\\[1mm]
 	\noindent
 	{\sc Proof of ${\bf (S2)}_{n + 1}$.} We now construct a Lipschitz extension for the eigenvalues $\mu_{n + 1}(j,\,\cdot\,) : \Omega_{n + 1}^\gamma \to  \,\C$. By the induction hypothesis, there exists a Lipschitz extension of $\mu_n(j;\omega)$,  denoted by $\widetilde \mu_{n}(j;\omega)$, to the whole set $\tD\tC(2\gamma, \tau)$ that satisfies ${\bf (S2)}_n$. By \eqref{frittata di maccheroni}, we have $\mu_{n + 1}(j) = \mu_n(j) + \tr_n(j)$, where $\tr_n(j, \omega)=\tr_n(j,\omega) := \widehat{\bE}_n(0;\omega)_j^j$ satisfies $|\tr_n(j)|^{{\rm Lip}(\gamma)} \lesssim  \tN_{n - 1}^{- \ta} |j|^{- M}  $. 
	Hence, by the Kirszbraun Theorem (see Lemma M.5 \cite{KukPo}) there exists a Lipschitz extension $\widetilde \tr_n(j,\,\cdot\,) : \tD\tC(2\gamma, \tau) \to \C$ of $\tr_n(j,\,\cdot\,) : \Omega_{n + 1}^\gamma \to \R$ satisfying $|\widetilde \tr_n(j)|^{{\rm Lip}(\gamma)} \lesssim |\tr_n(j)|^{{\rm Lip}(\gamma)} \lesssim \tN_{n - 1}^{- \ta} |j|^{- M}$. The claimed statement then follows by defining $\widetilde \mu_{n + 1}(j) := \widetilde  \mu_n(j) + \widetilde \tr_n(j)$. 
	
	\medskip
	
	\noindent
	{\sc Proof of ${\bf (S3)}_{n + 1}$.}  Let $w_1, w_2$ satisfy the ansatz \eqref{ansatz} with $\sigma = \Sigma(\tb)$ and let $\omega \in \Omega_{n + 1}^{\gamma_1}(w_1) \cap \Omega_{n + 1}^{\gamma_2}(w_2)$, $\gamma_1, \gamma_2 \in [\gamma/2\,,\, 2 \gamma]$. Let
	\begin{equation}
		\Delta_{12}{\mathcal F}_\tau := {\mathcal F}_\tau (w_1) - {\mathcal F}_\tau(w_2) \,, \quad {\mathcal F}_\tau := {\rm exp}(\tau \Psi_n), \quad \tau \in [-1, 1]\,.
	\end{equation}
	 By Lemmata \ref{proprieta standard norma decay}-$(iv)$, \ref{Lemma eq omologica riducibilita KAM} with $s_1=s_0$, together with $ \Delta_{12} \cF_{\tau} = ({\rm exp } (\tau \Delta_{12} \Psi_{n} )-{\rm Id}) \cF_{\tau} (w_2)$,
	  we have
\begin{equation}\label{stima delta 12 exp Psi n}
\begin{aligned}
& |\Delta_{12} {\mathcal F}_\tau|_{0, s_0} \lesssim \tN_{n}^{2 \tau} \tN_{n - 1}^{- \ta}\varepsilon^{M} \gamma^{-M}  \| w_1 - w_2 \|_{s_0 + \Sigma(\tb)}\,, \\
& |\Delta_{12} {\mathcal F}_\tau|_{0, s_0 + M} \lesssim \tN_{n}^{2 \tau + M}  \tN_{n - 1}^{- \ta}  \varepsilon^{M} \gamma^{-M}\| w_1 - w_2 \|_{s_0 + \Sigma(\tb)}\,, \\
&  |\Delta_{12} {\mathcal F}_\tau|_{0, s_0 + \beta}   \lesssim \tN_{n}^{2 \tau }  \tN_{n - 1}  \varepsilon^{M} \gamma^{-M}\| w_1 - w_2 \|_{s_0 + \Sigma(\tb)} \,, \\
& |\Delta_{12} {\mathcal F}_\tau|_{0, s_0 + M +  \beta}   \lesssim \tN_{n}^{2 \tau + M}  \tN_{n - 1}  \varepsilon^{M} \gamma^{-M}\| w_1 - w_2 \|_{s_0 + \Sigma(\tb)}\,. 
\end{aligned}
\end{equation}
We now shortly describe how to estimate $\Delta_{12} \bE_{n + 1}$ (recall the expression given in \eqref{2003.1}). By  Lemma \ref{lemma proiettori decadimento} and the estimates in \eqref{stime R nu i1 i2}, we have that
\begin{equation}\label{Delta 12 rRn bot}
\begin{aligned}
|\Delta_{12} \Pi_{\tN_n}^\bot \bE_n|_{- M, s_0} & \lesssim \tN_n^{- \tb} |\Delta_{12}  \bE_n|_{- M, s_0 + \tb}  \lesssim \tN_n^{- \tb} \tN_{n - 1}   \| w_1 - w_2 \|_{s_0 + \Sigma(\tb)}\,, \\
|\Delta_{12} \Pi_{\tN_n}^\bot \bE_n|_{- M , s_0+ \tb} & \leq |\Delta_{12}  \bE_n|_{- M , s_0+ \tb}  \lesssim  \tN_{n - 1}  \| w_1 - w_2\|_{s_0 + \Sigma(\tb)}\,. 
\end{aligned}
\end{equation}
 All the terms appearing in the two integrals in the formula of $\bE_{n + 1}$ (see \eqref{2003.1}) can be all estimated in a similar way. Hence, for sake of simplicity in the exposition, we concentrate on the estimate of the term 
 \begin{equation}\label{stima Delta 12 quadratico E n + 1}
\bB := \int_0^1 \bV_{\tau} \wrt \tau \,,  \quad \bV_{\tau} := {\mathcal F}_{- \tau} \bE_n \Psi_n {\mathcal F}_\tau\,. 
 \end{equation}
 Clearly it is enough to estimate $\bV_{\tau}$ uniformly with respect to $\tau \in [0, 1]$. 
 By applying the estimates \eqref{stime cal Rn rid}, \eqref{stime R nu i1 i2}, \eqref{stime Delta 12 Psin}, \eqref{stime Psin neumann}, \eqref{stime Phi n inv - Id}, \eqref{stima delta 12 exp Psi n} (using also that, by the ansatz \eqref{ansatz}, $\| w_i \|_{s_0 + \Sigma(\tb)}\leq 1$, $i = 1,2$)
 plus a repeated use of the triangular inequality and Lemma \ref{proprieta standard norma decay}-$(ii)$, one can show that $\Delta_{12} \bB$ satisfies the estimates
 \begin{equation}\label{stime delta 12 cal Q tau cal V tau}
 \begin{aligned}
&  |\Delta_{12} \bB |_{- M, s_0} \lesssim \tN_n^{2 \tau + 1 + M} \tN_{n - 1}^{- 2 \ta}  \varepsilon^{M} \gamma^{-M} \varepsilon_{M} \gamma^{- (M-1)} \| w_1 - w_2 \|_{s_0 + \Sigma(\tb)}\,, \\
&  |\Delta_{12} \bB|_{- M, s_0 + \tb} \lesssim \tN_n^{2 \tau + 1 + M} \tN_{n - 1}^{1 -  \ta}  \varepsilon^{M} \gamma^{-M} \varepsilon_{M} \gamma^{- (M-1)} \| w_1 - w_2 \|_{s_0 + \Sigma(\tb)} \,.
 \end{aligned}
 \end{equation}
 Collecting the estimates \eqref{Delta 12 rRn bot}, \eqref{stime delta 12 cal Q tau cal V tau}, we obtain that 
  $$
 \begin{aligned}
 & | \Delta_{12} \bE_{n + 1}   |_{- M, s_0} \lesssim \Big( \tN_n^{- \tb} \tN_{n - 1}   + \tN_n^{2 \tau + 1 + M} \tN_{n - 1}^{- 2 \ta} \varepsilon^{M} \gamma^{- M}\Big) \varepsilon_{M} \gamma^{-(M-1)} \| w_1 - w_2 \|_{s_0 + \Sigma(\tb)}\,, \\
& | \Delta_{12} \bE_{n + 1}   |_{- M, s_0 + \tb} \lesssim  \Big( \tN_{n - 1}  + \tN_n^{2 \tau + 1 + M} \tN_{n - 1}^{1 -  \ta} \varepsilon^{M} \gamma^{-M} \ \Big)  \varepsilon_{M} \gamma^{-(M-1)}  \| w_1 - w_2 \|_{s_0 + \Sigma(\tb)}\,,
 \end{aligned}
 $$
 which imply the claimed bounds \eqref{stime R nu i1 i2} at the step $n + 1$,  arguing as in \eqref{milan merda 0}-\eqref{forza inter 1}. 
 
 \noindent
 We now verify the estimates \eqref{r nu - 1 r nu i1 i2} at the step $n + 1$. Recalling \eqref{frittata di maccheroni} and by Lemma \ref{proprieta standard norma decay}-$(v)$, together with the estimate \eqref{stime R nu i1 i2}, we have, for any $j \in \Z^2 \setminus \{ 0 \}$,
 \begin{equation}
 \begin{aligned}
 |\Delta_{12}(\tz_{n + 1}(j) - \tz_n(j))| & = |\Delta_{12} \widehat{\bE}_n(0)_j^j|  \lesssim |j|^{- M} |\Delta_{12} \bE_n|_{- M, s_0} \\
 & \lesssim |j|^{- M} \tN_{n - 1}^{- \ta}  \varepsilon_{M} \gamma^{-(M-1)} \| w_1 - w_2 \|_{s_0 + \Sigma(\tb)} \,,
 \end{aligned}
 \end{equation}
 which is the first estimate in \eqref{r nu - 1 r nu i1 i2} at the step $n + 1$. Moreover, using that $\tz_{n + 1}(j) = \tz_0(j) + \sum_{k = 0}^n \tz_{k + 1} - \tz_k$ and recalling also that $|\Delta_{12} \tz_0| \lesssim \lambda^{\theta} |j|^{- 1} \| w_1 - w_2 \|_{s_0 + \Sigma(\tb)}$, we have, 
 \begin{equation}
 	 \begin{aligned}
 		|\Delta_{12} \tz_{n + 1}(j)| & \leq | \Delta_{12}\tz_0(j)| + \sum_{k = 0}^n | \Delta_{12}(\tz_{k + 1} - \tz_k) | \\
 		& \lesssim \Big( \lambda^{\theta} |j|^{- 1}  + |j|^{- M}   \sum_{k = 0}^n \tN_{k - 1}^{- \ta}  \varepsilon_{M} \gamma^{-(M-1)} \Big) \| w_1 - w_2 \|_{s_0 + \Sigma(\tb)} \\
 		& \lesssim \lambda^{\theta} |j|^{- 1} \| w_1 - w_2 \|_{s_0 + \Sigma(\tb)} \,,
 	\end{aligned}
 \end{equation}
 by using \eqref{vare.vare}, the fact that $M > 1$, and that the series $\sum_{k \geq 0} \tN_{k - 1}^{- \ta} < \infty$ is convergent and $\varepsilon \gamma^{-1}\ll 1$. Hence, the second estimate in \eqref{r nu - 1 r nu i1 i2} at the step $n + 1$ is proved and the proof of ${\bf (S3)}_{n + 1}$ is completed. 
\\[1mm]
	\noindent
	{\sc Proof of ${\bf (S4)}_{n + 1}$.} 
Let $ \omega \in \Omega_{n+1}^{\gamma}(w_1) $.
By \eqref{insiemi di cantor rid} and the induction hypothesis \eqref{inclusione cantor riducibilita S4 nu}  (i.e.\  ${\bf(S4)}_n$), we have that 
$ \Omega_{n+1}^{\gamma}(w_1) \subseteq  \Omega_{n}^{\gamma}(w_1) \subseteq  \Omega_{n}^{\gamma- \rho}(w_2) $. 
Moreover, $ \omega \in \Omega_n^{\gamma - \rho}(w_2) \subseteq \Omega_n^{\gamma/2}(w_2)  $ because $\rho \leq \gamma/2$. 
Thus $\Omega_{n + 1}^\gamma(w_1) \subseteq \Omega_n^{\gamma - \rho}(w_2) \subseteq \Omega_n^{\gamma/2}(w_2) $. We also have that \eqref{def.calLn calDn calQn} holds 
for any $\omega \in \mathtt \Omega_{n+1}^{\gamma}(w_1)$, as well estimate \eqref{r nu - 1 r nu i1 i2} 
 for any $j \in \Z^2 \setminus \{ 0 \}$, namely, for any $ \omega \in \Omega_{n + 1}^\gamma(w_1)$
\begin{equation}\label{stima delta 12 tz n j Omega n + 1 gamma}
|\Delta_{12} \mu_n(j; \omega)| = |\Delta_{12} \tz_n(j; \omega)|  \lesssim \lambda^{\theta}  |j|^{- 1}  \| w_1 -  w_2  \|_{ s_0  + \Sigma(\tb)}^{\rm sup}\,. 
\end{equation}
Let $\ell \in \Z^{\nu}\setminus\{0\}$ and $j,j'\in\Z^2\setminus\{0\}$, with $|\ell|\leq \tN_{n}$ and $\pi^\top(\ell)+j-j' =0$. We distinguish two regimes: 
\begin{equation}\label{regimi S3 n+1 ridu}
{\rm min}\{ |j|, |j'| \} \geq \tN_n^\tau \quad \text{and} \quad {\rm min}\{ |j|, |j'| \} \leq \tN_n^\tau\,. 
\end{equation}
We then organize the rest of the proof in two claims.

\medskip

\noindent
{\bf Claim 1.} Let $\omega \in \Omega_{n+1}^{\gamma}(w_1)$. If
$\min \{|j|, |j'|\} \geq \tN_{n}^\tau $, having $\lambda^{\theta-1} \gamma^{- 1} \ll 1 $ small enough,
then
\begin{equation}
	|\im\,\lambda\, \omega \cdot \ell + \mu_n(j ;w_2) - \mu_n(j' ; w_2)| \geq \frac{\lambda(\gamma - \rho)}{\langle \ell \rangle^\tau |j'|^\tau}\,. 
\end{equation}
\noindent
{\bf Proof of the Claim 1.} By \eqref{def.calLn calDn calQn} and the estimates \eqref{stime qn}, we have, for any $j\in\Z^2 \setminus\{0\}$, that
$
|\mu_n(j; w_2)| \leq C \lambda^{\theta} |j|^{- 1} 
$
for some constant $C > 0$. Therefore, using that ${\rm min}\{ |j|, |j'| \} \geq \tN_n^\tau$ and recalling that $\omega \in {\mathtt D} {\mathtt C}(2 \gamma, \tau)$, we obtain
\begin{equation}\label{inter juve 0}
\begin{aligned}
 |\im \, \lambda\, \omega \cdot \ell + \mu_n(j ; w_2) - \mu_n(j';w_2)| & \geq \lambda |\omega \cdot \ell| - C \lambda^{\theta} \big( |j|^{- 1} + |j'|^{- 1} \big) \\
 & \geq \frac{2 \lambda\,  \gamma}{\langle \ell \rangle^\tau} - \frac{2 C \lambda^{\theta}}{{\rm min}\{ |j|, |j'| \}} \\
 & \geq  \frac{2\lambda \,\gamma}{\langle \ell \rangle^\tau} - \frac{2 C \lambda^{\theta}}{\tN_n^\tau} \geq \frac{\lambda \,\gamma}{\langle \ell \rangle^\tau} \,,
\end{aligned}
\end{equation}
provided that
$$
 2C \lambda^{\theta-1} \gamma^{- 1} \langle \ell \rangle^\tau \tN_n^{- \tau} \leq 1 \,, \quad \forall\, \ell \in \Z^\nu \,, \quad 0 < |\ell| \leq \tN_n\,.
$$
The latter condition is verified by taking $\lambda^{\theta-1} \gamma^{- 1} \ll 1$ small enough, by \eqref{piccolo.ansatx}. The claimed statement then follows directly by \eqref{inter juve 0} since $\gamma \geq \gamma - \rho$ and $\langle \ell \rangle^\tau \leq \langle \ell \rangle^\tau |j'|^\tau$. 

\medskip

\noindent
We now analyze the other regime in \eqref{regimi S3 n+1 ridu}, namely when ${\rm min} \{ |j|, |j'| \} \leq \tN_n^\tau$.
\\[1mm]
\noindent
{\bf Claim 2.} Let $\omega \in \Omega_{n+1}^{\gamma}(w_1)$. If $\min \{|j|, |j'|\} \leq \tN_{n}^\tau $ and $   \tN_n^{\tau + \tau^2} \lambda^{\theta-1} \rho^{- 1}  \| w_1 - w_2 \|_{s_0 + \Sigma(\tb)}\ll 1 $ is small enough, then
\begin{equation}
	|\im\,\lambda\, \omega \cdot \ell + \mu_n(j; w_2) - \mu_n(j'; w_2)| \geq \frac{\gamma - \rho}{\langle \ell \rangle^\tau |j'|^\tau}\,. 
\end{equation}
{\bf Proof of the Claim 2.} 
By \eqref{stima delta 12 tz n j Omega n + 1 gamma}, and using that $\omega \in \Omega_{n+1}^{\gamma}(w_1)$ and $ \min \{|j|, |j'|\} \leq \tN_{n}^\tau $, we have that
\begin{equation}\label{inter juve 1}
\begin{aligned}
 |\im \, \lambda\, \omega \cdot \ell  + \mu_n(j; w_2) &- \mu_n(j'; w_2)|  \geq  |\im\,\lambda\, \omega \cdot \ell + \mu_n(j; w_1)   - \mu_n(j' ; w_1)|  \\
 &  - |\Delta_{12} \mu_n(j)| - |\Delta_{12} \mu_{n}(j')| \\
 & \geq \frac{\lambda \gamma}{\langle \ell \rangle^\tau |j'|^{\tau}} -  C \lambda^{\theta} \big(|j|^{- 1} + |j'|^{- 1}\big) \| w_1 - w_2 \|_{s_0 + \Sigma(\tb)}  \\
 & \geq \frac{\lambda \gamma}{\langle \ell \rangle^\tau |j'|^{\tau}} -  2 C \lambda^{\theta} \| w_1 - w_2 \|_{s_0 + \Sigma(\tb)}\geq \frac{\lambda (\gamma - \rho)}{\langle \ell \rangle^\tau |j'|^{\tau}} \,, 
\end{aligned}
\end{equation}
provided the following condition holds
\begin{equation}\label{inter juve 2}
2 C \lambda^{\theta-1} \rho^{- 1} \langle \ell \rangle^\tau |j'|^\tau \| w_1 - w_2 \|_{s_0 + \Sigma(\tb)} \leq 1\,. 
\end{equation}
Since $\langle \ell \rangle \leq \tN_n$ and the momentum conservation $\pi^\top(\ell) + j - j' = 0$ implies that $|j - j'| \lesssim |\ell| \lesssim \tN_n$, one has that 
$$
|j'| \leq {\rm min} \{ |j|, |j'| \} + |j - j'| \lesssim \tN_n^\tau + \tN_n \lesssim \tN_n^\tau\,.
$$
Hence the condition \eqref{inter juve 2} is implied by
$$
K \lambda^{\theta-1} \rho^{- 1} \tN_n^{\tau + \tau^2} \| w_1 - w_2 \|_{s_0 + \Sigma(\tb)} \leq 1
$$
for some large constant $K \gg 1$, and the claimed estimate \eqref{inter juve 1} holds.

\bigskip

\noindent
Finally, Claims 1 and 2 imply that, if $\omega \in \Omega_{n + 1}^\gamma(w_1)$, then, for any $\ell \in \Z^\nu\setminus\{0\}$ and $j, j' \in \Z^2 \setminus \{ 0 \}$, with $ |\ell| \leq \tN_n$ and $\pi^\top(\ell) + j - j' = 0$, we have that 
$$
| \lambda\, \omega \cdot \ell + \mu_n(j; w_2) - \mu_n(j'; w_2) | \geq \frac{\lambda(\gamma - \rho)}{\langle \ell \rangle^\tau |j'|^\tau} \,, 
$$
which is exactly the fact that $\omega \in \Omega_{n + 1}^{\gamma - \rho}(w_2)$. The proof of ${\bf(S4)}_{n+1}$ is then concluded. This concludes also the proof of Proposition \ref{prop riducibilita}.
 \end{proof}

 \subsection{KAM reducibility: convergence}\label{sez convergenza KAM}
 
 The eigenvalues of the normal form $\bD_{n}$ in \eqref{def.calLn calDn calQn} are Cauchy sequences and so they converge to some limit.
 \begin{lem}\label{lemma blocchi finali}
 	For any $j \in \Z^2 \setminus \{ 0 \}$, the sequence 
 	$\{ \widetilde \mu_n(j) \}_{n\in\N}$, defined in \eqref{def.calLn calDn calQn},
 	converges to some limit
	\begin{equation}\label{def cal N infty nel lemma}
	\mu_\infty(j) = \im \, \beta \,\tL(j) + \tz_\infty(j)\,, \quad \mu_\infty(j)= \mu_\infty(j;\,\cdot\,):\tD\tC(2\gamma,\tau)\to \C \,,
\end{equation}
 	satisfying the following estimates 
 	\begin{equation}\label{stime forma normale limite}
 		\begin{aligned}
 			& |  \mu_\infty(j) - \widetilde \mu_n(j) |^{{\rm Lip}(\gamma)} = | \tz_\infty(j) - \widetilde \tz_n(j) |^{{\rm Lip}(\gamma)} \lesssim  \tN_{n - 1}^{- \ta} \varepsilon_{M} \gamma^{-(M-1)} |j|^{- M}  \,, \\
 			& | \tz_\infty(j)|^{{\rm Lip}(\gamma)} \lesssim  \lambda^{\theta}  |j|^{- 1} \,. 
 		\end{aligned}
 	\end{equation}
 	In addition, if we assume $ w(\vf,x)=\odd(\vf,x)$, then $\mu_\infty(j;\,\cdot\,):\tD\tC(2\gamma,\tau)\to \im\,\R$ for any $j \in \Z^2 \setminus \{ 0 \}$.
 \end{lem}
 
\begin{proof}
	By Proposition \ref{prop riducibilita} and by
	 \eqref{lambdaestesi}, we have that the sequence $\{\widetilde \mu_n(j;\omega)\}_{n\in\N}\subset \C$ is Cauchy on the closed set $\tD\tC(2\gamma,\tau)$,  therefore it is convergent for any $ \omega \in \tD\tC(2\gamma,\tau)$. The estimates in \eqref{stime forma normale limite}
	 follow then by a telescoping argument with  \eqref{lambdaestesi}.
	 Finally, if we assume $ w(\vf,x)=\odd(\vf,x)$, then $\{\widetilde \mu_n(j;\omega)\}_{n\in\N}\subset \im\,\R$ for any $j \in \Z^2 \setminus \{ 0 \}$, by \eqref{reversibility reality auto}, implying that  $ \mu_\infty(j;\omega) \in  \im\,\R$ for any $j \in \Z^2 \setminus \{ 0 \}$. 
\end{proof}
 We define the set $\Omega_\infty^\gamma$ of the non-resonance conditions for the final eigenvalues as 
 \begin{equation}\label{cantor finale ridu}
 	\begin{aligned}
 		\Omega_\infty^\gamma 
 		:= \Big\{ \omega \in \tD\tC& (2 \gamma,\tau) 
 		 \, : \, |\im \, \lambda\, \omega \cdot \ell + \mu_\infty(j) -  \mu_\infty(j')| \geq \frac{2 \lambda \,\gamma}{\langle \ell \rangle^\tau | j' |^\tau } \,, \\
 		&  \forall \ell \in \Z^\nu \setminus \{ 0 \}\,, \ \ 	j,j' \in \Z^2 \setminus \{ 0 \}\,, \ \  \pi^\top(\ell) + j-j' =0\Big\}\,. 
 	\end{aligned}
 \end{equation}
 
 \begin{lem}\label{prima inclusione cantor}
 	We have $\Omega_\infty^\gamma \subseteq \cap_{n \geq 0} \, \Omega_n^\gamma$.
 \end{lem}
 
 \begin{proof}
 	We prove by induction that $\Omega_\infty^\gamma \subseteq \Omega_n^\gamma$ 
 	for any integer $n \geq 0$. 
 	The statement is trivial for $n=0$,
 	since $\Omega_0^\gamma := \tD\tC(2 \gamma,\tau)$ 
 	(see Proposition \ref{prop riducibilita}).
 	We now assume by induction that $\Omega_\infty^\gamma \subseteq \Omega_n^\gamma$ for some $n \geq 0$ and we  show that $\Omega_\infty^\gamma \subseteq \Omega_{n + 1}^\gamma$. Let $\omega \in \Omega_\infty^\gamma$, $\ell \in \Z^\nu \setminus \{ 0 \}$, $j, j' \in \Z^2 \setminus \{ 0 \}$, with $\pi^\top( \ell) +j-j'=0$ and $|\ell|  \leq \tN_n$. By \eqref{stime forma normale limite}, \eqref{cantor finale ridu}, we compute
 	$$
 	\begin{aligned}
 		|\im \, \lambda \,\omega \cdot \ell + \mu_n(j) - \mu_n(j')|   &\geq |\im\, \lambda \,\omega \cdot \ell + \mu_\infty(j) - \mu_\infty(j')| - |\mu_\infty(j) - \mu_n(j)| \\
 		&  \quad  - |\mu_\infty(j') - \mu_n(j')| \\
 		& \geq \frac{2\lambda\,\gamma}{\langle \ell \rangle^\tau | j' |^\tau } - C \tN_{n - 1}^{- \ta}  \varepsilon_{M} \gamma^{-(M-1)}\big(  |j|^{- M} + |j'|^{- M} \big)  \\
 		& \geq \frac{\lambda\,\gamma}{\langle \ell \rangle^\tau | j' |^\tau } \,,
 	\end{aligned}
 	$$
 	for some positive constant $C>0$, provided
 	\begin{equation}\label{frittata di maccheroni 10}
 		C  \tN_{n - 1}^{- \ta}\varepsilon^{M}\gamma^{- M}\langle \ell \rangle^\tau  \big(  | j' |^\tau |j|^{- M} + |j'|^{\tau - M} \big) \leq 1\,.
 	\end{equation}
 	Recalling that $M >  \tau$ by \eqref{definizione.param.KAM}, we have that $|j'|^{\tau - M} \leq 1$. Moreover, using that $ |\ell| \leq \tN_n$ and the momentum condition $\pi^\top(\ell)  + j-j'=0$, we have that
 	\begin{equation*}
 		\begin{aligned}
 			& |j'| \leq |j| + |j - j'| \leq |j| + \tN_n \lesssim \tN_n |j|\,, 
 		\end{aligned}
 	\end{equation*}
	and therefore, using again that $M > \tau$, we have 
	$$
	| j' |^\tau |j|^{- M}  \lesssim \tN_n^\tau |j|^{\tau - M} \lesssim \tN_n^\tau\,.
	$$
 	We then deduce that 
	$$
	C  \tN_{n - 1}^{- \ta}\varepsilon^{M}\gamma^{- M}\langle \ell \rangle^\tau  \Big(  | j' |^\tau |j|^{- M} + |j'|^{\tau - M} \Big) \leq C_0 \tN_n^{2 \tau} \tN_{n - 1}^{- \ta} \varepsilon^{M} \gamma^{- M}
	$$
	for some large constant $C_0 \gg 0$. Hence, the condition \eqref{frittata di maccheroni 10} is verified since 
	$$
	C_0 \tN_n^{2 \tau} \tN_{n - 1}^{- \ta} \varepsilon^{M} \gamma^{- M} \leq 1\,,
	$$
	recalling \eqref{definizione.param.KAM}, $\varepsilon \gamma^{-1} = \lambda^{\theta-1}\gamma^{-1}\ll 1$ in \eqref{piccolo.ansatx} and the smallness condition \eqref{KAM smallness condition}
 	 We conclude that $\omega \in \Omega_{n + 1}^\gamma$ and the claim is proved. 
 \end{proof}

 Now we define the sequence of invertible maps
 \begin{equation}\label{trasformazioni tilde ridu}
 	\widetilde \Phi_n := \Phi_0 \circ \Phi_1 \circ \ldots \circ \Phi_n \,, \quad n \in \N\,. 
 \end{equation}
 
 
 \begin{prop}\label{lemma coniugio finale}
 	Let $S > s_0 + \Sigma(\tb)$. There exists $\delta := \delta (S, \tau,\nu) > 0$ such that, if  \eqref{ansatz} \eqref{KAM smallness condition} are verified, then the following holds. 
 	For any $\omega \in \Omega_\infty^\gamma$, the sequence $(\widetilde \Phi_n)_{n\in\N}$
 	converges in norm $| \cdot |_{0, s}^{{\rm Lip}(\gamma)}$ 
 	to an invertible map $\Phi_\infty$, 
 	satisfying, for any $s_0\leq s\leq S-\Sigma(\tb)$,
 	\begin{equation}\label{stima Phi infty}
 		\begin{aligned}
 			& |\Phi_\infty^{\pm 1} - \widetilde \Phi_n^{\pm 1}|_{0, s}^{{\rm Lip}(\gamma)} 
 			\lesssim_{s} 
 			\tN_{n + 1}^{2 \tau + 1} \tN_n^{- \ta} \varepsilon^{- M} \gamma^{- M} \|  w \|_{s + \Sigma(\tb)}^{\Lip(\gamma)} \,, \\
 			& |\Phi_\infty^{\pm 1} - {\rm Id}|_{0, s}^{{\rm Lip}(\gamma)} 
 			\lesssim_{s}  \varepsilon^{M} \gamma^{- M} \|  w \|_{s + \Sigma(\tb)}^{\Lip(\gamma)} \,.
 		\end{aligned}
 	\end{equation}
 	The operators $\Phi_\infty^{\pm 1} : H^s_0 \to H^s_0$ are real and momentum preserving. Moreover, for any $\omega\in \Omega_\infty^\gamma$, one has the conjugation
 	\begin{equation}\label{cal L infty e}
 		{\bL}_{\infty} := \Phi_\infty^{- 1} {\bL}_{0} \Phi_\infty =\lambda\, \omega \cdot \partial_\vphi + {\bD}_\infty\,, \quad {\bD}_\infty := {\rm diag}_{j \in \Z^2 \setminus \{ 0 \}} \mu_\infty(j)  \,,
 	\end{equation} 
 	where the operator $\bL_{0}$ is given in \eqref{bL0_inizioKAM}
 	and  the final eigenvalues $\mu_\infty(j)$ are given in Lemma \ref{lemma blocchi finali}. In addition, if we assume that $ w(\vf,x)=\odd(\vf,x)$, the operators $\Phi_{\infty}^{\pm 1}$ are reversibility preserving and ${\bf D}_\infty$ is a reversible operator.
 \end{prop}
 
 \begin{proof}
 	The existence of the invertible map $\Phi_\infty^{\pm 1}$ and the estimates \eqref{stima Phi infty} follow by 
 	\eqref{stime Psi n rid}, 
 	\eqref{trasformazioni tilde ridu}, 
 	arguing as in Corollary 4.1 in \cite{BBM-Airy}. 
 	By \eqref{trasformazioni tilde ridu}, Lemma \ref{prima inclusione cantor} 
 	and Proposition \ref{prop riducibilita}, one has 
 	$\widetilde \Phi_n^{- 1} {\bL}_0 \widetilde \Phi_n 
 	=\lambda\, \omega \cdot \partial_\vphi + {\bD}_n + {\bE}_n$ for all $n \geq 0$. 
 	The claimed statement then follows by passing to the limit as $n \to \infty$, 
 	by using \eqref{stime cal Rn rid}, \eqref{stima Phi infty} and Lemma \ref{lemma blocchi finali}. 
 \end{proof}

 \section{Inversion of the linearized operator ${\mathcal L}$}\label{sez:inverti}
 In this section we show that the linearized operator ${\mathcal L}$ in \eqref{operatore linearizzato} is invertible and we prove tame estimates for its inverse. By Propositions \ref{proposizione trasporto}, \ref{prop coniugio cal L L1}, \ref{prop normal form lower orders}, \ref{lemma coniugio finale}, using also Lemma \ref{proprieta standard norma decay}-$(i)$, for any $\omega \in \Omega_\infty^\gamma$, we define
 \begin{equation}\label{def cal W infty}
 {\mathcal W}_\infty := \cB_{\perp} \circ {\bf \Phi}_M \circ \Phi_\infty
 \end{equation}
 and we have that 
 \begin{equation}
 {\mathcal W}_\infty^{- 1} {\mathcal L} {\mathcal W}_\infty = {\bL}_{\infty} \,, \quad \forall \,\omega \in \Omega_\infty^\gamma\,,
 \end{equation}
 with the maps ${\mathcal W}_\infty\,,\, {\mathcal W}_\infty^{- 1} : H^s_0(\T^{\nu + 2}) \to H^s_0(\T^{\nu + 2})$ are real and momentum preserving and, if \eqref{ansatz} holds with $\sigma_0 = \Sigma(\mathtt b)$, they satisfy  for any $s_0 \leq s \leq S-\Sigma(\tb)$, the bounds
 \begin{equation}\label{stime cal W infty pm 1}
 \| {\mathcal W}_\infty^{\pm 1} h \|_s^{\Lip(\gamma)} \lesssim_s \| h \|_s^{\Lip(\gamma)} + \|  w \|_{s + \Sigma(\tb)}^{\Lip(\gamma)} \| h \|_{s_0}^{\Lip(\gamma)}\,.
 \end{equation}
In addition, if $ w = {\rm odd}(\vf, x)$, then ${\mathcal W}_\infty, {\mathcal W}_\infty^{- 1}$ are reversibility preserving.

\noindent
 Now we define the set $\Lambda_\infty^\gamma$ as
 \begin{equation}\label{prime di Melnikov}
 \begin{aligned}
 \Lambda_\infty^\gamma  := \Big\{  \omega  & \in \tD\tC(2 \gamma, \tau )  \, : \,  |\im \,\lambda \,\omega \cdot \ell + \mu_\infty(j) | \geq \frac{2 \lambda \,  \gamma}{\langle \ell \rangle^\tau } \,,  \\
 & \quad  \forall\, \ell \in \Z^\nu \,,\  \  j \in \Z^2 \setminus \{ 0 \}, \ \ \pi^\top(\ell) + j = 0 \Big\}\,. 
 \end{aligned}
 \end{equation}
First, we discuss the invertibility of the diagonal operator ${\bf L}_\infty$ given in \eqref{cal L infty e}. 
 \begin{lem}\label{invertibilita cal L infinito}
For any $\omega \in  \Lambda_\infty^\gamma$, the linear operator ${\bf L}_\infty(\omega)$ is invertible and its inverse ${\bf L}_\infty^{- 1}$ satisfies the bound, for any $s\geq 0$,
\begin{equation}\label{stima tame L infty inverse}
\| {\bf L}_\infty^{- 1} h  \|_s^{\Lip(\gamma)} \lesssim \lambda^{- 1} \gamma^{- 1} \| h \|_{s + 2 \tau + 1}^{\Lip(\gamma)} \,. 
\end{equation}
In addition, if $ w = {\rm odd}(\vf, x)$, then  ${\bf L}_\infty^{- 1}$ is reversible.
 \end{lem}
 \begin{proof}
 If $\omega \in \Lambda_\infty^\gamma$, the inverse ${\bf L}_\infty(\omega)^{- 1}$ is defined by 
 \begin{equation}\label{giornata.uggiosa}
 	 {\bf L}_\infty(\omega)^{- 1} h(\vphi, x) = \sum_{
 			\ell \in \Z^\nu \setminus\{0\}\,,\, j \in \Z^2 \setminus \{ 0 \} \atop
 			\pi^\top(\ell) + j = 0	} \frac{1}{\eta_{\ell j}(\omega)} \widehat h(\ell, j) e^{\im \ell \cdot \vphi} e^{\im j \cdot x} \,,
 \end{equation}
 where $\eta_{\ell j }(\omega) = \im \,\lambda \,\omega \cdot \ell + \mu_\infty(j;\omega)$ and we used the momentum condition $\pi^\top(\ell) + j = 0$ to get the restriction $\ell \neq 0$, since we already have $j \neq 0$.
 This clearly implies the bound, by \eqref{prime di Melnikov}, for any $s\geq 0$,
 \begin{equation}\label{inter sassuolo 0}
 \| {\bf L}_\infty(\omega)^{- 1} h \|_s \lesssim \lambda^{- 1} \gamma^{- 1} \| h \|_{s + \tau}\,.
 \end{equation}
 Moreover, given $\omega_1, \omega_2 \in \Lambda_\infty^\gamma$, by \eqref{prime di Melnikov} Lemma \ref{lemma blocchi finali} and using $\lambda^{\theta-1}\gamma^{-1}\leq 1$ in \eqref{piccolo.ansatx}, we have that 
 $$
 \begin{aligned}
 \Big| \frac{1}{\eta_{\ell j}(\omega_1)} - \frac{1}{\eta_{\ell j}(\omega_2)} \Big|
 & \lesssim \frac{\lambda\,|\ell| |\omega_1 - \omega_2| + |\mu_\infty(j; \omega_1) - \mu_\infty(j; \omega_2)|}{|\eta_{\ell j}(\omega_1)| |\eta_{\ell j}(\omega_2)|} \\
 & \lesssim \langle \ell \rangle^{2 \tau}\lambda^{- 2} \gamma^{- 2} \big( \lambda \,|\ell| + \lambda^{\theta} \gamma^{- 1} \big) |\omega_1 - \omega_2| \\
 & \lesssim \lambda^{- 1} \gamma^{- 2} \langle \ell \rangle^{2 \tau + 1} |\omega_1 - \omega_2| \,,
 \end{aligned}
 $$
 which implies that, for any $s\geq 0$,
 \begin{equation}\label{inter sassuolo 1}
 \big\| \big( {\bf L}_\infty(\omega_1)^{- 1} - {\bf L}_\infty(\omega_2)^{- 1}\big) h \big\|_s \lesssim \lambda^{- 1} \gamma^{- 2} \| h \|_{s + 2 \tau + 1} \,.
 \end{equation}
 Now, using that, for any $h(\omega) \in H^{s + 2 \tau + 1}_0$ and any $\omega_1, \omega_2 \in \Lambda_\infty^\gamma$,
 $$
 \begin{aligned}
& \| {\bf L}_\infty(\omega_1)^{- 1} h(\omega_1) - {\bf L}_\infty(\omega_2)^{- 1} h(\omega_2)  \|_s  \\
& \leq  \| {\bf L}_\infty(\omega_1)^{- 1}\big( h(\omega_1) - h(\omega_2) \big) \|_s + \|\big( {\bf L}_\infty(\omega_1)^{- 1} - {\bf L}_\infty(\omega_2)^{- 1} \big) h(\omega_2) \|_s  \,,
\end{aligned}
 $$
 then, together with the bounds \eqref{inter sassuolo 0}, \eqref{inter sassuolo 1}, we obtain the claimed estimate \eqref{stima tame L infty inverse}. Finally, if we assume that $ w = {\rm odd}(\vf,x)$, then  $\bL_{\infty}, {\bf L}_\infty^{- 1}$ are reversible operators, by \eqref{giornata.uggiosa}, Lemma \ref{lemma blocchi finali} and Lemma \ref{lemma real rev matrici}.
 \end{proof}
 We are now in position to state the following proposition on the invertibility of the linearized operator $\cL$ in \eqref{operatore linearizzato}.
 \begin{prop}\label{invertibilita linearizzato}
Let  $\bar\sigma := \Sigma(\tb) + 2 \tau + 1$, with $\Sigma(\tb)$ given in \eqref{definizione.param.KAM}. Let $S > s_0 + \bar \sigma$. There exists $\delta := \delta (S, \tau, \nu) > 0$ such that, if  \eqref{ansatz} \eqref{KAM smallness condition} are verified (with $\sigma = \bar \sigma$), then the following holds. For any $\omega \in \Omega_\infty^\gamma \cap \Lambda_\infty^\gamma$ the linear operator $\cL={\mathcal L}(\omega)$, defined \eqref{operatore linearizzato}, is invertible on the subspace of the quasi-periodic traveling waves $\tau_\vs h(\vf)=h(\vf-\pi(\vs))$, $\vs\in\R$, and its inverse $\cL^{-1}={\mathcal L}(\omega)^{- 1}$ satisfies the tame bound, for any $s_0 \leq s \leq S-\bar\sigma$ and any $h\in H^{s + \bar \sigma}_0(\T^{\nu+2})$
\begin{equation}\label{stima lin inverso finale}
\| {\mathcal L}^{- 1} h \|_s^{\Lip(\gamma)} \lesssim_s \lambda^{- 1} \gamma^{- 1}\Big( \| h \|_{s + \bar \sigma}^{\Lip(\gamma)} + \|  w \|_{s + \bar \sigma}^{\Lip(\gamma)} \| h \|_{s_0 + \bar \sigma}^{\Lip(\gamma)} \Big)\,.
\end{equation}
Moreover, $\cL^{-1}h$ is a quasi-periodic traveling wave.
In addition, if $ w = {\rm odd}(\vf, x)$, then $ {\mathcal L}^{- 1}$ is reversible.
  \end{prop}
 \begin{proof}
 For any $\omega \in \Omega_\infty^\gamma \cap \Lambda_\infty^\gamma$, one has that ${\mathcal L}^{- 1} = {\mathcal W}_\infty^{- 1} {\bf L}_\infty^{- 1} {\mathcal W}_\infty$. Hence, the estimate \eqref{stima lin inverso finale} follows by \eqref{stime cal W infty pm 1} and Lemma \ref{invertibilita cal L infinito}, using also the ansatz \eqref{ansatz} with $\sigma = \bar \sigma$. Since $\bL_{\infty}^{\pm 1}$ are momentum preserving by \eqref{cal L infty e}, \eqref{giornata.uggiosa}, and the maps $\cW_{\infty}^{\pm 1}$ are momentum preserving by \eqref{def cal W infty} and Propositions \ref{proposizione trasporto}, \ref{prop normal form lower orders}, \ref{lemma coniugio finale}, then, by Lemmata \ref{A.mom.cons}, \ref{lem:mom_pres}, we conclude that $\cL^{-1}h$ is a quasi-periodic traveling wave, whenever is $h$. Finally, if we assume that $ w = {\rm odd}(\vf,x)$, then $\cL^{-1}$ is reversible by Lemma \ref{invertibilita cal L infinito} and by the fact that $\cW_{\infty}^{\pm1}$ are reversibility preserving.
 \end{proof}

 \section{The Nash Moser scheme}\label{sezione:NASH}
 In this section we construct the solution of the equation ${\mathcal F}(v) = 0$ in  \eqref{internal.rescaled} by means of a Nash Moser nonlinear iteration. 
We denote by $\Pi_n$ the orthogonal projector $\Pi_{\tN_n}$ (on quasi-periodic traveling waves, see \eqref{def:smoothings}) 
on the finite dimensional space
$$
{\mathcal H}_n := \big\{ w \in L_0^2(\T^{\nu + 2}, \R) \, \ \text{ such that \eqref{condtraembedd} holds } \, \, : \, \  
w = { \Pi}_n w  \big\} \,,
$$
and $\Pi_n^\bot := {\rm Id} - \Pi_n$. 
The projectors $ \Pi_n $, $ \Pi_n^\bot$ satisfy the 
usual smoothing properties in Lemma \ref{lemma:smoothing}, namely 
\begin{equation}\label{smoothing-u1}
\|\Pi_{n} v \|_{s+b}^{\Lip(\gamma)} 
\leq \tN_{n}^{b} \| v \|_{s}^{\Lip(\gamma)} \,,
\quad \ 
\|\Pi_{n}^\bot v \|_{s}^{\Lip(\gamma)} 
\leq \tN_n^{- b} \| v \|_{s + b}^{\Lip(\gamma)} \,,
\quad \ 
s,b \geq 0 \,.
\end{equation}
Given $\tau, \tN_{0}>0$, we define the constants 
\begin{equation} \label{costanti nash moser}
\begin{aligned}
& \tN_{n} := \tN_{0}^{\chi^n}, \quad n \geq 0, \quad \tN_{- 1} := 1\,, \quad \chi = 3/2, \quad  \kappa :=  12( \bar \sigma + 1) + 1\,, \\
&  \mathtt a_1 := {\rm max}\{ 6 \bar \sigma + 13\,,\, \chi^2(\tau + \tau^2 + 1) + \chi(2 \bar \sigma + 1) + 1 \} \,, \\
& \bar \tau := 2 \bar \sigma + 4 + \mathtt a_1 + \chi(\tau + \tau^2 + 1), \quad {\mathtt b}_1 := 2 \bar \sigma + 4  +\chi^{-1} ( \mathtt a_1 + \kappa)\,.
\end{aligned}
\end{equation}
where $\bar \sigma > 1$ is given in Proposition \ref{invertibilita linearizzato}. 

 \begin{rem}\label{remark.su.param.NM}
	Let us describe of the parameters in \eqref{costanti nash moser} and their role in the following Proposition \ref{iterazione-non-lineare}. The parameter $\ta_{1}>0$ appears in the negative exponent of the estimate in \eqref{stima.F.to.0} and measures how the low regularity norm of the nonlinear functional is fastly decreasing to 0 at each approximate solution of the iteration. The parameter $\tb_{1}>0$ appears in the regularity index $s_0+\tb_{1}$ of the second estimate in \eqref{stima.alta.NM}, which is the high regularity norms of the approximate solution and the functional evaluated at it that are allowed to diverge, with exponent rate $\kappa>0$. The parameter $\bar\tau>0$ appears as an exponent in the smallness condition \eqref{nash moser smallness condition}. The parameter $\bar\sigma>0$ accounts for the loss of derivatives of the KAM iteration. The role of the parameters in \eqref{costanti nash moser} is comparable with the ones in \eqref{definizione.param.KAM} for the KAM reducibility of the linearized operators (see also Remark \ref{remark.su.param.KAM}). The differences are in their definitions, due to the nonlinear iteration and the loss of derivative $\bar\sigma>\Sigma(\tb)>0$ (see Proposition \ref{invertibilita linearizzato}), and the exponent $\kappa>0$ in the divergence of the high regularity norms.
\end{rem}

Note that Proposition \ref{lemma.approx} provides an approximate solution $v_{{\rm app} ,N}$ which is a quasi periodic traveling wave such that, for any $s\geq 0$, $\| v_{{\rm app},N}\|_{s}^{\Lip(\gamma)}$ is of size 1  and  $\| {\mathcal F}(v_{{\rm app},N}) \|_s^{\Lip(\gamma)}$ is uniformly bounded with respect to $\lambda \gg 1$. We state these properties as follows.


\begin{lem}
	{\bf (Initialization of the Nash-Moser iteration).}
	Let $N\in\N$ and $\gamma\in (0,1)$ as in \eqref{condi.per.NM.dopo}.
	For any $s \geq 0$, there exists a constant $C(s)>0$ such that
	\begin{equation}\label{stima soluzione approssimata nash moser}
	\begin{aligned}
	&	\| {\mathcal F}(v_{{\rm app},N}) \|_s^{\Lip(\gamma)} \leq C(s)\,,\quad
		\| v_{{\rm app},N} \|_s^{\Lip(\gamma)} \leq C(s)\,, \\
		& \inf_{\omega \in \Omega_\gamma} \| v_{\rm app, N}(\cdot; \omega)\|_s \geq K(s) \lambda^{- \mathtt c}
		\end{aligned}
	\end{equation}
	uniformly in $\lambda\geq \bar \lambda \gg 1$ sufficiently large, for some constants $C(s), K(s) > 0$. In addition, if the forcing term $f$ in \eqref{internal.rescaled} is ${\rm even}(\vf, x)$, then $v_{\rm app, N} = {\rm odd}(\vf, x)$. 
\end{lem}


The lower bound in \eqref{stima soluzione approssimata nash moser}, that has been proved in Proposition \ref{lemma.approx}, is actually not needed in the nonlinear iteration. It will be used to prove Theorem \ref{teo principale beta plane} in the next section. 

We now prove the following proposition.

\begin{prop}\label{iterazione-non-lineare} 
{\bf (Nash-Moser)} 
Let $\tau > 0$, and let $\Omega_{\gamma}$ be defined in \eqref{Omega.gamma.SA}
With the notation in \eqref{costanti nash moser},
there exist $ \delta \in (0, 1)$, $C_* > 0$ such that if
\begin{equation}  \label{nash moser smallness condition} 
	\begin{aligned}
		& \tN_{0}^{\overline \tau} \lambda^{- 1}  \leq \delta \,, \quad \tN_{0} := \gamma^{- 1}\,, \quad \gamma := \lambda^{- \mathtt c}\,,  \quad \text{for some} \ \ \  0 < \mathtt c < {\rm min} \{ \bar\tau\,^{-1}, \tfrac13(2-\alpha) \}\,,
	\end{aligned}
\end{equation}
then the following properties hold for all $n \geq 0$:
\\[1mm]
\noindent $({\mathcal P}1)_{n}$
There exist a constant $C_0 > 0$ large enough and 
$w_n - w_0 : {\mathcal G}_n \to {\mathcal H}_{n - 1} $, $n \geq 1$, $w_0 := v_{{\rm app},N}$, ${\mathcal H}_{- 1} := \{ 0 \}$, 
satisfying
\begin{equation} \label{stima.bassa.NM}
\| w_{n} \|_{s_0 + \overline \sigma}^{\Lip(\gamma)} \leq C_0 \,.
\end{equation}
If $n \geq 1$, 
the difference $ h_n := w_n -  w_{n - 1}$ satisfies 
$\| h_1 \|_{s_0 + \overline \sigma}^{\Lip(\gamma)} 
\lesssim \lambda^{- 1} \gamma^{-1}$ and 
\begin{equation} \label{hn}
\| h_n \|_{s_0 + \overline \sigma}^{\Lip(\gamma)} 
\leq C_*  \tN_{n-1}^{2 \overline \sigma + 1} \tN_{n-2}^{-\mathtt a_1} \lambda^{- 1}  \,.
\end{equation}
The sets $\{ {\mathcal G}_n\}_{n \geq 0}$ are defined as follows. For $n\in\N_{0}$ we define ${\mathcal G}_0 := \Omega_\gamma$ and
\begin{equation}\label{G-n+1}
{\mathcal G}_{n + 1} := {\mathcal  G}_n \cap 
\big( \Omega_\infty^{\gamma_n}(w_n) 
\cap \Lambda_\infty^{\gamma_n}(w_n) \big), 
\end{equation} 
where $ \gamma_{n}:=\gamma (1 + 2^{-n}) $ 
and the sets $\Omega_\infty^{\gamma_n}(w_n)$, $\Lambda_\infty^{\gamma_n}(w_n)$ 
are defined in \eqref{cantor finale ridu}, \eqref{prime di Melnikov}.
Moreover, the function $w_n(\vf,x;\omega)$ is a quasi-periodic traveling wave. In addition, if the forcing term $f$ in \eqref{internal.rescaled} is ${\rm even}(\vf, x)$, then $w_n$ is ${\rm odd}(\vf, x)$;
\\[1mm]
\noindent $({\mathcal P}2)_{n}$
On the set ${\mathcal G}_n$, we have the estimate
\begin{equation} \label{stima.F.to.0}
\| {\mathcal F}(w_n) \|_{s_{0}}^{\Lip(\gamma)} 
\leq C_* \tN_{n - 1 }^{- \mathtt a_1}\, ;
\end{equation}
\noindent $({\mathcal P}3)_{n}$
On the set ${\mathcal G}_n$, we have the estimate 
\begin{equation}\label{stima.alta.NM}
	\| w_n \|_{s_0 + \mathtt b_1}^{\Lip(\gamma)}+ 	\| \cF(w_n) \|_{s_0 + \mathtt b_1}^{\Lip(\gamma)}
	\leq C_*  \tN_{n-1}^{\kappa} \,.
\end{equation}
\end{prop}

\begin{proof}
To simplify notations, in this proof we write $\| \cdot \|_s$ instead of $\| \cdot \|_s^{\Lip(\gamma)}$. 


%
%

\smallskip

\noindent
{\sc Proof of $({\mathcal P}1, 2, 3)_0$}.
By \eqref{stima soluzione approssimata nash moser} with $s=s_0,s_0+\tb_{1}$, we have $\| w_0\|_{s} = \| v_{{\rm app},N} \|_{s} \leq C(s)$ and
$\| {\mathcal F} (w_0 ) \|_s = \| {\mathcal F} (v_{{\rm app},N} ) \|_s \leq C(s)$. Then \eqref{stima.bassa.NM}, \eqref{stima.F.to.0} and \eqref{stima.alta.NM} hold taking $ \tfrac12 C_0, \tfrac12C_*(s) \geq C(s)$ sufficiently large. In particular, we have
\begin{equation}\label{stima.w0}
	\| w_0 \|_{s_0+\bar\sigma} \leq \tfrac12 C_0 \leq C_0\,.
\end{equation}

\smallskip

\noindent
{\sc Assume that $({\mathcal P}1,2,3)_n$ hold for some $n \geq 0$, 
and prove $({\mathcal P}1,2,3)_{n+1}$.}
By $({\mathcal P}1)_n$, one has $\| w_n\|_{s_0 + \bar \sigma} \leq C_0$, for some $C_0 \gg 0$ large enough, independent of $n\in\N_0$.  
The assumption 
\eqref{nash moser smallness condition} implies 
the smallness condition 
$\lambda^{- 1} \gamma^{- 1} \leq \delta$ 
of Proposition \ref{invertibilita linearizzato}
by taking $\overline \tau (k_0, \tau, \nu)$ large enough and $S = s_0 + \mathtt b_1$. 
Then Proposition \ref{invertibilita linearizzato} applies 
to the linearized operator 
\begin{equation}\label{definizione cal Ln}
{\mathcal L}_n \equiv {\mathcal L}(w_n) : = \di_{v} {\mathcal F}(w_n) \,.
\end{equation}
This implies that, for any $\omega \in {\mathcal G}_{n + 1}$, the operator ${\mathcal L}_n(\omega)$ admits a right inverse ${\mathcal L}_n(\omega)^{- 1}$ satisfying the tame estimates, for any $s_0 \leq s \leq s_0 + \tb_{1}$,
\begin{equation}\label{stima Tn}
\| {\mathcal L}_n^{-1} h \|_s 
\lesssim_s \lambda^{- 1}\gamma^{-1} \big( \| h \|_{s + \overline \sigma}
+ \| w_n \|_{s + \overline \sigma} \,  
\| h \|_{s_0 + \overline \sigma} \big)\,,
\end{equation}
using the bound $\gamma_n = \gamma(1 + 2^{- n}) \in [\gamma, 2\gamma]$.
By \eqref{stima Tn} for $s= s_0$ and \eqref{stima.bassa.NM}, 
one has 
\begin{equation}\label{stima Tn norma bassa} 
\| {\mathcal L}_n^{- 1} h \|_{s_0}
\lesssim \lambda^{- 1}\gamma^{- 1} \| h \|_{s_0 + \overline \sigma}\,.  
\end{equation}
We define the successive approximation 
\begin{equation}\label{soluzioni approssimate}
w_{n + 1} := w_n + h_{n + 1} \,, \quad 
{h}_{n + 1} :=  - \Pi_n {\mathcal L}_n^{-1}  \Pi_n {\mathcal F}(w_n) 
\in {\mathcal H}_{n} \,,
\end{equation}
defined for any $\omega \in {\mathcal G}_{n + 1}$, and the remainder  
$$
Q_n := {\mathcal F}(w_{n+1}) - {\mathcal F}(w_n) - {\mathcal L}_n h_{n+1}\,.
$$
We now estimate $h_{n + 1}$. Note that for any $s, \mu \geq 0$, 
\begin{equation}\label{stima wn - w0 w0}
\begin{aligned}
& \| w_n \|_{s+\mu}  \leq \| w_0 \|_{s + \mu} + \| w_n - w_0 \|_{s + \mu} \stackrel{w_0 = v_{\rm app, N}\,,\, \eqref{stima soluzione approssimata nash moser}}{\lesssim_{s, \mu}}   1 + \tN_{n - 1}^\mu  \| w_n - w_0 \|_s  \\
& \lesssim_{s, \mu} \tN_{n - 1}^\mu(1 + \| w_n \|_s)\, \quad \text{and hence by \eqref{stima Tn}}, \\
& \| {\mathcal L}_n^{- 1} h \|_{s_0 + 1}
\lesssim \tN_{n - 1}\lambda^{- 1}\gamma^{- 1} \| h \|_{s_0 + \overline \sigma + 1}
\end{aligned}
\end{equation}
By the estimates \eqref{stima Tn}, \eqref{stima Tn norma bassa}, \eqref{stima.bassa.NM}, \eqref{smoothing-u1}, \eqref{costanti nash moser}, \eqref{nash moser smallness condition}, \eqref{stima.F.to.0}, \eqref{stima wn - w0 w0}, by Lemma \eqref{stime tame cal F}-$(i)$, we get that 
\begin{equation}\label{dawn1}
	\begin{aligned}
		\| h_{n + 1} \|_{s_0} & \lesssim \lambda^{- 1}\gamma^{- 1} \| \Pi_n {\mathcal F}(w_n) \|_{s_0 + \overline \sigma} \lesssim \tN_n^{\bar \sigma} \lambda^{- 1} \gamma^{- 1} \| {\mathcal F}(w_n) \|_{s_0} \,, \\
		\| h_{n + 1} \|_{s_0 + \bar \sigma} & \lesssim \tN_n^{\bar \sigma} \| {\mathcal L}_n^{-1}  \Pi_n {\mathcal F}(w_n) \|_{s_0} \lesssim \tN_n^{2 \bar \sigma} \lambda^{- 1} \gamma^{- 1} \| {\mathcal F}(w_n) \|_{s_0}\,.		
		\end{aligned}		
		\end{equation}
and 
	\begin{equation}\label{bach.1}
	\begin{aligned}
		\| h_{n + 1} \|_{s_0 + \mathtt b_1} & 
		\lesssim \| \cL_{n}^{-1}  \Pi_{n} \cF(w_n)  \|_{s_0+\tb_{1}} \\
		& \lesssim \lambda^{- 1} \gamma^{- 1} \Big( \| \Pi_n {\mathcal F}(w_n) \|_{s_0 + \mathtt b_1 + \overline \sigma}  + \| w_n \|_{s_0 + \mathtt b_1 + \overline \sigma} \| \Pi_n {\mathcal F}(w_n) \|_{s_0 +  \overline \sigma}  \Big) \\
		& \lesssim \lambda^{- 1} \gamma^{- 1} \tN_n^{2 \bar \sigma } \Big( \|  {\mathcal F}(w_n) \|_{s_0 + \mathtt b_1}  + \| w_n \|_{s_0 + \mathtt b_1 } \|  {\mathcal F}(w_n) \|_{s_0}  \Big) \\ 
		& \lesssim \lambda^{- 1}  \tN_n^{2 \bar \sigma +1} \Big( \|  {\mathcal F}(w_n) \|_{s_0 + \mathtt b_1}  +C_* \tN_{n-1}^{-\ta_{1}} \| w_n \|_{s_0 + \mathtt b_1 } \Big) \\ 
		& \lesssim \lambda^{- 1}  \tN_n^{2 \bar \sigma +1} \Big( \|  {\mathcal F}(w_n) \|_{s_0 + \mathtt b_1}  + \| w_n \|_{s_0 + \mathtt b_1 } \Big)\,.
	\end{aligned}
\end{equation}
Moreover, \eqref{soluzioni approssimate} and estimate \eqref{bach.1}, using that $\lambda^{-1}\ll1$, imply that 
\begin{equation}\label{bach.2}
\begin{aligned}
		\| w_{n + 1} \|_{s_0 + \mathtt b_1} & \leq \| w_n \|_{s_0 + \mathtt b_1} + \| h_{n + 1} \|_{s_0 + \mathtt b_1} \\
		& \lesssim \tN_{n}^{2 \bar \sigma + 1} \Big( \|  {\mathcal F}(w_n) \|_{s_0 + \mathtt b_1}  + \| w_n \|_{s_0 + \mathtt b_1 } \Big)\,.
\end{aligned}
\end{equation}
Next, we estimate $\|{\mathcal F}(w_{n + 1}) \|_{s_0}$. By the definition of $h_{n+1}$ in \eqref{soluzioni approssimate}, recalling that $\Pi_{n} + \Pi_{n}^\perp ={\rm Id}$ and $\Pi_{n}^\perp \Pi_{n} = 0$, we obtain that, for any $\omega \in {\mathcal G}_{n + 1}$, 
\begin{equation}\label{forma cal F w n + 1}
\begin{aligned}
{\mathcal F}(w_{n + 1}) & = {\mathcal F}(w_n) + {\mathcal L}_n h_{n + 1} + Q_n \\
& = \cF(w_n) - \cL_{n} ({\rm Id} - \Pi_{n}^\perp) \cL_{n}^{-1} \Pi_{n} \cF(w_n) + Q_n \\
& = \Pi_n^\bot {\mathcal F}(w_n) + {\mathcal L}_n \Pi_n^\bot {\mathcal L}_n^{- 1} \Pi_n {\mathcal F}(w_n) + Q_n \\
& = \Pi_n^\bot {\mathcal F}(w_n) + [{\mathcal L}_n, \Pi_n^\bot ]{\mathcal L}_n^{- 1} \Pi_n {\mathcal F}(w_n) + Q_n \,. 
\end{aligned}
\end{equation}
We estimate separately the three terms in \eqref{forma cal F w n + 1}. 
Note that, by \eqref{nash moser smallness condition}, we have
 $\gamma^{- 1} = \tN_{0} \leq \tN_{n}$.
By \eqref{smoothing-u1}, Lemma \ref{stime tame cal F}-$(i)$, \eqref{stima.bassa.NM} and \eqref{nash moser smallness condition}, we have
\begin{equation}\label{mozart1}
	\begin{aligned}
		& \| \Pi_n^\bot {\mathcal F}(w_n) \|_{s_0}  \leq  \tN_{n}^{- \mathtt b_1} \| {\mathcal F}(w_n) \|_{s_0 + \mathtt b_1 } \,, \quad  \| \Pi_n^\bot {\mathcal F}(w_n) \|_{s_0+\tb_{1}}  \leq   \| {\mathcal F}(w_n) \|_{s_0 + \mathtt b_1 } \,.
	\end{aligned}
\end{equation}
By \eqref{stima.bassa.NM}, Lemma \ref{stime tame cal F}-$(i)$,$(ii)$, \eqref{smoothing-u1}, \eqref{stima Tn}, \eqref{nash moser smallness condition}, \eqref{stima.F.to.0}, \eqref{stima wn - w0 w0}, we have
\begin{equation}\label{mozart2}
	\begin{aligned}
		\|  [&{\mathcal L}_n, \Pi_n^\bot] {\mathcal L}_n^{- 1} \Pi_n {\mathcal F}(w_n) \|_{s_0}   \\
		& \lesssim \lambda^{\theta} \,\tN_{n}^{1-\tb_{1}}\Big(  \| \cL_{n}^{- 1} \Pi_n \cF(w_n) \|_{s_0 + \tb_{1}} +  \| w_n \|_{s_0 + \tb_{1}}  \| \cL_{n}^{- 1} \Pi_n \cF(w_n) \|_{s_0 + 1} \Big)  \\
		& \lesssim \lambda^{\theta-1}\gamma^{-1} \,\tN_{n}^{1- \tb_{1}} \Big(  \| \Pi_n \cF(w_n) \|_{s_0 + \tb_{1}+\bar\sigma} +\tN_{n}  \| w_n \|_{s_0 + \tb_{1}+\bar\sigma}  \|  \Pi_n \cF(w_n) \|_{s_0 + \bar\sigma} \Big) \\
		&  \lesssim \lambda^{\theta-1}\,\tN_{n}^{2- \tb_{1}} \Big( \tN_{n}^{\bar\sigma} \| \cF(w_n) \|_{s_0 + \tb_{1}} +\tN_{n}^{1+2\bar\sigma}  \| w_n \|_{s_0 + \tb_{1}}  \| \cF(w_n) \|_{s_0} \Big) \\
		& \lesssim  \lambda^{\theta-1}\,\tN_{n}^{2\bar\sigma +3- \tb_{1}} \Big( \| \cF(w_n) \|_{s_0 + \tb_{1}} + C_* \tN_{n-1}^{-\ta_{1}}  \| w_n \|_{s_0 + \tb_{1}} \Big)  \\
		& \lesssim  \lambda^{\theta-1}\,\tN_{n}^{2\bar\sigma +3- \tb_{1}} \Big( \| \cF(w_n) \|_{s_0 + \tb_{1}} +  \| w_n \|_{s_0 + \tb_{1}} \Big) \,, \\
	\end{aligned}
\end{equation}
\begin{equation}\label{mozart3}
	\begin{aligned}
			\| [&{\mathcal L}_n, \Pi_n^\bot] {\mathcal L}_n^{- 1} \Pi_n {\mathcal F}(w_n) \|_{s_0+\tb_{1}} = 	\| [\Pi_n,\cL_{n}] {\mathcal L}_n^{- 1} \Pi_n {\mathcal F}(w_n) \|_{s_0+\tb_{1}}   \\
		& \lesssim \lambda^{\theta} \,\tN_{n}^{}\Big(  \| \cL_{n}^{- 1} \Pi_n \cF(w_n) \|_{s_0 + \tb_{1}} +  \| w_n \|_{s_0 + \tb_{1}+1}  \| \cL_{n}^{- 1} \Pi_n \cF(w_n) \|_{s_0 + 1} \Big)  \\
		& \lesssim \lambda^{\theta-1}\gamma^{-1} \,\tN_{n}^{} \Big(  \| \Pi_n \cF(w_n) \|_{s_0 + \tb_{1}+\bar\sigma} +\tN_{n}  \| w_n \|_{s_0 + \tb_{1}+\bar\sigma}  \|  \Pi_n \cF(w_n) \|_{s_0 + \bar\sigma} \Big) \\
		&  \lesssim \lambda^{\theta-1}\,\tN_{n}^{2} \Big( \tN_{n}^{\bar\sigma} \| \cF(w_n) \|_{s_0 + \tb_{1}} +\tN_{n}^{1+2\bar\sigma}  \| w_n \|_{s_0 + \tb_{1}}  \| \cF(w_n) \|_{s_0} \Big) \\
		& \lesssim  \lambda^{\theta-1}\,\tN_{n}^{2\bar\sigma +3} \Big( \| \cF(w_n) \|_{s_0 + \tb_{1}} + C_* \tN_{n-1}^{-\ta_{1}}  \| w_n \|_{s_0 + \tb_{1}} \Big) \\
		& \lesssim  \lambda^{\theta-1}\,\tN_{n}^{2\bar\sigma +3} \Big( \| \cF(w_n) \|_{s_0 + \tb_{1}} +  \| w_n \|_{s_0 + \tb_{1}} \Big)\,.
	\end{aligned}
\end{equation}
By Lemma \ref{stime tame cal F}-$(iii)$, \eqref{smoothing-u1}, \eqref{stima.bassa.NM}, \eqref{dawn1} and \eqref{nash moser smallness condition}, \eqref{bach.1}, \eqref{stima.F.to.0}, using also that $\lambda^{\theta-1}=\lambda^{\alpha-2+\tc}<1$ for $\lambda\gg 1$, we have
\begin{equation}\label{mozart4}
	\begin{aligned}
		\| Q_n \|_{s_0} & \lesssim \lambda^{\theta} \tN_{n}^2 \| h_{n + 1} \|_{s_0}^2  \lesssim \lambda^{\theta}  \tN_{n}^{2 \bar \sigma + 2} \lambda^{- 2} \gamma^{- 2} \| {\mathcal F}(w_n) \|_{s_0}^2 \\
		& \lesssim  \lambda^{\theta-2}  \tN_{n}^{2 \bar \sigma + 4}  \| {\mathcal F}(w_n) \|_{s_0}^2  \lesssim \tN_{n}^{2 \bar \sigma + 4} \lambda^{- 1}  \| {\mathcal F}(w_n) \|_{s_0}^2 \,, \\
		\| Q_n \|_{s_0+\tb_{1}} & \lesssim \lambda^{\theta} \| h_{n+1}\|_{s_0+\tb_{1}+1} \| h_{n+1} \|_{s_0+1}  \lesssim \lambda^{\theta} \tN_{n}^{2} \| h_{n+1}\|_{s_0+\tb_{1}} \| h_{n+1} \|_{s_0} \\
		& \lesssim \lambda^{\theta-1}\gamma^{-1} \tN_{n}^{4\bar\sigma+3} \Big( \| \cF(w_n) \|_{s_0+\tb_{1}}  + \| w_{n} \|_{s_0+\tb_{1}} \Big) \| \cF(w_n) \|_{s_0}  \\
		& \lesssim \lambda^{\theta-1}\gamma^{-1} \tN_{n}^{4\bar\sigma+3} \Big( \| \cF(w_n) \|_{s_0+\tb_{1}}  + \| w_{n} \|_{s_0+\tb_{1}} \Big) C_{*} \tN_{n-1}^{-\ta_{1}} \\
		& \lesssim\ \tN_{n}^{4\bar\sigma+4} \Big( \| \cF(w_n) \|_{s_0+\tb_{1}}  + \| w_{n} \|_{s_0+\tb_{1}} \Big) \,.
 	\end{aligned}
\end{equation}
Therefore, by \eqref{bach.2} and estimating \eqref{forma cal F w n + 1}  with \eqref{mozart1}, \eqref{mozart2}, \eqref{mozart3},  \eqref{mozart4},  we have proved the following inductive inequalities for any $n \geq 0$ and on the set ${\mathcal G}_{n + 1}$: 
\begin{equation}\label{stime.induttive.proof.NM}
	\begin{aligned}
		& 	\| w_{n + 1} \|_{s_0 + \mathtt b_1}  + \| \cF(w_{n+1})\|_{s_0+\tb_{1}} \lesssim \tN_{n}^{4 \bar \sigma + 4} \Big( \|  {\mathcal F}(w_n) \|_{s_0 + \mathtt b_1}  + \| w_n \|_{s_0 + \mathtt b_1 } \Big)\,, \\
		& \| \cF(w_{n+1}) \|_{s_0} \lesssim \tN_{n}^{2 \bar \sigma + 3 - \mathtt b_1} \Big( \|  {\mathcal F}(w_{n}) \|_{s_0 + \mathtt b_1}  + \| w_n \|_{s_0 + \mathtt b_1 } \Big)\ +   \tN_{n}^{2 \bar \sigma + 4} \lambda^{- 1}  \| {\mathcal F}(w_n) \|_{s_0}^2\,.
	\end{aligned}
\end{equation}

\medskip

\noindent
{\bf Proof of $({\mathcal P}1)_{n + 1}$.} By \eqref{dawn1}, \eqref{stima.F.to.0}, \eqref{nash moser smallness condition} and \eqref{costanti nash moser}, we obtain
\begin{equation}\label{hn.n+1}
	\| h_{n+1} \|_{s_0+\bar\sigma} \lesssim \tN_{n}^{2\bar\sigma} \tN_{n-1}^{-\ta_{1}} \gamma^{-1} \lambda^{-1} \lesssim \tN_{n}^{2\bar\sigma+1} \tN_{n-1}^{-\ta_{1}} \lambda^{-1}\,,
\end{equation}
which proves \eqref{hn} at the step $n+1$. By \eqref{stima.w0}, by the definition of the constants in \eqref{costanti nash moser} and by the smallness condition in \eqref{nash moser smallness condition}, the estimate \eqref{stima.bassa.NM} at the step $n + 1$ follows since 
\begin{equation}
	\begin{aligned}
		\| w_{n + 1} \|_{s_0 + \bar \sigma}  \leq  \| w_0 \|_{s_0 + \bar \sigma} + \sum_{k = 1}^{n + 1} \| h_k \|_{s_0 + \bar \sigma}   \leq \tfrac12 C_0 + C_*\sum_{k = 1}^\infty \tN_{k-1}^{2 \overline \sigma} \tN_{k -2}^{-\mathtt a_1} \lambda^{- 1}  \leq C_0\,,
	\end{aligned}
\end{equation}
for $\tN_{0}=\gamma^{-1}=\lambda^\tc\gg 1$ large enough and, if needed, $C_0>1$ a bit larger. We prove now that $w_{n+1}$ is a quasi-periodic traveling wave. By induction, we have that $w_{n}$ is a quasi-periodic. Moreover, the function $h_{n+1}$ in \eqref{soluzioni approssimate} is a quasi-periodic traveling wave by Proposition \ref{invertibilita linearizzato}, Lemma \ref{stime tame cal F}, Lemma \ref{A.mom.cons}, and the definition of the projections only on momentum preserving sites, see \eqref{def:smoothings}. Therefore, by \eqref{soluzioni approssimate}, $w_{n+1}:=w_n + h_{n+1}$ is a quasi-periodic traveling wave. Moreover, in the case where the forcing term $f$ in \eqref{internal.rescaled} is ${\rm even}(\vf, x)$, if by induction hypothesis $w_n = {\rm odd}(\vf, x)$, then ${\mathcal F}(w_n) = {\rm even}(\vf, x)$, ${\mathcal L}_n^{- 1}$ is reversibile, implying that $h_{n + 1} = {\rm odd}(\vf, x)$ and hence also $w_{n + 1} = w_n + h_{n + 1} = {\rm odd}(\vf, x)$. 

\noindent
{\bf Proof of $({\mathcal P}2)_{n + 1}, ({\mathcal P}3)_{n + 1}$.} By the induction estimate \eqref{stime.induttive.proof.NM} and using the induction hypothesis on $({\mathcal P}2)_{n }, ({\mathcal P}3)_{n }$, we obtain that 
\begin{equation}
	\begin{aligned}
			\| w_{n + 1} \|_{s_0 + \mathtt b_1}  + \| \cF(w_{n+1})\|_{s_0+\tb_{1}} & \leq C\, \tN_{n}^{4 \bar \sigma + 4} \Big( \|  {\mathcal F}(w_n) \|_{s_0 + \mathtt b_1}  + \| w_n \|_{s_0 + \mathtt b_1 } \Big)  \\
			& \leq C C_* \tN_{n}^{4\bar\sigma +4} \tN_{n-1}^{\kappa} \leq C_* \tN_{n}^{\kappa}\,,
	\end{aligned}
\end{equation}
by the choice of the constant $\kappa$ in \eqref{costanti nash moser} and taking $\tN_{0} \gg 1 $ large enough. This latter chain of inequalities proves $({\mathcal P}3)_{n + 1}$. Moreover, by \eqref{stime.induttive.proof.NM}, by the smallness condition \eqref{nash moser smallness condition} and by the choice of the constants in \eqref{costanti nash moser}, we have
\begin{equation}
	\begin{aligned}
		\| \cF(w_{n+1}) \|_{s_0} & \leq C\, \tN_{n}^{2 \bar \sigma + 3 - \mathtt b_1} \Big( \|  {\mathcal F}(w_n) \|_{s_0 + \mathtt b_1}  + \| w_n \|_{s_0 + \mathtt b_1 } \Big)\ +   C\,\tN_{n}^{2 \bar \sigma + 4} \lambda^{- 1}  \| {\mathcal F}(w_n) \|_{s_0}^2 \\
		& \leq C C_* \tN_{n}^{2 \bar \sigma + 3 - \mathtt b_1} \tN_{n-1}^{\kappa} +   C C_* \tN_{n}^{2 \bar \sigma + 4} \tN_{n-1}^{-2\ta_{1}} \lambda^{- 1}  \\
		& \leq C_* \tN_{n}^{-\ta_{1}} \,.
	\end{aligned}
\end{equation}
 The proof of the claimed statement is then concluded. 
\end{proof}

\section{Proof of Theorem \ref{teo principale beta plane} and measure estimates}\label{sez:measures}
Fix $\gamma = \lambda^{- \mathtt c}$  as in 
\eqref{costanti nash moser}, 
\eqref{nash moser smallness condition}, in particular with $0 <{ \mathtt c} < \bar \tau\,^{-1}$. Then we have that
$$
\tN_{0}^{\bar \tau} \lambda^{- 1} = \gamma^{- \bar \tau} \lambda^{- 1} = \lambda^{\mathtt c \bar \tau - 1} \ll 1
$$
by taking $\lambda \gg 0$ large enough, since $1 - \mathtt c \bar \tau > 0$. This implies that the smallness condition \eqref{nash moser smallness condition} is fullfilled. We define the set ${\mathcal G}_\infty$ as 
\begin{equation}\label{def cal G infty}
{\mathcal G}_\infty := \bigcap_{n \geq 0} {\mathcal G}_n 
\end{equation}
where the sets ${\mathcal G}_n$ are given in Proposition \ref{iterazione-non-lineare}-$({\mathcal P}1)_n$. 
Hence, by Proposition \ref{iterazione-non-lineare}-$({\mathcal P}1)_n$, using a telescoping argument, for any $\omega \in {\mathcal G}_\infty$,
the sequence $(w_n)_{n \geq 0}$ converges to $w_\infty \in H^{s_0 + \overline \sigma}_0$ 
with respect to the norm $\| \cdot  \|_{s_0 + \overline \sigma}^{\Lip(\gamma)}$, and 
\begin{equation}\label{conv vn v infty}
\| w_\infty \|_{s_0 + \overline \sigma}^{\Lip(\gamma)} \leq C_0, \quad \| w_\infty - w_n \|_{s_0 + \overline \sigma}^{\Lip(\gamma)} \lesssim \tN_{n}^{2 \overline \sigma} \tN_{n- 1}^{-\mathtt a_1} \lambda^{- 1} \gamma^{-1}, \quad \forall \,n \geq 1\,,
\end{equation}
and consequently, by the same arguments, recalling \eqref{stima soluzione approssimata nash moser},
\begin{equation}\label{lower.esti}
\begin{aligned}
	\inf_{\omega \in {\mathcal G}_\infty}\|  w_{\infty}(\cdot; \omega) \|_{s_0+\bar\sigma} & \geq  \inf_{\omega \in {\mathcal G}_\infty} \|  w_{0}(\cdot; \omega) \|_{s_0+\bar\sigma} - 	\|  w_{\infty} -w_0 \|_{s_0+\bar\sigma}^{\Lip(\gamma)}  \\
	&   \geq  K \lambda^{-\tc} - C_* \tN_{0}^{2\bar\sigma} \lambda^{-1}\gamma^{-1}  \geq \tfrac{K}{2} \lambda^{-\tc}  \,,
	\end{aligned}
\end{equation}
for $\lambda\gg 1$ large enough, by  \eqref{costanti nash moser}, \eqref{nash moser smallness condition} and $\gamma=\lambda^{-\tc}$.
 By  \eqref{def cal G infty} and Proposition \ref{iterazione-non-lineare}-$({\mathcal P}2)_n$, for any $\omega \in {\mathcal G}_\infty$, we have ${\mathcal F}(w_n) \to 0$ as $n \to \infty$. Therefore, the estimate \eqref{conv vn v infty} implies that ${\mathcal F}(w_\infty) = 0$ for any $\omega \in {\mathcal G}_\infty$, whereas estimates \eqref{conv vn v infty}-\eqref{lower.esti} imply \eqref{stima.grande.solutione}, after scaling back to 
 $v_{\lambda} := \lambda^{\theta} w_{\infty}=\lambda^{\alpha-1+\tc} w_{\infty}$ (recall \eqref{theta.def.ridu}). 
 If we assume that $f$ in \eqref{internal.rescaled} is ${\rm even}(\vf, x)$, then $w_n = {\rm odd}(\vf, x)$ for any $n \geq 0$ and hence also $w_\infty= {\rm odd}(\vf, x)$. The linearized equation at the quasi-periodic traveling wave solution $
v_\lambda(\lambda\, \omega t, x) = \lambda^\theta w_\infty(\lambda\, \omega t, x)$, obtained by linearizing \eqref{beta.waves.large.eq}, has the form
 $$
 \begin{aligned}
&  \partial_t h + L(\lambda\, \omega t)[h ] = 0\,, \quad L(\vf) := {\bf a}_0(\vf, x) \cdot \nabla + {\mathcal E}_0(\vf) \,, \\
& \ba_{0}(\vf, x):=    \fB\big[  v_\lambda ](\vf, x)  \,, \quad \cE_{0}(\vf)[h] :=  \nabla v_\lambda(\vf, x) \cdot \fB [h]\,, \quad \fB = \nabla^\perp (-\Delta)^{-1} \,. 
 \end{aligned}
 $$ 
 The linearized equation can be fully reduced to the constant coefficients equation
 $$
 \partial_t \phi + {\bf D}_\infty \phi = 0
 $$
 by the normal form scheme that we implemented in Section \ref{ridusezione}, see Propositions \ref{prop coniugio cal L L1}, \ref{prop normal form lower orders}, \ref{lemma coniugio finale}. In particular, since $w_\infty = {\rm odd}(\vf, x)$, we can apply Lemma \ref{lemma blocchi finali} and we deduce that the eigenvalues of the reduced diagonal operator ${\bf D}_\infty$ are purely imaginary. Consequently, we conclude that the quasi-periodic traveling wave solution $v_\lambda = \lambda^\theta w_\infty$ is linearly stable in the reversible case, and, for any $s \geq 0$, one has  $\| \phi(t) \|_{H^s_x} = \| \phi(0) \|_{H^s_x}$ for any $t \in \R$ and hence $\| h(t) \|_{H^s_x} \lesssim_s \| h(0 ) \|_{H^s_x}$ for any $t \in \R$. 


\noindent
 By setting $\Omega_\lambda := {\mathcal G}_\infty$, the proof of Theorem \ref{teo principale beta plane} is concluded once we estimate the Lebesgue measure of the set $\Omega \setminus \Omega_\lambda$, with $\Omega\subset \R^\nu$ as in \eqref{anello}. This is the content of the last part of the paper. 
 
 The main result is the following proposition.
 
 \begin{prop} \label{prop measure estimate finale}
 Let $\Omega\subset \R^\nu$ be given as in \eqref{anello} and let
 \begin{equation}\label{choice tau}
 \tau := \nu + 4\,.
 \end{equation}
Then, we have that $|\Omega \setminus {\mathcal G}_\infty| \lesssim \gamma$. As a consequence, with $\gamma = \lambda^{- \mathtt c}$  as in \eqref{costanti nash moser}, \eqref{nash moser smallness condition}, we have that $|\Omega \setminus {\mathcal G}_\infty| \lesssim \lambda^{-\tc}\to 0$ as $\lambda\to\infty$.
\end{prop}

The rest of this section is devoted to the proof 
of Proposition \ref{prop measure estimate finale}. 
By \eqref{def cal G infty}, \eqref{G-n+1}, \eqref{Omega.gamma.SA} and \eqref{DC.2gamma}, we have 
\begin{equation}\label{prima inclusione stime misura}
\Omega \setminus {\mathcal G}_\infty 
\subseteq \big( \Omega \setminus \tD\tC(\gamma,\tau) \big) \bigcup  \big( \tD\tC(\gamma,\tau) \setminus \Omega_{\gamma} \big)  \bigcup_{n \geq 0} ({\mathcal G}_n \setminus {\mathcal G}_{n + 1})\,.
\end{equation}
By Lemma \ref{lemma.equa.linear.approx} and the standard estimate $|\Omega\setminus \tD\tC(\gamma,\tau)|$, it remains to estimate the measure of ${\mathcal G}_n \setminus {\mathcal G}_{n + 1}$. By \eqref{G-n+1} and using elementary properties of set theory, one has that
\begin{equation}\label{seconda inclusione stime misura}
{\mathcal G}_n \setminus {\mathcal G}_{n + 1} \subseteq ({\mathcal G}_n \setminus  \Omega_\infty^{\gamma_n}(w_n)) \cup ({\mathcal G}_n \setminus  \Lambda_\infty^{\gamma_n}(w_n))\,
\end{equation}
where we recall that the sets $\Omega_\infty^{\gamma_n}(w_n)$, $\Lambda_\infty^{\gamma_n}(w_n)$ 
are defined in \eqref{cantor finale ridu}, \eqref{prime di Melnikov}.
Note that ${\mathcal G}_{n+1} \subseteq \cG_{0} = \Omega_{\gamma}$ for any $n \geq 0$,
by \eqref{G-n+1}.


\begin{prop}\label{Gn Omegan Lambdan}
For any $n \geq 0$, the following estimates hold:
\\[1mm]
\noindent $(i)$ ${\mathcal G}_0 \setminus \Omega_\infty^{\gamma_0}(w_0) \lesssim \gamma$ 
and, for any $n \geq 1$, 
$|{\mathcal G}_n \setminus  \Omega_\infty^{\gamma_n}(w_n)| 
\lesssim \gamma \tN_{n}^{- (\tau - 1)}$;
\\[1mm]
\noindent
$(ii)$ ${\mathcal G}_0 \setminus \Lambda_\infty^{\gamma_0}(w_0) \lesssim \gamma$ and for any $n \geq 1$,$|{\mathcal G}_n \setminus  \Lambda_\infty^{\gamma_n}(w_n)| \lesssim \gamma \tN_{n}^{- (\tau - 1)}$. 
\\[1mm]
\noindent
As a consequence $|{\mathcal G}_0 \setminus {\mathcal G}_1| \lesssim \gamma$ and for any $n \geq 1$, $|{\mathcal G}_n \setminus {\mathcal G}_{n + 1}| \lesssim \gamma \tN_{n }^{- (\tau - 1)}$.
\end{prop}

Propositions \ref{prop measure estimate finale} 
and \ref{Gn Omegan Lambdan}
are proved at the end of this section. 
Now 
we estimate the measure of the set ${\mathcal G}_n \setminus  \Omega_\infty^{\gamma_n}(w_n)$,  $n\geq 0$. 
The estimate of the measure of ${\mathcal G}_n \setminus  \Lambda_\infty^{\gamma_n}(v_n)$ can be done arguing similarly (it is actually even easier) and therefore it is omitted here. 
By the definitions \eqref{G-n+1}, \eqref{cantor finale ridu}, we get that 
\begin{equation}\label{quarta inclusione stime misura}
\begin{aligned}
& {\mathcal G}_n  \setminus \Omega_\infty^{\gamma_n}(w_n) \subseteq \bigcup_{(\ell, j, j') \in {\mathcal I}}{\mathcal R}_{\ell j j'}(w_n)\,,
\end{aligned}
\end{equation}
where 
\begin{equation}\label{def cal I res}
{\mathcal I} := \big\{ (\ell, j, j') \in (\Z^\nu \setminus \{0 \}) \times (\Z^2 \setminus \{ 0 \}) \times (\Z^2 \setminus \{ 0 \}) : \pi^\top(\ell) + j - j' = 0 \big\}\,,
\end{equation}
and 
\begin{equation}\label{def cal R vn l j j'}
\begin{aligned}
{\mathcal R}_{\ell j j'}(w_n) 
:= \Big\{ \omega \in {\mathcal G}_n : |\im\, \lambda\, \omega \cdot \ell + \mu_\infty(j; w_n) - \mu_\infty(j';  w_n) | < \frac{2 \gamma_n\, \lambda}{\langle \ell \rangle^\tau |j'|^\tau}   \Big\}\,.
\end{aligned}
\end{equation}
In the next lemma, we estimate the measure of the resonant sets ${\mathcal R}_{\ell j j'}(w_n)$.

\begin{lem}\label{stima misura risonanti sec melnikov}
For any $n \geq 0$, if ${\mathcal R}_{\ell j j'}(w_n) \neq \emptyset$, then we have that  $|{\mathcal R}_{\ell j j'}(w_n)| 
\lesssim \gamma \,\langle \ell \rangle^{-(\tau + 1)} |j'|^{-\tau}$.
\end{lem}

\begin{proof}
Recalling \eqref{def cal N infty nel lemma}, \eqref{cal L infty e},   
we write $\mu_\infty(j;\omega) \equiv \mu_\infty(j; \omega, w_n(\omega))$, with $(\ell, j, j') \in {\mathcal I}$ as in \eqref{def cal I res}, and we set 
$$
\phi(\omega) := \im \,\lambda \,\omega \cdot \ell + \mu_\infty(j; \omega) - \mu_\infty(j'; \omega)\,. 
$$ 
Since $\ell \neq 0$, we write 
$ \omega = \frac{\ell }{|\ell|} s + v$,  with $s>0$ and $v \cdot \ell = 0$.
We estimate the measure of the set 
\begin{equation}
	\begin{aligned}
		{\mathcal Q} & := \Big\{ s : |\psi(s)| < \frac{2 \gamma_n\, \lambda}{\langle \ell \rangle^\tau |j'|^\tau}, \ \ \frac{\ell }{|\ell|} \,s + v \in \Omega_{\gamma_n}  \Big\} \,,
	\end{aligned}
\end{equation}
where
\begin{equation}
	\begin{aligned}
			\psi(s) & := \phi\Big( \tfrac{\ell }{|\ell|} s + v\Big) = \im \,\lambda  |\ell| s + \mu_\infty(j; s) - \mu_\infty(j'; s) \,, \quad	\mu_\infty(j; s) \equiv \mu_\infty \Big(j; \tfrac{\ell }{|\ell|} s + v \Big)\,. 
	\end{aligned}
\end{equation}
By Lemma \ref{lemma blocchi finali}, we have, for any $j\in\Z^2\setminus\{0\}$,
$$
| \mu_\infty(j; s_1) - \mu_\infty(j; s_2)| \lesssim \gamma^{- 1}\lambda^{\alpha - 1+\tc} |j|^{- 1} |s_1 - s_2| \,,
$$
and therefore 
$$
\begin{aligned}
| \psi(s_1) - \psi(s_2)| & \geq \big( \lambda |\ell| - C \lambda^{\alpha - 1+\tc} \gamma^{- 1}  \big) |s_1 - s_2|\geq \frac{\lambda |\ell|}{2} |s_1 - s_2| \,,
\end{aligned}
$$
since  $\lambda^{\alpha - 2+\tc} \gamma^{- 1} = \lambda^{\alpha - 2(1-\tc)} \ll 1$ for $\lambda \gg 1$ large enough (note that $\alpha - 2(1-\tc) < 0$ by \eqref{nash moser smallness condition}). This implies that the measure of the set ${\mathcal Q}$ satisfies 
$$
|{\mathcal Q}| \lesssim \frac{2 \lambda \gamma_n}{\langle \ell \rangle^\tau |j'|^\tau } \frac{1}{\lambda |\ell|} \lesssim \frac{\gamma}{\langle \ell \rangle^{\tau + 1} |j'|^{\tau}}\,.
$$
By the Fubini theorem, we also get the claimed bound on $|{\mathcal R}_{\ell j j'}(w_n)|$.
\end{proof}

Several resonant sets $\cR_{\ell j j'}(w_n)$ are actually empty.

\begin{lem}\label{lemma triviale modi alti}
	Let $(\ell, j, j') \in {\mathcal I}$, with $|\ell| \leq \tN_{n}$ and ${\rm min}\{ |j|, |j'| \} \geq \tN_{n}^{\tau}$, $n \geq 1$. Then ${\mathcal R}_{\ell j j'}(w_n) = \emptyset$.
\end{lem}
\begin{proof}
	Let $\omega \in {\mathcal G}_n \subseteq \tD\tC(2\gamma_{n - 1},\tau)$, We have to prove that, if $0 < |\ell| \leq \tN_{n}$, ${\rm min}\{ |j|, |j'| \} \geq \tN_{n}^\tau$ and $ \pi^\top(\ell) + j - j' = 0$, then one has 
	\begin{equation}\label{lower bound R ljj vuoto}
		\begin{aligned}
			& |\im\, \lambda\, \omega \cdot \ell + \mu_\infty(j;w_n) - \mu_\infty(j';w_n)| \geq \frac{2\lambda \gamma_{n} }{\braket{\ell}^\tau} \geq \frac{2 \lambda \gamma_n}{\langle \ell \rangle^\tau |j'|^\tau} \,, 
		\end{aligned}
	\end{equation}
	which implies that ${\mathcal R}_{\ell j j'}(w_n) = \emptyset$. 
	By  Lemma \ref{lemma blocchi finali}, recalling also that $\theta=\alpha-1+\tc$ as in \eqref{theta.def.ridu}, we have
	\begin{equation}
		\begin{aligned}
				|\im\, \lambda\, \omega \cdot \ell &+ \mu_\infty(j;w_n) - \mu_\infty(j';w_n)|  \geq \lambda |\omega\cdot \ell| - |\mu_{\infty}(j;w_n)| - |\mu_{\infty}(j';w_n)| \\
				& \geq \frac{2\lambda \gamma_{n-1}}{\braket{\ell}^\tau} - \frac{C \lambda^{\alpha-1+\tc}}{{\rm min}\{ |j|,|j'| \}} \\
				& \geq  \frac{2\lambda \gamma_{n}}{\braket{\ell}^\tau} + \frac{2\lambda \big(\gamma_{n - 1} - \gamma_{n} \big)  }{\braket{\ell}^\tau}- \frac{C \lambda^{\alpha-1+\tc}}{ \tN_{n}^{\tau} } \geq  \frac{2\lambda \gamma_{n}}{\braket{\ell}^\tau}\,,
		\end{aligned}
	\end{equation}
	provided that, since $\gamma_{n - 1}-\gamma_{n} = \gamma \,2^{-n}$,
	\begin{equation}
		 \frac{2\lambda \big(\gamma_{n - 1} - \gamma_{n} \big)  }{\braket{\ell}^\tau}- \frac{C \lambda^{\alpha-1+\tc}}{ \tN_{n}^{\tau} } \geq \frac{2\lambda\,\gamma\, 2^{-n}}{\braket{\ell}^{\tau}} \Big(  1  - C \lambda^{\alpha-2+\tc} \gamma^{-1} 2^{n-1} \tN_{n}^{-\tau} \braket{\ell}^\tau \Big) \geq 0 \,.
	\end{equation}
	This latter condition is satisfied since $|\ell| \leq \tN_{n}$ and $\lambda^{\alpha - 2+\tc} \gamma^{- 1} = \lambda^{\alpha - 2(1-\tc)} \ll 1$, with $\lambda \gg 1$ large enough (note that $\alpha - 2(1-\tc)< 0$ by \eqref{nash moser smallness condition}) and $\tN_{0}\gg 1$ large enough. 
	The claimed bound \eqref{lower bound R ljj vuoto} is then proved. 
\end{proof}


When non-empty, the sequence of resonant sets $(\cR_{\ell j j'}(w_n))_{n\in\N_0}$ is nested.

\begin{lem}\label{inclusione risonanti v n n - 1}
Let $(\ell, j, j') \in {\mathcal I}$, with $|\ell| \leq \tN_{n}$ and ${\rm min}\{ |j|, |j'| \} \leq \tN_{n}^\tau$, $n \geq 1$. 
Then 
\begin{equation}\label{claimed.inclusion}
	{\mathcal R}_{\ell j j'}(w_n) \subseteq {\mathcal R}_{\ell j j'}(w_{n - 1})\,.
\end{equation}
\end{lem}

\begin{proof}
We split the proof in two steps. 

\noindent
{\sc Step 1.} We claim that, 
for any $j' \in \Z^2 \setminus \{ 0 \}$ and
for any $\omega \in {\mathcal G}_n$, we have
\begin{equation}\label{L infty vn - v n - 1}
\sup_{j \in \Z^2 \setminus \{ 0 \}}|\mu_\infty(j; w_n) - \mu_\infty(j'; w_{n - 1})| 
\lesssim \tN_{n-1}^{2 \overline \sigma + 1} \tN_{n-2}^{-\mathtt a_1}
+  \tN_{n - 1}^{- \ta}
\end{equation}
(the constant $\mathtt a$ is defined in \eqref{definizione.param.KAM}). 
By \eqref{def cal N infty nel lemma} and the inductive definition of the sets ${\mathcal G}_n$ in \eqref{G-n+1}, recalling \eqref{insiemi di cantor rid}, \eqref{cantor finale ridu} and Lemma \ref{prima inclusione cantor}, we have the inclusion 
\begin{equation}\label{inclusione cantor 100}
\begin{aligned}
{\mathcal G}_n  \subseteq \Omega_n^{\gamma_{n - 1}}(w_{n - 1}) \,.
\end{aligned}
\end{equation}
We aim to apply Proposition \ref{prop riducibilita}-${\bf (S4)_n}$ with 
$w_1 \equiv w_{n - 1}$, 
$w_2 \equiv w_n$, $\gamma\equiv \gamma_{n-1}$ and 
$\rho \equiv \gamma_{n - 1} - \gamma_n = \gamma 2^{-n}$. 
By \eqref{hn}, we have (recall \eqref{vare.vare})
\begin{equation}\label{verifica rid S3 n}
\begin{aligned}
 & K \tN_{n - 1}^{\tau + \tau^2} \lambda^{\alpha - 2+\tc} \| w_n - w_{n - 1} \|_{s_0 + \Sigma(\tb)} 
\\
& \leq K C_{*} \tN_{n - 1}^{\tau + \tau^2 + 2 \overline \sigma}  \tN_{n-2}^{- \mathtt a_1} \,  \lambda^{\alpha - 2+\tc} \lambda^{- 1} \gamma^{-1} 
\leq \gamma_{n - 1} - \gamma_n \,,
 \end{aligned}
\end{equation}
for some constant $K > 0$, where we have used that  $\gamma^{- 1} = \tN_{0} \leq \tN_{n - 1}$ for $n \geq 1$ and that
\begin{equation}
	K C_{*} \, 2^n \, \tN_{n - 1}^{\tau + \tau^2 + 2 \overline \sigma +  2}  \tN_{n-2}^{- \mathtt a_1} \,  \lambda^{\alpha - 2+\tc} \lambda^{- 1} 
	\leq 1 \,, \quad n \geq 1\,,
\end{equation}
 recalling that, by \eqref{definizione.param.KAM}, \eqref{costanti nash moser} and the smallness condition \eqref{nash moser smallness condition}, we have  $\overline \sigma > \Sigma(\tb)$, 
$\mathtt a_1 > \chi(\tau + \tau^2 + 2 \bar \sigma + 2)$ and $\bar \tau \geq \tau + \tau^2 + 2 \bar \sigma + 2$.
 The estimate \eqref{verifica rid S3 n} implies that Proposition \ref{prop riducibilita}-${\bf (S4)_n}$ applies, leading, together with \eqref{inclusione cantor 100}, to ${\mathcal G}_n  \subseteq \Omega_n^{\gamma_{n - 1}}(w_{n - 1})  \subseteq \Omega_n^{\gamma_{n}}(w_{n })$. Then, we apply the estimate \eqref{r nu - 1 r nu i1 i2} of Proposition \ref{prop riducibilita}-${\bf (S3)_n}$, with $w_1 \equiv w_{n - 1}$, $w_2 \equiv w_n$, $\gamma_1 \equiv \gamma_n$, $\gamma_2 \equiv \gamma_{n - 1}$, and obtain that, for any $\omega \in \cG_{n}$,
\begin{equation}\label{capitone 100}
\begin{aligned}
\sup_{j \in \Z^2 \setminus \{ 0 \}} |\mu_n(j; w_n) - \mu_n(j; w_{n - 1})| 
& \lesssim \lambda^{\alpha - 1+\tc} \| w_n - w_{n - 1} \|_{s_0 + \bar \sigma}  \\
&\lesssim
\tN_{n-1}^{2 \overline \sigma + 1} \tN_{n-2}^{-\mathtt a_1} \lambda^{\alpha - 2+\tc}\,,
\end{aligned}
\end{equation}
where we have applied in the last inequality the estimate \eqref{hn}.
Hence, by triangular inequality
and by \eqref{stime forma normale limite}, \eqref{capitone 100}, 
\eqref{stima.bassa.NM}, we have, for any $j \in \Z^2 \setminus \{ 0 \}$,  
\begin{equation}\label{capitone 1}
\begin{aligned}
|\mu_\infty(j; w_n) -&  \mu_\infty(j; w_{n - 1})|  \leq 
 |\mu_n(j;w_n) - \mu_n(j; w_{n - 1}) |  
\\
& +|\mu_\infty(j; w_n) - \mu_n(j; w_{n})|  
+ 
| \mu_n(j; w_{n - 1}) - \mu_\infty(j; w_{n - 1}) | \\
& 
\lesssim \tN_{n-1}^{2 \overline \sigma + 1} \tN_{n-2}^{-\mathtt a_1} \lambda^{\alpha - 2+\tc}
+ 2 \lambda^{-(M(\theta-1)+1)} \gamma^{-(M+1)}  \tN_{n - 1}^{- \mathtt a} \,,
\end{aligned}
\end{equation}
which, by \eqref{nash moser smallness condition}, implies the claimed estimate \eqref{L infty vn - v n - 1}, since $M>\tfrac{1-\tc}{2(1-\tc)-\alpha}$ by \eqref{definizione.param.KAM},  $\alpha - 2+\tc< 0$ and $\lambda \gg 1$ is large enough.   

\smallskip

\noindent
{\sc Step 2.} 
Let $\omega \in {\mathcal G}_n$ 
and $(\ell, j, j') \in {\mathcal I}$, with $|\ell| \leq \tN_{n}$ and ${\rm min}\{ |j|, |j'| \} \leq \tN_{n}^\tau$, $n \geq 1$.  Note that, since $\pi^\top(\ell) + j - j' = 0$, we have that $|j - j'| \lesssim |\ell| \lesssim \tN_{n}$ and hence 
\begin{equation}\label{max j j' Nn tau bla}
{\rm max}\{ |j|, |j'| \} \lesssim |j - j'| + {\rm min}\{ |j|, |j'| \} \lesssim \tN_{n}  + \tN_{n}^\tau \lesssim \tN_{n}^\tau\,.
\end{equation}
Therefore, by \eqref{L infty vn - v n - 1}, using that $\omega \in {\mathcal G}_n$,  we get, for some constant $C > 0$ large enough, 
\begin{equation} \label{bound bla so ri 0}
		\begin{aligned}
		|\lambda\, &\im\, \omega \cdot \ell  + \mu_\infty(j; w_n) - \mu_\infty(j'; w_n)|  \\
		& \geq |\lambda \,\im\, \omega \cdot \ell + \mu_\infty(j; w_{n - 1}) - \mu_\infty(j'; w_{n - 1})| \\
		& \quad - 
		|\mu_\infty(j; w_n) - \mu_\infty(j; w_{n - 1})|  - 
		|\mu_\infty(j'; w_n) - \mu_\infty(j'; w_{n - 1})|  \\
		& \geq \frac{2 \lambda \gamma_{n - 1}}{\langle \ell \rangle^\tau |j'|^\tau} - C \Big( \tN_{n-1}^{2 \overline \sigma + 1} \tN_{n-2}^{-\mathtt a_1}
		+  \tN_{n - 1}^{- \mathtt a} \Big)  \geq \frac{2 \lambda \gamma_{n }}{\langle \ell \rangle^\tau |j'|^\tau} \,, 
	\end{aligned}
\end{equation}
provided 
$$
C \Big( \tN_{n-1}^{2 \overline \sigma + 1} \tN_{n-2}^{-\mathtt a_1} 
+  \tN_{n - 1}^{- \mathtt a} \Big) \leq \frac{2 \lambda (\gamma_{n - 1} - \gamma_n)}{\langle \ell \rangle^\tau |j'|^\tau}\,.
$$
Using that $|\ell| \leq  \tN_{n}$, the estimate \eqref{max j j' Nn tau bla}, $\gamma_{n - 1} - \gamma_n = \gamma 2^{- n}$, $\gamma^{- 1} = \tN_{0} \leq \tN_{n}$, the latter condition is implied by 
\begin{equation}\label{giobbe1}
	C' 2^{n - 1} \tN_{n}^{\tau + \tau^2 + 1} \lambda^{- 1}\Big( \tN_{n-1}^{2 \overline \sigma + 1} \tN_{n-2}^{-\mathtt a_1} 
	+  \tN_{n - 1}^{- \mathtt a} \Big) \leq 1\,, \quad n \geq 1\,, 
\end{equation}
for some constant $C' \geq C$. The condition \eqref{giobbe1} is verified by \eqref{definizione.param.KAM}, \eqref{costanti nash moser} and the smallness condition \eqref{nash moser smallness condition}.
We conclude then that the claimed inclusion  \eqref{claimed.inclusion} follows by the estimate \eqref{bound bla so ri 0}. 
%
\end{proof}

\begin{lem}\label{inclusione per sommare serie}
For any $n \geq 1$, the following inclusion holds:
\begin{equation}\label{quinta inclusione stime misura}
{\mathcal G}_n \setminus \Omega_\infty^{\gamma_n}(w_n) \subseteq \bigcup_{\begin{subarray}{c}(\ell, j, j') \in {\mathcal I} \\
|\ell| \geq \tN_{n}
\end{subarray}}{\mathcal R}_{\ell j j'}(w_n)\,.
\end{equation}
\end{lem}
\begin{proof}
Let $n \geq 1$. By the definition \eqref{def cal R vn l j j'}, ${\mathcal R}_{\ell j j'}(w_n) \subseteq {\mathcal G}_n $. By Lemmata \ref{lemma triviale modi alti}, \ref{inclusione risonanti v n n - 1}, for any $(\ell, j, j') \in {\mathcal I}$, $|\ell| \leq \tN_{n}$, if ${\rm min}\{ |j|, |j'| \} \geq \tN_{n}^\tau$, then ${\mathcal R}_{\ell, j j'}(w_n) = \emptyset$, and, if ${\rm min}\{ |j|, |j'| \} \leq \tN_{n}^\tau$, then ${\mathcal R}_{\ell j j'}(w_n) \subseteq {\mathcal R}_{\ell j j'}(w_{n - 1})$. On the other hand, by \eqref{G-n+1}, we have ${\mathcal G}_n \cap {\mathcal R}_{\ell j j'}(w_{n - 1}) = \emptyset$. Therefore \eqref{quinta inclusione stime misura} holds.
\end{proof} 
{\sc Proof of Proposition \ref{Gn Omegan Lambdan}.}
We prove item $(i)$. For $(\ell, j, j') \in {\mathcal I}$, one has that $\pi^\top(\ell) + j - j' = 0$ and hence 
\begin{equation}\label{so ri 2 bla car}
|j| \leq |\pi^\top(\ell)| + |j'| \lesssim |\ell| + |j'| \lesssim \langle \ell \rangle |j'|\,. 
\end{equation}
Let $n \geq 1$. By Lemmata  \ref{inclusione per sommare serie}, \ref{stima misura risonanti sec melnikov},  the inclusion \eqref{quinta inclusione stime misura}, \eqref{so ri 2 bla car} and \eqref{choice tau}, we have 
\begin{equation}\label{Gn meno Omega (vn)}
\begin{aligned}
\Big| {\mathcal G}_n \setminus \Omega_\infty^{\gamma_n}(w_n)\Big|& \lesssim \sum_{(\ell, j, j') \in {\mathcal I} \atop
|\ell| \geq \tN_{n}
} |{\mathcal R}_{\ell j j'}(w_n)|   \lesssim \gamma  \sum_{(\ell, j, j') \in {\mathcal I} \atop
	|\ell| \geq \tN_{n}
}\frac{1}{\langle \ell \rangle^{\tau + 1} |j'|^\tau} \\
& \lesssim \gamma \sum_{
|\ell| \geq \tN_{n}} \frac{1}{\langle \ell \rangle^{\tau - 1}} \sum_{j' \in \Z^2 \setminus \{ 0 \}} \frac{1}{|j'|^{\tau - 2}}  \lesssim \gamma\, \tN_{n}^{- (\tau - 1)}\,.
\end{aligned}
\end{equation}
Similarly, by the inclusion \eqref{quarta inclusione stime misura}, Lemma \ref{stima misura risonanti sec melnikov} and using the standard volume estimate $|\Omega \setminus \tD\tC(\gamma, \tau)| \lesssim \gamma$, one also prove that $|{\mathcal G}_0 \setminus \Omega_\infty^{\gamma_0}(w_0)| \lesssim \gamma$. The item $(ii)$ can be proved by similar arguments. Finally the estimate on ${\mathcal G}_n \setminus {\mathcal G}_{n + 1}$ follows by recalling 
formula \eqref{seconda inclusione stime misura} and by applying items $(i),(ii)$.

\medskip

\begin{proof}
[\textsc{Proof of Proposition \ref{prop measure estimate finale}}]
It follows by applying 
Proposition \ref{Gn Omegan Lambdan}, Lemma \ref{lemma.equa.linear.approx}, using the inclusion \eqref{prima inclusione stime misura}, the standard estimate $|\Omega\setminus \tD\tC(\gamma,\tau) |\lesssim \gamma$ and the fact that the series $\sum_{n \geq 1} \tN_{n}^{-(\tau - 1)}$ 
converges. 
\end{proof}

\begin{footnotesize}
	
\end{footnotesize}

\bigskip

\begin{flushleft}
	\textbf{Roberta Bianchini}
	
	Consiglio Nazionale Delle Ricerche
	
	00185 Roma,	Italy
	\smallskip 
	
	\texttt{roberta.bianchini@cnr.it}
	
\end{flushleft}

\begin{flushleft}
	\textbf{Luca Franzoi}
	
	\smallskip
	
	Dipartimento di Matematica ``Federigo Enriques''
	
	Universit\`a degli Studi di Milano
	
	Via Cesare Saldini 50
	
	20133 Milano, Italy
	
	\smallskip 
	
	\texttt{luca.franzoi@unimi.it}
	
\end{flushleft}
\begin{flushleft}
	\textbf{Riccardo Montalto}
	
	\smallskip
	
	Dipartimento di Matematica ``Federigo Enriques''
	
	Universit\`a degli Studi di Milano
	
	Via Cesare Saldini 50
	
	20133 Milano, Italy
	
	\smallskip 
	
	\texttt{riccardo.montalto@unimi.it}
	
\end{flushleft}

\begin{flushleft}
	\textbf{Shulamit Terracina}
	
	\smallskip
	
	Dipartimento di Matematica ``Federigo Enriques''
	
	Universit\`a degli Studi di Milano
	
	Via Cesare Saldini 50
	
	20133 Milano, Italy
	
	\smallskip 
	
	\texttt{shulamit.terracina@unimi.it}
	
\end{flushleft}

\end{document}